\documentclass[final,leqno]{siamltex}

\usepackage[leqno]{amsmath}
\usepackage{graphicx}
\usepackage{graphics}
\usepackage{epstopdf}
\usepackage{subfigure}
\usepackage{threeparttable}

\usepackage{bm}
\usepackage{amssymb}
\usepackage{leftidx}
\usepackage{empheq}
\usepackage{color}
\usepackage{makecell}
\usepackage{enumerate}
\allowdisplaybreaks
\renewcommand{\theequation}{\arabic{section}.\arabic{equation}}
\numberwithin{equation}{section}
\newtheorem{remark}{Remark}[section]
\newtheorem{example}{Example}[section]
\title{High-efficiency and positivity-preserving stabilized SAV methods for gradient flows.
        \thanks{
We would like to acknowledge the assistance of volunteers in putting together this example manuscript and supplement. This work is supported by National Natural Science Foundation of China (Grant Nos: 12001336, 12271302, 12131014) and Natural Science Outstanding Youth Fund of Shandong Province (Grant No: ZR2023YQ007).}}

      \author{Zhengguang Liu
            \thanks{School of Mathematics and Statistics, Shandong Normal University, Jinan, Shandong, 250358, China. Email: liuzhg@sdnu.edu.cn.}
             \and
         	Yanrong Zhang
         	 \thanks{Department of Applied Mathematics, The Hong Kong Polytechnic University, Hung Hom, Hong Kong. Email: yanrongzhang\_math@163.com.}
             \and
            Xiaoli Li
             \thanks{Corresponding author. School of Mathematics, Shandong University, Jinan, Shandong, 250100, China. Email: xiaolimath@sdu.edu.cn.}
             }
\begin{document}
\maketitle

\begin{abstract}
The scalar auxiliary variable (SAV)-type methods are very popular techniques for solving various nonlinear dissipative systems. Compared to the semi-implicit method, the baseline SAV method can keep a modified energy dissipation law but doubles the computational cost. The general SAV approach    
does not add additional computation but needs to solve a semi-implicit solution in advance, which may potentially compromise the accuracy and stability. In this paper, we construct a novel first- and second-order unconditional energy stable and positivity-preserving stabilized SAV (PS-SAV) schemes for $L^2$ and $H^{-1}$ gradient flows. The constructed schemes can reduce nearly half computational cost of the baseline SAV method and preserve its accuracy and stability simultaneously. Meanwhile, the introduced auxiliary variable is always positive while the baseline SAV cannot guarantee this positivity-preserving property. Unconditionally energy dissipation laws are derived for the proposed numerical schemes. We also establish a rigorous error analysis of the first-order scheme for the Allen-Cahn type equation in $l^{\infty}(0,T; H^1(\Omega) ) $ norm. 
In addition we propose an energy optimization technique to optimize the modified energy close to the original energy. Several interesting numerical examples are presented to demonstrate the accuracy and effectiveness of the proposed methods.  
\end{abstract}



\pagestyle{myheadings}
\thispagestyle{plain}
\markboth{Zhengguang Liu, Yanrong Zhang, Xiaoli Li} {Positivity-preserving Stabilized SAV Methods}
\section{Introduction}
The gradient flows are very important models in physics, engineering, materials science and mathematics that can accurately and effectively describe the complex interfacial behavior of multi-phase materials. Many modern scientific problems, such as multi-phase industrial alloy casting, metal additive manufacturing, shale oil and gas development, image processing, biomedicine, chip packaging, and many other practical applications can be described by corresponding gradient flow models \cite{chen2002phase,keller1970initiation,osher1988fronts,rudin1992nonlinear}. In recent years, they have also gained rapid development in many high-precision fields, such as integrated circuits, lithium-ion batteries, 3D printing, etc \cite{fallah2012phase,wang2020application,zuo2015phase}.

In this paper, we consider a gradient flow  with respect to the following free energy $E(\phi)$:
\begin{equation*}
\aligned
E(\phi)=\frac{\epsilon^2}{2}(A\phi,\phi)+\int_\Omega F(\phi(\textbf{x}))d\textbf{x},
\endaligned
\end{equation*}
where $\epsilon>0$ denotes the interfacial width, $A$ is a linear self-adjoint elliptic operator and $F(\phi)$ is a nonlinear potential functional.     
By introducing a chemical potential $\mu=\frac{\delta E}{\delta\phi}$, we can write the gradient flow as follows:
\begin{equation}\label{intro-e1}
   \begin{array}{l}
\displaystyle\frac{\partial \phi}{\partial t}=-M\mathcal{G}\mu,\\
\displaystyle\mu=\epsilon^2A\phi+F'(\phi).
   \end{array}
\end{equation}
with periodic or homogeneous Neumann boundary condition, and $\mathcal{G}$ is a positive definite operator. For instance, if we let the operator $\mathcal{G}=I$, $A=-\Delta$ and $F(\phi)=\frac14(\phi^2-1)^2$, the above gradient flow \eqref{intro-e1} will be the known Allen-Cahn model:
\begin{equation}\label{intro-e2}
   \begin{array}{l}
\displaystyle\frac{\partial \phi}{\partial t}=M\epsilon^2\Delta\phi-MF'(\phi).
   \end{array}
\end{equation}

The gradient flow is generally a high-order nonlinear partial differential equation, which is a complex system with energy dissipation law. However, it's very difficult to design efficient and energy stable numerical algorithms. In general, the fully explicit discrete scheme for the nonlinear term of the gradient flow \eqref{intro-e1} cannot preserve its physical constraints of the original system. Fully implicit schemes can guarantee the structure of the model, but such methods may require strict time-step restrictions to guarantee the unique solvability and need to solve nonlinear equations at each step, so they are not efficient in practice. The more widely used and effective methods mainly include convex splitting methods \cite{baskaran2013convergence,eyre1998unconditionally}, stabilization methods \cite{chen1998applications,shen2010numerical,xu2006stability}, exponential time-differencing (ETD) methods \cite{du2019maximum,du2021maximum,ju2018energy}, invariant energy quadratization (IEQ) methods \cite{yang2018linear,yang2017numerical,zhao2017numerical}, scalar auxiliary variable (SAV) methods \cite{huang2020highly,ShenA,shen2018scalar}, Lagrange multiplier methods \cite{cheng2020new,cheng2022new}, etc.

In recent years, the SAV-type methods have attracted much attention in numerical solutions for various nonlinear dissipative systems due to their inherent advantage of preserving energy dissipation law. 
In these SAV-type methods, the baseline SAV method \cite{shen2018scalar} can keep a modified energy dissipation law but doubles the computational cost compared with a semi-implicit approach. It has attracted a lot of attention and has been successfully applied to solve various kinds of complex nonlinear problems, such as various phase field models \cite{cheng2018multiple,cheng2019highly,hou2021robust,jiang2022improving,li2020stability,li2019energy,liu2020exponential}, the Navier-Stokes equation \cite{xiaoli2020error,lin2019numerical}, the Schrödinger equation \cite{antoine2021scalar}, the magnetohydrodynamic (MHD) equation \cite{li2022stability}, etc. The recently general SAV approach \cite{huang2020highly} does not add additional computation but needs to solve a semi-implicit solution in advance which may weaken the accuracy and stability. The main purpose of this paper is to construct a positivity-preserving stabilized SAV (PS-SAV) approach which enjoys the following advantages:

$\bullet$ The introduced scalar auxiliary variable always keeps a positive property, whereas the baseline SAV scheme fails to do so;

$\bullet$ It only requires solving one linear system with constant coefficients as opposed to the two linear systems by the baseline SAV approach, thus the computational cost of the proposed approach is essentially half that of the SAV approach;

$\bullet$ It provides an enhanced stability and accuracy compared to the GSAV approach, while maintaining nearly identical computational costs.

We prove the unconditional energy dissipation law for the proposed numerical schemes. Furthermore, a rigorous error analysis is derived for the fully-discrete finite difference method with first-order accuracy in time. 
In particular, it is important to note that the major difficulty in the error estimate is caused by the implicit treatment for $R_h$ and explicit discretization for $\Delta_h \phi_h$ in time. The essential tools used in the proof are unconditional energy dissipation law,  the induction process to give a first estimates for the phase function and show that the discrete $l^{\infty}$ norm of the numerical solution is uniformly bounded. Thus by establishing several auxiliary lemmas, we finally obtain the optimal convergence rates for the phase function in $l^{\infty}(0,T; H^1(\Omega) ) $ norm. We believe that our constructed schemes and optimal error estimate are the first linear, positivity-preserving and unconditionally energy stable method with implicit treatment for the scalar auxiliary variable. 
 
The rest of this paper is organized as follows. In Section 2, we provide a brief review of the SAV-type approaches such as the baseline SAV and GSAV methods for gradient flows. In Section 3, we present the first-order semi-discrete and fully discrete positivity-preserving stabilized SAV schemes for $L^2$ gradient flows together with the energy dissipation law and convergence analysis of the resulting fully discrete scheme. In Section 4, we extend the considered PS-SAV technique to construct second-order Crank-Nicloson scheme. A semi-discrete numerical scheme based on the PS-SAV approach for $H^{-1}$ gradient flow models is given in Section 5. In Section 6, an energy optimization technique is proposed to optimize the modified energy close to the original energy. In Section 7, we give some comparisons of the proposed PS-SAV approach with the baseline SAV and GSAV approaches to validate its high efficiency.

\section{A brief review of the SAV-type approaches}
In this section, we give a brief review of the SAV-type methods for the gradient flow \eqref{intro-e1} to better introduce our newly proposed methods. 
\subsection{The baseline SAV approach}
Assume the nonlinear free energy $E_1(\phi)=\int_\Omega F(\phi(\textbf{x}))d\textbf{x}$ is bound from below, that is $E_1(\phi)+C>0$ for some constant $C>0$. Let us introduce an auxiliary variable $q(t)=\sqrt{E_1(\phi)+C}$ and reformulate the gradient flow \eqref{intro-e1} to the following equivalent system:
\begin{equation}\label{SAV-e1}
   \begin{array}{l}
\displaystyle\frac{\partial \phi}{\partial t}=-M\mathcal{G}\mu,\\
\displaystyle\mu=\epsilon^2A\phi+\frac{q(t)}{\sqrt{E(\phi)+C}}F'(\phi),\\
\displaystyle\frac{dq}{dt}=\frac{1}{2\sqrt{E(\phi)+C}}(F'(\phi),\frac{\partial \phi}{\partial t}).
   \end{array}
\end{equation}  

Before giving a semi-discrete formulation, we let $N>0$ be a positive integer and set
\begin{equation*}
\Delta t=T/N,\quad t^n=n\Delta t,\quad \text{for}\quad n\leq N.
\end{equation*}

Then we give the following first-order SAV scheme:
\begin{equation}\label{SAV-e2}
   \begin{array}{l}
\displaystyle\frac{\phi^{n+1}-\phi^n}{\Delta t}=-M\mathcal{G}\mu^{n+1},\\
\displaystyle\mu^{n+1}=\epsilon^2A\phi^{n+1}+\frac{q^{n+1}}{\sqrt{E(\phi^n)+C}}F'(\phi^n),\\
\displaystyle\frac{q^{n+1}-q^n}{\Delta t}=\frac{1}{2\sqrt{E(\phi^n)+C}}(F'(\phi^n),\frac{\phi^{n+1}-\phi^n}{\Delta t}).
   \end{array}
\end{equation} 

The scheme \eqref{SAV-e2} is unconditionally energy stable in the sense that:
\begin{equation*}
\left(\frac{\epsilon^2}{2}(A\phi^{n+1},\phi^{n+1})+|q^{n+1}|^2\right)-\left(\frac{\epsilon^2}{2}(A\phi^{n},\phi^{n})+|q^{n}|^2\right)\leq-M\Delta t(\mathcal{G}\mu^{n+1},\mu^{n+1})\leq0.
\end{equation*}

The above first-order SAV scheme requires the solution of two linear systems with constant coefficients at each time step. The unknown $q^{n+1}$ and $\phi^{n+1}$ can be calculated decoupled. By setting $\phi^{n+1}=\phi_1^{n+1}+q^{n+1}\phi_2^{n+1}$ 
, we find that $\phi_1^{n+1}$ and $\phi_2^{n+1}$ are solutions of the following two linear equations with constant coefficients:
\begin{equation*}
\left(I+M\Delta t\epsilon^2\mathcal{G}A\right)\phi_1^{n+1}=\phi^n,\quad \left(I+M\Delta t\epsilon^2\mathcal{G}A\right)\phi_2^{n+1}=-\frac{M\Delta t}{\sqrt{E(\phi^n)+C}}\mathcal{G}F'(\phi^n).
\end{equation*}
Once $\phi_1^{n+1}$ and $\phi_2^{n+1}$ are known, we can determine $q^{n+1}$ explicitly by the following equation:
\begin{equation*}
\left[1-\frac{1}{2\sqrt{E(\phi^n)+C}}(F'(\phi^n),\phi_2^{n+1})\right]q^{n+1}=q^n+\frac{1}{2\sqrt{E(\phi^n)+C}}(F'(\phi^n),\phi_1^{n+1}).
\end{equation*}
\begin{remark}\label{sav-lemma1}
The unknown variables $q^{n+1}$ and $\phi^{n+1}$ in the SAV scheme \eqref{SAV-e2} can be calculated decoupled.  It requires solving two linear equations with constant coefficients at each time step, so its computational cost is essentially double of the semi-implicit approach. 
\end{remark}
\subsection{The general SAV approach}
To reduce the computational cost, Shen et al. \cite{huang2020highly} considered a general SAV approach that is based on a semi-implicit correction. Firstly, assume that the free energy $E(\phi)$ is bounded from below which means $E(\phi)+C>0$ for a positive constant $C$. Introduce a scalar variable $R(t)=E(\phi)+C$ and rewrite the gradient flow \eqref{intro-e1} as the following equivalent system:
\begin{equation}\label{gSAV-e1}
   \begin{array}{l}
\displaystyle\frac{\partial \phi}{\partial t}=-M\mathcal{G}\mu,\\
\displaystyle\mu=\epsilon^2A\phi+F'(\phi),\\
\displaystyle\xi=\frac{R(t)}{E(\phi)+C},\\
\displaystyle\frac{dR}{dt}=-M\xi(\mu,\mu).
   \end{array}
\end{equation}
It is not difficult to obtain the following modified energy dissipation law for above equivalent system:
\begin{equation*}
\aligned
\frac{dR(t)}{dt}=\frac{d}{dt}(E(\phi)+C)=-M\xi(\mathcal{G}\mu,\mu)\leq0.
\endaligned
\end{equation*}

We discretisize the state variable $\phi$ and the introducing variable $R$ implicitly and discretisize the
energy density function $F'(\phi)$ explicitly to obtain the following $k$th-order implicit-explicit (IMEX) schemes:
\begin{equation}\label{gSAV-e2}
   \begin{array}{l}
\displaystyle\frac{\alpha_k\overline{\phi}^{n+1}-\beta_k(\phi^n)}{\Delta t}=-M\mathcal{G}\overline{\mu}^{n+1},\\
\displaystyle\overline{\mu}^{n+1}=\epsilon^2A\overline{\phi}^{n+1}+F'(\widehat{\phi}^{n+1}),\\
\displaystyle\xi^{n+1}=\frac{R^{n+1}}{E(\widehat{\phi}^{n+1})+C},\\
\displaystyle\frac{R^{n+1}-R^n}{\Delta t}=-M\xi^{n+1}(\mathcal{G}\overline{\mu}^{n+1},\overline{\mu}^{n+1}),\\
\phi^{n+1}=\left[1-(1-\xi^{n+1})^{k+1}\right]\overline{\phi}^{n+1}.
   \end{array}
\end{equation}
Here $\alpha_k$, $\beta_k$ and $\widehat{\phi}^{n+1}$ are different for $k$th-order schemes. For example, they can be defined as follows:

First-order:
\begin{equation*}
\aligned
\alpha_k=1,\quad \beta_k(\phi^n)=\phi^n,\quad \widehat{\phi}^{n+1}=\phi^n,
\endaligned
\end{equation*}

Second-order:
\begin{equation*}
\aligned
\alpha_k=\frac32,\quad \beta_k(\phi^n)=2\phi^n-\frac12\phi^{n-1},\quad \widehat{\phi}^{n+1}=2\phi^n-\phi^{n-1}.
\endaligned
\end{equation*}
For more details, please see \cite{zhang2022generalized}.

The above numerical schemes \eqref{gSAV-e2} is unconditional energy stable with a modified energy $\mathcal{E}=R^{n+1}-C$ to keep $R^{n+1}\leq R^n$.
\begin{remark}\label{gsav-lemma1}
The $k$th-order GSAV scheme \eqref{gSAV-e2} requires solving only one linear equation with constant coefficients at each time step. However, it requires a semi-implicit solution in advance at each time step, which may weaken its stability and accuracy. In practical calculations, it may be necessary to use smaller time steps to achieve long time simulations.
\end{remark}
\section{A positivity-preserving stabilized SAV (PS-SAV) method}
In this section, we consider a positivity-preserving stabilized SAV (PS-SAV) method for solving the gradient flow \eqref{intro-e1} effectively. This new proposed method holds the positivity-preserving property of the introduced auxiliary variable. Meanwhile, it reduces the computational cost of the baseline SAV method and preserve its accuracy and stability. We first consider the semi-discrete and fully discrete schemes based on PS-SAV method for the $L^2$ gradient flow.

\subsection{The $L^2$ gradient flow}
Firstly, we set $\mathcal{G}=I$ to transform the gradient flow \eqref{intro-e1} into the following $L^2$ gradient flow:
\begin{equation}\label{L2-gradient-flow}
   \begin{array}{l}
\displaystyle\frac{\partial \phi}{\partial t}=-M\mu,\\
\displaystyle\mu=\epsilon^2A\phi+F'(\phi).
   \end{array}
\end{equation}

Similar as the general SAV approach, we also assume $E(\phi)+C>0$ for a positive constant $C$ and introduce a same scalar variable $R(t)=E(\phi)+C$. Then, we change the third equation in the equivalent system \eqref{SAV-e1} by the following formulation:
\begin{equation}\label{PSSAV-e1}
\displaystyle\frac{dR}{dt}=\frac{dE}{dt}=(\frac{\delta E}{\delta\phi},\frac{\partial \phi}{\partial t})=(\mu,\frac{\partial \phi}{\partial t})=-\frac{1}{M}(\frac{\partial \phi}{\partial t},\frac{\partial \phi}{\partial t}).
\end{equation}

Combining above equation \eqref{PSSAV-e1} with the $L^2$ gradient flow \eqref{L2-gradient-flow}, we can reformulate it to the following equivalent system:
\begin{equation}\label{PSSAV-e2}
   \begin{array}{l}
\displaystyle\frac{\partial \phi}{\partial t}=-M\mu,\\
\displaystyle\mu=\frac{R(t)}{E(\phi)+C}\left(\epsilon^2A\phi+F'(\phi)\right),\\
\displaystyle\frac{dR}{dt}=-\frac{1}{M}(\frac{\partial \phi}{\partial t},\frac{\partial \phi}{\partial t}).
   \end{array}
\end{equation}
Obviously the third equation in \eqref{PSSAV-e2} can keep the energy dissipation law.

Based on such an equivalent form \eqref{PSSAV-e2}, we next give the first-order semi-discrete PS-SAV scheme.
\subsection{First-order semi-discrete PS-SAV scheme} 
A first-order positivity-preserving stabilized SAV scheme based on backward Euler formulation is given by:
\begin{equation}\label{PSSAV-e3}
   \begin{array}{l}
\displaystyle\frac{\phi^{n+1}-\phi^n}{\Delta t}=-M\mu^{n+1},\\
\displaystyle\mu^{n+1}=s\epsilon^2(A\phi^{n+1}-A\phi^{n})+\frac{R^{n+1}}{E(\phi^n)+C}\left[\epsilon^2A\phi^{n}+F'(\phi^n)\right],\\
\displaystyle \frac{R^{n+1}-R^n}{\Delta t}=\displaystyle-\frac{1}{M}(\frac{\phi^{n+1}-\phi^n}{\Delta t},\frac{\phi^{n+1}-\phi^n}{\Delta t}),
   \end{array}
\end{equation}
where $s>0$ is a stabilizing constant.

From the first two equations in \eqref{PSSAV-e3}, we can rewrite \eqref{PSSAV-e3} equivalently as the following formulation:
\begin{equation}\label{PSSAV-e4}
   \begin{array}{l}
\displaystyle(E(\phi^n)+C)(I+Ms\epsilon^2\Delta tA)\frac{\phi^{n+1}-\phi^n}{\Delta t}=-MR^{n+1}\left[\epsilon^2A\phi^{n}+F'(\phi^n)\right],\\
\displaystyle \frac{M}{\Delta t}(R^{n+1}-R^n)=\displaystyle-(\frac{\phi^{n+1}-\phi^n}{\Delta t},\frac{\phi^{n+1}-\phi^n}{\Delta t}).
   \end{array}
\end{equation}
Setting $\phi^{n+1}=\phi^{n}+\Delta tR^{n+1}\phi_1^{n+1}$, we find that $\phi_1^{n+1}$ is the solution of the following linear equation with constant coefficients:
\begin{equation}\label{PSSAV-e5}
\displaystyle(E(\phi^n)+C)(I+Ms\epsilon^2\Delta tA)\phi_1^{n+1}=-M\left[\epsilon^2A\phi^{n}+F'(\phi^n)\right].
\end{equation}
Once $\phi_1^{n+1}$ is known, noting that  
\begin{equation}\label{PSSAV-e6}
\displaystyle\phi^{n+1}-\phi^n=\Delta tR^{n+1}\phi_1^{n+1},
\end{equation}
and combining it with the second equation in \eqref{PSSAV-e4}, we obtain
\begin{equation}\label{PSSAV-e7}
\displaystyle(\phi_1^{n+1},\phi_1^{n+1})(R^{n+1})^2+\frac{M}{\Delta t}R^{n+1}-\frac{M}{\Delta t}R^n=0.
\end{equation}
If $\phi_1^{n+1}=0$, we obtain $(\phi_1^{n+1},\phi_1^{n+1})=0$. Then we directly get $\phi^{n+1}=\phi^n$ and $R^{n+1}=R^n$. If $\phi_1^{n+1}\neq0$, we obtain $(\phi_1^{n+1},\phi_1^{n+1})\neq0$. The above equation \eqref{PSSAV-e6} is a quadratic equation with one variable for $R^{n+1}$. 
\begin{theorem}\label{PSSAV-theorem1}
The quadratic equation with one variable for $R^{n+1}$ \eqref{PSSAV-e6} has and only one positive solution:
\begin{equation}\label{PSSAV-e8}
\displaystyle R^{n+1}=\frac{-\frac{M}{\Delta t}+\sqrt{\frac{M^2}{\Delta t^2}+4\frac{M}{\Delta t}R^n(\phi_1^{n+1},\phi_1^{n+1})}}{2(\phi_1^{n+1},\phi_1^{n+1})}>0.
\end{equation}
\end{theorem}
\begin{proof}
Noting that $R^{0}=E(\phi_0)+C>0$, then we assume that $R^n>0$. The quadratic equation \eqref{PSSAV-e6} is determined to have a solution because of 
\begin{equation*}
\displaystyle\Delta=\frac{M^2}{\Delta t^2}+4\frac{M}{\Delta t}R^n(\phi_1^{n+1},\phi_1^{n+1})>\frac{M^2}{\Delta t^2}>0.
\end{equation*}
One can obviously see that \eqref{PSSAV-e6} has the following two solutions:
\begin{equation*}
\begin{array}{l}
\displaystyle R_1^{n+1}=\frac{-\frac{M}{\Delta t}-\sqrt{\frac{M^2}{\Delta t^2}+4\frac{M}{\Delta t}R^n(\phi_1^{n+1},\phi_1^{n+1})}}{2(\phi_1^{n+1},\phi_1^{n+1})}<0,\\
\displaystyle R_2^{n+1}=\frac{-\frac{M}{\Delta t}+\sqrt{\frac{M^2}{\Delta t^2}+4\frac{M}{\Delta t}R^n(\phi_1^{n+1},\phi_1^{n+1})}}{2(\phi_1^{n+1},\phi_1^{n+1})}>0.
\end{array}
\end{equation*}
By the positive property of $R$, we have that $R^{n+1}=R_2^{n+1}$.
\end{proof}

Then we can obtain $\phi^{n+1}$ directly by the following equation:
\begin{equation}\label{PSSAV-e9}
\displaystyle\phi^{n+1}=\phi^n+\Delta tR^{n+1}\phi_1^{n+1}.
\end{equation}

To summarize, the first-order PS-SAV scheme \eqref{PSSAV-e3} can be implemented as follows:
\begin{itemize}
  \item solve $\phi_1^{n+1}$ from \eqref{PSSAV-e5};
  \item compute $R^{n+1}$ from \eqref{PSSAV-e8};
  \item update $\phi^{n+1}=\phi^n+\Delta tR^{n+1}\phi_1^{n+1}$ and goto next time step.
\end{itemize}

We observe that the above procedure only requires solving one linear equation with constant coefficients as in a semi-implicit scheme with stabilization. As for the energy stability, we have the following result.
\begin{theorem}\label{PSSAV-theorem2}
Given $R^0>0$, we have $R^n>0$, and the first-order PS-SAV scheme \eqref{PSSAV-e3} is unconditionally energy stable in the sense that
\begin{equation*}
R^{n+1}-R^{n}=\displaystyle-\frac{\Delta t}{M}(\frac{\phi^{n+1}-\phi^n}{\Delta t},\frac{\phi^{n+1}-\phi^n}{\Delta t})\leq0.
\end{equation*}
\end{theorem}
\subsection{Spacial discretization}
In this subsection, we consider a fully discrete scheme based on the proposed PS-SAV approach by applying finite difference method for the spacial discretization for the Allen-Cahn type model. For simplicity, we consider the two-dimensional square domain $\Omega=(0,L)\times(0,L)$ with periodic boundary conditions. 

We set $h=L/N_{xy}$ to be size of the uniform mesh where $N_{xy}$ is a positive integer. The grid points are denoted by $(x_i.y_j)=(ih,jh)$ for $1\leq i,j\leq N_{xy}$. The discrete Laplace operator $\Delta_h$ is defined by 
\begin{equation*}
\Delta_hu_{i,j}=\frac{1}{h^2}(u_{i+1,j}+u_{i,j+1}+u_{i-1,j}+u_{i,j-1}-4u_{i,j}),
\end{equation*}
and the discrete gradient operator $\nabla_h$ is defined by 
\begin{equation*}
\aligned
\nabla_hu_{i,j} = & (\frac{u_{i+1,j}-u_{i,j}}{h},\frac{u_{1,j+1}-u_{i,j}}{h}) \\
: = & ( \nabla_h^1 u_{i+1/2,j}, \nabla_h^2 u_{i,j+1/2} ).
\endaligned
\end{equation*}

Define the discrete inner products and norms are
\begin{equation*}
\aligned
& (u,v)_m = h^2 \sum\limits_{i,j=1}^{ N_{xy} } u_{i,j} v_{i,j} , \ \|u \|_m^2 = (u,u)_m, \\ 
& (u,v)_x = h^2 \sum\limits_{i=0}^{ N_{xy}-1 }  \sum\limits_{j=1}^{ N_{xy} } u_{i+1/2,j} v_{i+1/2,j}, \\
& (u,v)_y=h^2 \sum\limits_{i=1}^{ N_{xy} }  \sum\limits_{j=0}^{ N_{xy} -1} u_{i,j+1/2} v_{i,j+1/2},  \\
& \|\nabla_h u\|_{TM}^2 = ( \nabla_h^1 u, \nabla_h^1 u )_x + ( \nabla_h^2 u, \nabla_h^2 u)_y .
\endaligned
\end{equation*} 

The following discrete-integration-by-part formula plays an important role in the analysis:
\begin{equation}\label{gradient-inner}
\aligned
(u,\Delta_hv)_m=- \left[( \nabla_h^1 u, \nabla_h^1 u )_x + ( \nabla_h^2 u, \nabla_h^2 u)_y\right]= (\Delta_hu,v)_m.
\endaligned
\end{equation}
A first-order fully discrete PS-SAV scheme for the Allen-Cahn type model is given by:
\begin{equation}\label{fully-scheme1}
   \begin{array}{l}
\displaystyle\frac{\phi_h^{n+1}-\phi_h^n}{\Delta t}=-M\mu_h^{n+1},\\
\displaystyle\mu_h^{n+1}=-s\epsilon^2(\Delta_h\phi_h^{n+1}-\Delta_h\phi_h^{n})+\frac{R^{n+1}}{E(\phi^n)+C}\left[-\epsilon^2\Delta_h\phi_h^{n}+F'(\phi_h^n)\right],\\
\displaystyle \frac{R_h^{n+1}-R_h^n}{\Delta t}=\displaystyle-\frac{1}{M}(\frac{\phi_h^{n+1}-\phi_h^n}{\Delta t},\frac{\phi_h^{n+1}-\phi_h^n}{\Delta t}),
   \end{array}
\end{equation}
where $s>0$ is a stabilizing constant.

Similar as semi-discrete scheme \eqref{PSSAV-e3}, we are easy to obtain the following energy dissipation law.
\begin{theorem}\label{fully-scheme-theorem1}
Given $R_h^0>0$, we have $R_h^n>0$, and the first-order fully discrete PS-SAV scheme \eqref{fully-scheme1} is unconditionally energy stable in the sense that
\begin{equation*}
R_h^{n+1}-R_h^{n}=\displaystyle-\frac{\Delta t}{M}(\frac{\phi_h^{n+1}-\phi_h^n}{\Delta t},\frac{\phi_h^{n+1}-\phi_h^n}{\Delta t})\leq0.
\end{equation*}
\end{theorem}
\subsection{Error estimates}
In this subsection, we will derive error estimates for the proposed first-order fully discrete PS-SAV scheme \eqref{fully-scheme1} applied to Allen-Cahn type equation. 

For simplicity, we set 
\begin{equation*}
e_\phi^{n+1}=\phi_h^{n+1}-\phi(t^{n+1}),\quad e_\mu^{n+1}=\mu_h^{n+1}-\mu(t^{n+1}),\quad e_R^{n+1}=R_h^{n+1}-R(t^{n+1}).
\end{equation*}

\begin{theorem}\label{err-estimate}
Assume $ \phi \in W^{2,\infty}(0,T; L^2(\Omega)) \bigcap  W^{1,\infty}(0,T; W^{2,2}(\Omega))  \bigcap  L^{\infty}(0,T; W^{4,\infty}(\Omega))$ and $F(\phi)\in C^2(\mathbb{R})$, then for the fully discrete scheme \eqref{fully-scheme1} with stabilizing constant $s\geq \frac{R^0}{2K_1}$, there exists a positive constant $C$ independent $h$ and $\Delta t$ such that
\begin{equation*}
\aligned
&\displaystyle\sum\limits_{n=1}^{k}\Delta t\|d_t e_\phi^{n+1}\|^2_{m} + \|\nabla_he_\phi^{k+1}\|_{TM}^2+\|e_\phi^{k+1}\|_m^2+|e_R^{k+1}|^2\leq C(h^4+\Delta t^2),
\endaligned
\end{equation*}  
where the positive constant $K_1$ is the lower bound of $E_h(\phi_h^n)+C$.
\end{theorem}

We shall split the proof of the above results into three lemmas below.
\begin{lemma}\label{err-lemma1}
Under the conditions of Theorem \ref{err-estimate}, there exists positive constants $C$ and $K_1$ independent $h$ and $\Delta t$ such that
\begin{equation}\label{err-lemma1-e1}
\aligned
&\displaystyle\frac{K_1}{2}\|d_te_\phi^{n+1}\|^2_m +\left(s-\frac{R_h^{n+1}}{E_h(\phi_h^n)+C}\right)K_1\epsilon^2M\frac{\|\nabla_he_\phi^{n+1}-\nabla_he_\phi^n\|_{TM}^2}{\Delta t}+R_h^{n+1}\epsilon^2M\frac{\|\nabla_he_\phi^{n+1}\|_{TM}^2}{2\Delta t}\\
&\leq C|e_R^{n+1}|^2+C\|\nabla_he_\phi^n\|_{TM}^2+C\|e_\phi^{n}\|_m^2+R_h^{n}\epsilon^2M\frac{\|\nabla_he_\phi^n\|_{TM}^2}{2\Delta t}+C(h^4+\Delta t^2).
\endaligned
\end{equation} 
\end{lemma}
\begin{proof}
Subtracting equations in \eqref{fully-scheme1} from equations in \eqref{PSSAV-e2} respectively, we obtain the following three error equations:
\begin{equation}\label{err-e1}
\displaystyle\frac{e_\phi^{n+1}-e_\phi^n}{\Delta t}=-Me_\mu^{n+1}+\left.\frac{\partial \phi}{\partial t}\right|_{t^{n+1}}-\frac{\phi(t^{n+1})-\phi(t^n)}{\Delta t},
\end{equation}
\begin{equation}\label{err-e2}
\aligned
\displaystyle e_\mu^{n+1}=
&-s\epsilon^2(\Delta_he_\phi^{n+1}-\Delta_he_\phi^{n})+\frac{R_h^{n+1}}{E_h(\phi_h^n)+C}\left[-\epsilon^2\Delta_h\phi_h^n+F'(\phi_h^n)\right]\\
&-\frac{R(t^{n+1})}{E(\phi(t^{n+1}))+C}\left[-\epsilon^2\Delta\phi(t^{n+1})+F'(\phi(t^{n+1}))\right]\\
&-s\epsilon^2(\Delta_h\phi(t^{n+1})-\Delta\phi(t^{n+1}))+s\epsilon^2(\Delta_h\phi(t^{n})-\Delta\phi(t^{n+1})),
\endaligned
\end{equation}
and
\begin{equation}\label{err-e3}
\aligned
\displaystyle\frac{e_R^{n+1}-e_R^n}{\Delta t}
&=-\frac{1}{M}\left(\frac{\phi_h^{n+1}-\phi_h^n}{\Delta t},\frac{\phi_h^{n+1}-\phi_h^n}{\Delta t}\right)_m+\frac{1}{M}\left(\frac{\partial\phi(t^{n+1})}{\partial t},\frac{\partial\phi(t^{n+1})}{\partial t}\right)_m\\
&=-\frac{1}{M}\left(\frac{\phi_h^{n+1}-\phi_h^n}{\Delta t}+\frac{\partial\phi(t^{n+1})}{\partial t},\frac{\phi_h^{n+1}-\phi_h^n}{\Delta t}-\frac{\partial\phi(t^{n+1})}{\partial t}\right)_m.
\endaligned
\end{equation}

Next we shall first make the hypotheses that there exist two positive constant $C^*$ and $C_*$ such that 
 \begin{subequations}
\begin{align}
&\|\phi_h^n\|_\infty\leq C^*,\label{err-e4a}\\
&\|e_\phi^{n}\|_m+\|\nabla_h e_\phi^{n}\|_{TM}+|e_R^n|\leq C_*(\Delta t+h^2)^{\frac12}.\label{err-e4b}
\end{align} 
\end{subequations} 
These two hypotheses will be verified in Lemma \ref{err-lemma4}. 
 
Multiplying \eqref{err-e1} by $\frac{e_\phi^{n+1}-e_\phi^n}{\Delta t}h^2$ and making summation on $i,j$ for $1\leq i\leq N_{xy}$, $1\leq j\leq N_{xy}$, we have 
\begin{equation}\label{err-e5}
\aligned
\displaystyle\|d_te_\phi^{n+1}\|_m^2=-M(e_\mu^{n+1},d_te_\phi^{n+1})_m+\left(\left.\frac{\partial \phi}{\partial t}\right|_{t^{n+1}}-\frac{\phi(t^{n+1})-\phi(t^n)}{\Delta t},d_te_\phi^{n+1}\right)_m.
\endaligned
\end{equation}

Multiplying \eqref{err-e2} by $d_te_\phi^{n+1}h^2$ and making summation on $i,j$ for $1\leq i\leq N_{xy}$, $1\leq j\leq N_{xy}$, we have 
\begin{equation}\label{err-e6}
\aligned
\displaystyle\left(e_\mu^{n+1},d_te_\phi^{n+1}\right)_m=
&\displaystyle-s\epsilon^2\left(\Delta_he_\phi^{n+1}-\Delta_he_\phi^{n},d_te_\phi^{n+1}\right)_m\\
&\displaystyle-\epsilon^2\left(\frac{R_h^{n+1}}{E_h(\phi_h^n)+C}\Delta_h\phi_h^n-\frac{R(t^{n+1})}{E(\phi(t^{n+1}))+C}\Delta\phi(t^{n+1}),d_te_\phi^{n+1}\right)_m\\
&\displaystyle+\left(\frac{R_h^{n+1}}{E_h(\phi_h^n)+C}F'(\phi_h^n)-\frac{R(t^{n+1})}{E(\phi(t^{n+1}))+C}F'(\phi(t^{n+1})),d_te_\phi^{n+1}\right)_m\\
&\displaystyle-s\epsilon^2\left(\Delta_h\phi(t^{n+1})-\Delta\phi(t^{n+1}),d_te_\phi^{n+1}\right)_m\\
&\displaystyle+s\epsilon^2\left(\Delta_h\phi(t^{n})-\Delta\phi(t^{n+1}),d_te_\phi^{n+1}\right)_m.
\endaligned
\end{equation}

For the first term in the right-hand side of the equation \eqref{err-e6}, we have
\begin{equation}\label{err-e7}
\aligned
\displaystyle-s\epsilon^2\left(\Delta_he_\phi^{n+1}-\Delta_he_\phi^{n},d_te_\phi^{n+1}\right)_m=s\epsilon^2\frac{\|\nabla_he_\phi^{n+1}-\nabla_he_\phi^{n}\|^2_{TM}}{\Delta t},
\endaligned
\end{equation}
where $\nabla_hf=d_xf+d_yf$.

For the second term in the right-hand side of the equation \eqref{err-e6}, we have
\begin{equation}\label{err-e8}
\aligned
&\displaystyle\quad-\epsilon^2\left(\frac{R_h^{n+1}}{E_h(\phi_h^n)+C}\Delta_h\phi_h^n-\frac{R(t^{n+1})}{E(\phi(t^{n+1}))+C}\Delta\phi(t^{n+1}),d_te_\phi^{n+1}\right)_m\\
&\displaystyle=-\epsilon^2\left(\frac{R_h^{n+1}}{E_h(\phi_h^n)+C}\Delta_he_\phi^n,d_te_\phi^{n+1}\right)_m-\epsilon^2\frac{R_h^{n+1}}{E_h(\phi_h^n)+C}\left(\Delta_h\phi(t^n)-\Delta\phi(t^{n+1}),d_te_\phi^{n+1}\right)_m\\
&\displaystyle\quad-\epsilon^2\left(\frac{R_h^{n+1}}{E_h(\phi_h^n)+C}-\frac{R(t^{n+1})}{E(\phi(t^{n+1}))+C}\right)\left(\Delta\phi(t^{n+1}),d_te_\phi^{n+1}\right)_m.
\endaligned
\end{equation}

For the first term in the right-hand side of \eqref{err-e8}, we have
\begin{equation}\label{err-e9}
\aligned
&\displaystyle-\epsilon^2\left(\frac{R_h^{n+1}}{E_h(\phi_h^n)+C}\Delta_he_\phi^n,d_te_\phi^{n+1}\right)_m\\
&\displaystyle=\frac{R_h^{n+1}}{E_h(\phi_h^n)+C}\epsilon^2\left(\nabla_he_\phi^n,\frac{\nabla_he_\phi^{n+1}-\nabla_he_\phi^n}{\Delta t}\right)_{TM}\\
&\displaystyle=-\frac{R_h^{n+1}}{E_h(\phi_h^n)+C}\epsilon^2\left(\frac{\|\nabla_he_\phi^n\|_{TM}^2-\|\nabla_he_\phi^{n+1}\|_{TM}^2}{2\Delta t}+\frac{\|\nabla_he_\phi^n-\nabla_he_\phi^{n+1}\|_{TM}^2}{2\Delta t}\right)\\
&\displaystyle=\frac{R_h^{n+1}}{E_h(\phi_h^n)+C}\epsilon^2\left(\frac{\|\nabla_he_\phi^{n+1}\|_{TM}^2-\|\nabla_he_\phi^{n}\|_{TM}^2}{2\Delta t}-\frac{\|\nabla_he_\phi^n-\nabla_he_\phi^{n+1}\|_{TM}^2}{2\Delta t}\right).
\endaligned
\end{equation}

Noting that $R_h^{n+1}\leq R^0\leq C_1$ and $E_h(\phi_h^n)+C>K_1>0$, then for the second term in the right-hand side of \eqref{err-e8}, by using Cauchy-Schwartz inequality, we have
\begin{equation}\label{err-e10}
\aligned
\displaystyle-\epsilon^2\frac{R_h^{n+1}}{E_h(\phi_h^n)+C}\left(\Delta_h\phi(t^n)-\Delta\phi(t^{n+1}),d_te_\phi^{n+1}\right)_m
&\displaystyle\leq\frac{1}{10M}\|d_te_\phi^{n+1}\|_m^2+C\|\Delta_h\phi(t^n)-\Delta\phi(t^{n+1})\|_m^2,\\
&\displaystyle\leq\frac{1}{10M}\|d_te_\phi^{n+1}\|_m^2+C\|\phi\|^2_{L^\infty(0,T;W^{4,\infty}(\Omega))}h^4.
\endaligned
\end{equation}

Using equation \eqref{err-e4a} and supposing $F(\phi)\in C^2(\mathbb{R})$, then we have the following inequality for the last term in the right-hand side of \eqref{err-e8}:
\begin{equation}\label{err-e11}
\aligned
&\displaystyle-\epsilon^2\left(\frac{R_h^{n+1}}{E_h(\phi_h^n)+C}-\frac{R(t^{n+1})}{E(\phi(t^{n+1}))+C}\right)\left(\Delta\phi(t^{n+1}),d_te_\phi^{n+1}\right)_m\\
&\displaystyle-\epsilon^2\frac{e_R^{n+1}}{E_h(\phi_h^n)+C}\left(\Delta\phi(t^{n+1}),d_te_\phi^{n+1}\right)_m
+\epsilon^2\frac{R(t^{n+1})\left(E(\phi(t^{n+1}))-E_h(\phi_h^n)\right)}{\left[E_h(\phi_h^n)+C\right]\left[E(\phi(t^{n+1})+C\right]}\left(\Delta\phi(t^{n+1}),d_te_\phi^{n+1}\right)_m\\
&\displaystyle\leq\frac{1}{10M}\|d_te_\phi^{n+1}\|^2+C|e_R^{n+1}|^2+C\|\nabla_he_\phi^n\|_{TM}^2+C\|e_\phi^n\|_m^2+C(h^4+\Delta t^2).
\endaligned
\end{equation}

Using similar technique and Cauchy-Schwartz inequality, we can obtain the following inequality for the third term in the right-hand side of \eqref{err-e6}:
\begin{equation}\label{err-e12}
\aligned
&\displaystyle\left(\frac{R_h^{n+1}}{E_h(\phi_h^n)+C}{F'}(\phi_h^n)-\frac{R(t^{n+1})}{E(\phi(t^{n+1}))+C}{F'}(\phi(t^{n+1})),d_te_\phi^{n+1}\right)_m\\
&\displaystyle=\frac{R_h^{n+1}}{E_h(\phi_h^n)+C}\left({F'}(\phi_h^n)-{F'}(\phi(t^{n+1})),d_te_\phi^{n+1}\right)_m\\
&\displaystyle\quad+\left(\frac{R_h^{n+1}}{E_h(\phi_h^n)+C}-\frac{R(t^{n+1})}{E_h(\phi(t^{n+1}))+C}\right)\left({F'}(\phi(t^{n+1})),d_te_\phi^{n+1}\right)_m\\
&\displaystyle\leq \frac{1}{10M}\|d_te_\phi^{n+1}\|_m^2+C\|e_\phi^{n}\|_m^2+C(\Delta t)^2\\
&\displaystyle\quad+\frac{1}{10M}\|d_te_\phi^{n+1}\|_m^2+C|e_R^{n+1}|^2+C\|\nabla_he_\phi^n\|_{TM}^2+C\|e_\phi^n\|_m^2+C(h^4+\Delta t^2)\\
&\displaystyle\leq\frac{1}{5M}\|d_te_\phi^{n+1}\|_m^2+C\|e_\phi^{n}\|_m^2+C|e_R^{n+1}|^2+C\|\nabla_he_\phi^n\|_{TM}^2+C(h^4+\Delta t^2).
\endaligned
\end{equation}

For the last two terms in the right-hand side of \eqref{err-e6}, using Cauchy-Schwartz inequality, we have
\begin{equation}\label{err-e13}
\aligned
&\displaystyle-s\epsilon^2\left(\Delta_h\phi(t^{n+1})-\Delta\phi(t^{n+1}),d_te_\phi^{n+1}\right)_m+s\epsilon^2\left(\Delta_h\phi(t^{n})-\Delta\phi(t^{n+1}),d_te_\phi^{n+1}\right)_m\\
&\displaystyle\leq\frac{1}{10M}\|d_te_\phi^{n+1}\|_m^2+C(h^4+\Delta t^2)\left(\|\phi\|^2_{L^\infty(0,T;W^{4,\infty}(\Omega))} + \|\phi\|^2_{W^{1,\infty}(0,T; W^{2,2}(\Omega))}\right).
\endaligned
\end{equation}

Using \eqref{err-e4a}, we obtain there are two positive constants $K_1$, $K_2$ to satisfy $0<K_1<E_h(\phi_h^n)+C<2(E(\phi(t^n))+C)<K_2$. We choose $s\geq \frac{R^0}{2K_1}$ to satisfy that $s-\frac{R_h^{n+1}}{E_h(\phi_h^n)+C}>0$. Multiplying both sides of the inequality \eqref{err-e14} by $E_h(\phi_h^n)+C$ yields
\begin{equation*}
\aligned
&\displaystyle\frac{K_1}{2}\|d_te_\phi^{n+1}\|_m^2+\left(s-\frac{R_h^{n+1}}{E_h(\phi_h^n)+C}\right)K_1\epsilon^2M\frac{\|\nabla_he_\phi^{n+1}-\nabla_he_\phi^n\|_{TM}^2}{\Delta t}+R_h^{n+1}\epsilon^2M\frac{\|\nabla_he_\phi^{n+1}\|_{TM}^2}{2\Delta t}\\
&\leq C|e_R^{n+1}|^2+C\|\nabla_he_\phi^n\|_{TM}^2+C\|e_\phi^{n}\|_m^2+R_h^{n}\epsilon^2M\frac{\|\nabla_he_\phi^n\|_{TM}^2}{2\Delta t}+C(h^4+\Delta t^2).
\endaligned
\end{equation*} 
\end{proof}

\begin{lemma}\label{err-lemma2}
Under the conditions of Theorem \ref{err-estimate}, there exists a positive constant $C$ independent $h$ and $\Delta t$ such that
\begin{equation}\label{err-e18}
\aligned
&\displaystyle\frac{\|e_\phi^{n+1}\|_m^2-\|e_\phi^{n}\|_m^2}{2\Delta t}+\frac{\|e_\phi^{n+1}-e_\phi^{n}\|_m^2}{2\Delta t}+MS\epsilon^2\left(\frac{\|\nabla_he_\phi^{n+1}\|_{TM}^2-\|\nabla_he_\phi^{n}\|_{TM}^2}{2}+\frac{\|\nabla_he_\phi^{n+1}-\nabla_he_\phi^{n}\|_{TM}^2}{2}\right)\\
&\leq C\|e_\phi^{n+1}\|_m^2+C\|\nabla_he_\phi^{n+1}\|_{TM}^2+C|e_R^{n+1}|^2+C(h^4+\Delta t^2).
\endaligned
\end{equation} 
\end{lemma}
\begin{proof}
Combining \eqref{err-e5}$\sim$\eqref{err-e6} with above inequalities \eqref{err-e7}$\sim$\eqref{err-e13}, we can obtain that
\begin{equation}\label{err-e14}
\aligned
&\displaystyle\frac{1}{2}\|d_te_\phi^{n+1}\|_m^2+\left(s-\frac{R_h^{n+1}}{E_h(\phi_h^n)+C}\right)\epsilon^2M\frac{\|\nabla_he_\phi^{n+1}-\nabla_he_\phi^n\|_{TM}^2}{\Delta t}+\frac{R_h^{n+1}}{E_h(\phi_h^n)+C}\epsilon^2M\frac{\|\nabla_he_\phi^{n+1}\|_{TM}^2}{2\Delta t}\\
&\leq C|e_R^{n+1}|^2+C\|\nabla_he_\phi^n\|_{TM}^2+C\|e_\phi^{n}\|_m^2+\frac{R_h^{n}}{E_h(\phi_h^n)+C}\epsilon^2M\frac{\|\nabla_he_\phi^n\|_{TM}^2}{2\Delta t}+C(h^4+\Delta t^2).
\endaligned
\end{equation} 

Next we multiply \eqref{err-e1} by $e_\phi^{n+1}h^2$, make summation on $i,j$ for $1\leq i\leq N_{xy}$, $1\leq j\leq N_{xy}$, and combine it with \eqref{err-e2} to obtain
\begin{equation}\label{err-e15}
\aligned
&\displaystyle\quad\left(\frac{e_\phi^{n+1}-e_\phi^n}{\Delta t},e_\phi^{n+1}\right)_m-Ms\epsilon^2\left(\Delta_he_\phi^{n+1}-\Delta_he_\phi^{n},e_\phi^{n+1}\right)_m\\
&\displaystyle=-\left(\frac{MR_h^{n+1}}{E_h(\phi_h^n)+C}\left[-\epsilon^2\Delta_h\phi_h^n+{F'}(\phi_h^n)\right]
-\frac{MR(t^{n+1})}{E_h(\phi(t^{n+1}))+C}\left[-\epsilon^2\Delta\phi(t^{n+1})+{F'}(\phi(t^{n+1}))\right],e_\phi^{n+1}\right)_m\\
&\displaystyle\quad+s\epsilon^2M\left(\Delta_h\phi(t^{n+1})-\Delta\phi(t^{n+1}),e_\phi^{n+1}\right)_m
-s\epsilon^2M\left(\Delta_h\phi(t^{n})-\Delta\phi(t^{n+1}),e_\phi^{n+1}\right)_m\\
&\displaystyle\quad+\left(\left.\frac{\partial\phi}{\partial t}\right|_{t^{n+1}}-\frac{\phi(t^{n+1})-\phi(t^{n})}{\Delta t},e_\phi^{n+1}\right)_m=RHD.
\endaligned
\end{equation}

For all terms on the left-hand side of \eqref{err-e15}, we have
\begin{equation}\label{err-e16}
\aligned
&\displaystyle\quad\left(\frac{e_\phi^{n+1}-e_\phi^n}{\Delta t},e_\phi^{n+1}\right)_m-Ms\epsilon^2\left(\Delta_he_\phi^{n+1}-\Delta_he_\phi^{n},e_\phi^{n+1}\right)_m\\
&\displaystyle=\frac{\|e_\phi^{n+1}\|_m^2-\|e_\phi^{n}\|_m^2}{2\Delta t}+\frac{\|e_\phi^{n+1}-e_\phi^{n}\|_m^2}{2\Delta t}+MS\epsilon^2\left(\frac{\|\nabla_he_\phi^{n+1}\|_{TM}^2-\|\nabla_he_\phi^{n}\|_{TM}^2}{2}+\frac{\|\nabla_he_\phi^{n+1}-\nabla_he_\phi^{n}\|_{TM}^2}{2}\right).
\endaligned
\end{equation}

Using similar technique for the right-hand side of \eqref{err-e6}, we can obtain the following inequality for the right-hand side of \eqref{err-e15}:
\begin{equation}\label{err-e17}
\aligned
RHD\leq C\|e_\phi^{n+1}\|_m^2+C\|\nabla_he_\phi^{n+1}\|_{TM}^2+C|e_R^{n+1}|^2+C(h^4+\Delta t^2).
\endaligned
\end{equation}

Combining \eqref{err-e15} with above inequalities \eqref{err-e16}$\sim$\eqref{err-e17}, we can obtain that
\begin{equation*}
\aligned
&\displaystyle\frac{\|e_\phi^{n+1}\|_m^2-\|e_\phi^{n}\|_m^2}{2\Delta t}+\frac{\|e_\phi^{n+1}-e_\phi^{n}\|_m^2}{2\Delta t}+MS\epsilon^2\left(\frac{\|\nabla_he_\phi^{n+1}\|_{TM}^2-\|\nabla_he_\phi^{n}\|_{TM}^2}{2}+\frac{\|\nabla_he_\phi^{n+1}-\nabla_he_\phi^{n}\|_{TM}^2}{2}\right)\\
&\leq C\|e_\phi^{n+1}\|_m^2+C\|\nabla_he_\phi^{n+1}\|_{TM}^2+C|e_R^{n+1}|^2+C(h^4+\Delta t^2).
\endaligned
\end{equation*} 
\end{proof}

We next give the estimate analysis for $|e_R^{n+1}|$.
\begin{lemma}\label{err-lemma3}
Under the conditions of Theorem \ref{err-estimate}, there exists a positive constant $C$ independent $h$ and $\Delta t$ such that
\begin{equation}\label{err-e23}
\aligned
\displaystyle\frac{|e_R^{n+1}|^2-|e_R^{n}|^2}{2\Delta t}+\frac{|e_R^{n+1}-e_R^{n}|^2}{2\Delta t}
\leq \frac14\|d_t e_\phi^{n+1}\|^2_m +C\|d_t e_\phi^{n+1}\|^2_m |e_R^{n+1}|^2+C|e_R^{n+1}|^2.
\endaligned
\end{equation} 
\end{lemma}
\begin{proof}
Multiplying \eqref{err-e3} with $e_R^{n+1}$ results in
\begin{equation}\label{err-e19}
\aligned
\displaystyle\frac{|e_R^{n+1}|^2-|e_R^{n}|^2}{2\Delta t}+\frac{|e_R^{n+1}-e_R^{n}|^2}{2\Delta t}
&\displaystyle=-\frac{e_R^{n+1}}{M}\left(\frac{e_\phi^{n+1}-e_\phi^n}{\Delta t}+\frac{\phi(t^{n+1})-\phi(t^n)}{\Delta t}+\frac{\partial \phi(t^{n+1})}{\partial t},\frac{e_\phi^{n+1}-e_\phi^n}{\Delta t}\right)_m\\
&\displaystyle=-\frac{e_R^{n+1}}{M}\|d_te_\phi^{n+1}\|^2_m -\frac{e_R^{n+1}}{M}\left(\frac{\phi(t^{n+1})-\phi(t^n)}{\Delta t}+\frac{\partial \phi(t^{n+1})}{\partial t},d_te_\phi^{n+1}\right)_m .
\endaligned
\end{equation} 

For the first term in the right-hand side of above equation \eqref{err-e19}, using Cauchy-Schwartz inequality, we have
\begin{equation}\label{err-e20}
\aligned
\displaystyle-\frac{e_R^{n+1}}{M}\|d_te_\phi^{n+1}\|_m^2\leq\frac18\|d_t e_\phi^{n+1}\|^2_m +C\|d_t e_\phi^{n+1}\|^2_m |e_R^{n+1}|^2.
\endaligned
\end{equation} 

For the second term in the right-hand side of above equation \eqref{err-e19}, using Cauchy-Schwartz inequality, we have
\begin{equation}\label{err-e21}
\aligned
\displaystyle-\frac{e_R^{n+1}}{M}\left(\frac{\phi(t^{n+1})-\phi(t^n)}{\Delta t}+\frac{\partial \phi(t^{n+1})}{\partial t},d_te_\phi^{n+1}\right)_m
&\displaystyle\leq\frac{|e_R^{n+1}|}{M}\|\frac{\phi(t^{n+1})-\phi(t^n)}{\Delta t}+\frac{\partial \phi(t^{n+1})}{\partial t}\|_m \|d_te_\phi^{n+1}\|_m\\
&\displaystyle\leq\frac18\|d_t e_\phi^{n+1}\|^2_m +C|e_R^{n+1}|^2.
\endaligned
\end{equation} 
 
Combining \eqref{err-e19} with above inequalities \eqref{err-e20}$\sim$\eqref{err-e21}, we can obtain that
\begin{equation*}
\aligned
\displaystyle\frac{|e_R^{n+1}|^2-|e_R^{n}|^2}{2\Delta t}+\frac{|e_R^{n+1}-e_R^{n}|^2}{2\Delta t}
\leq \frac14\|d_t e_\phi^{n+1}\|^2_m +C\|d_t e_\phi^{n+1}\|^2_m |e_R^{n+1}|^2+C|e_R^{n+1}|^2.
\endaligned
\end{equation*} 
\end{proof}

\begin{lemma}\label{err-lemma4}
Under the conditions of Theorem \ref{err-estimate}, there exists two positive constants $C^*$ and $C_*$ independent $h$ and $\Delta t$ such that
\begin{equation*}
\aligned
&\|\phi_h^n\|_\infty\leq C^*,\\
&\|e_\phi^{n}\|_m+\|\nabla_h e_\phi^{n}\|_{TM}+|e_R^n|\leq C_*(\Delta t+h^2)^{\frac12}.
\endaligned
\end{equation*} 
\end{lemma}
\begin{proof}
Using the scheme \eqref{fully-scheme1} for $n=0$ and applying the inverse assumption, we can get the approximation $\phi_h^1$ with the following property:
\begin{equation*}
\aligned
\|\phi_h^1\|_{\infty}\leq&\|\phi_h^1-\phi(t^1)\|_{\infty}+\|\phi(t^1)\|_{\infty}
\leq \|\phi_h^1-\Pi_h\phi(t^1)\|_{\infty}+\|\Pi_h\phi(t^1)-\phi(t^1)\|_{\infty}+\|\phi(t^1)\|_{\infty}\\
\leq&Ch^{-1}(\|\phi_h^1-\phi(t^1) \|_m+\|\phi(t^1)-\Pi_h\phi(t^1)\|_m)+\|\Pi_h\phi(t^1)-\phi(t^1)\|_{\infty}+\|\phi(t^1)\|_{\infty}\\
\leq&C(h+h^{-1}\Delta t)+\|\phi(t^1)\|_{\infty}\leq C.
\endaligned
\end{equation*}
where $\Pi_h$ is an bilinear interpolant operator with the following estimate:
\begin{equation}\label{e_H1_error_estimate_boundedness_added2}
\aligned
\|\Pi_h\phi^1-\phi^1\|_{\infty}\leq Ch^2.
\endaligned
\end{equation}
 Thus we can choose the positive constant $C^*$ independent of $h$ and $\Delta t$ such that
\begin{align*}
C^*&\geq \max\{\|\phi_h^{1}\|_{\infty}, 2\|\phi(t^{n})\|_{\infty}\}.
\end{align*}

By the definition of $C^*$, it is trivial that hypothesis \eqref{err-e4a} holds true for $n=1$. Supposing that $\|\phi_h^{k}\|_{\infty}\leq C^*$ holds true for an integer $k=1,\cdots,n$, with the aid of Lemmas \ref{err-lemma1}$\sim$\ref{err-lemma3}, we have that
$$\|\phi_h^{k}-\phi(t^k)\|_m\leq C(\Delta t+h^2).$$
Next we prove that $\|\phi_h^{n+1}\|_{\infty}\leq C^*$ holds true.
Since
\begin{equation}\label{e_H1_error_estimate_boundedness_added3}
\aligned
\|\phi_h^{n+1}\|_{\infty}
&\leq\|\phi_h^{n+1}-\phi(t^{n+1})\|_{\infty}+\|\phi(t^{n+1})\|_{\infty}\\
&\leq \|\phi_h^{n+1}-\Pi_h\phi(t^{n+1})\|_{\infty}+\|\Pi_h\phi(t^{n+1})-\phi(t^{n+1})\|_{\infty}+\|\phi(t^{n+1})\|_{\infty}\\
&\leq Ch^{-1}(\|\phi_h^{n+1}-\phi(t^{n+1}) \|_m+\|\phi(t^{n+1})-\Pi_h\phi(t^{n+1})\|_m)\\
&\quad+\|\Pi_h\phi(t^{n+1})-\phi(t^{n+1})\|_{\infty}+\|\phi(t^{n+1})\|_{\infty}\\
&\leq C_1(h+h^{-1}\Delta t)+\|\phi(t^{n+1})\|_{\infty}.
\endaligned
\end{equation}
Let $\Delta t\leq C_2h^2$ and a positive constant $h_1$ be small enough to satisfy
$$C_1(1+C_2)h_1\leq\frac{C^*}{2}.$$
Then for $h\in (0,h_1],$ we derive from (\ref{e_H1_error_estimate_boundedness_added3}) that
\begin{equation*}
\aligned
\|\phi_h^{n+1}\|_{\infty}
\leq&C_1(h+h^{-1}\Delta t)+\|\phi(t^{n+1})\|_{\infty}\\
\leq &C_1(h_1+C_2h_1)+\frac{C^*}{2}
\leq C^*.
\endaligned
\end{equation*}
This completes the induction \eqref{err-e4a}.   

We next give a proof of the second inequality \eqref{err-e4b}. It is obvious that $\|e_\phi^{0}\|_m+\|\nabla_h e_\phi^{0}\|_{TM}+|e_R^0|=0\leq C_*(\Delta t+h^2)^{\frac12}$. Assume that $\|e_\phi^{k}\|_m+\|\nabla_h e_\phi^{k}\|_{TM}+|e_R^k|\leq C_*(\Delta t+h^2)^{\frac12}$ for all $n=1,2,\ldots k$, then for $n=k+1$, from Theorem \ref{err-estimate}, we have
\begin{equation}\label{err-lemma4-e21}
\aligned
\|e_\phi^{k+1}\|_m +\|\nabla_h e_\phi^{k+1}\|_{TM}+|e_R^{k+1}|\leq C(\Delta t+h^2).
\endaligned
\end{equation} 

We choose sufficiently small $\Delta t$ and $h$ such that $C(\Delta t+h^2)^{\frac12}\leq C_*$, then above equality \eqref{err-lemma4-e21} can be transformed as
\begin{equation}\label{err-lemma4-e22}
\aligned
\|e_\phi^{k+1}\|_m +\|\nabla_h e_\phi^{k+1}\|_{TM}+|e_R^{k+1}|\leq C_*(\Delta t+h^2)^{\frac12},
\endaligned
\end{equation} 
which completes the proof.
\end{proof}

We are now in position to prove our main results of Theorem \ref{err-estimate}. Combining Lemmas \ref{err-lemma1}$\sim$\ref{err-lemma3}, we have
\begin{equation}\label{err-e24}
\aligned
&\displaystyle\frac{K_1}{4}\|d_t e_\phi^{n+1}\|^2_m +R_h^{n+1}\epsilon^2M\frac{\|\nabla_he_\phi^{n+1}\|_{TM}^2}{2\Delta t}+K_1\frac{\|e_\phi^{n+1}\|_m^2-\|e_\phi^{n}\|_m^2}{2\Delta t}+K_1\frac{\|e_\phi^{n+1}-e_\phi^n\|_m^2}{2\Delta t}\\
&+Ms\epsilon^2K_1\frac{\|\nabla_he_\phi^{n+1}\|_{TM}^2-\|\nabla_he_\phi^{n}\|_{TM}^2}{2}+Ms\epsilon^2K_1\frac{\|\nabla_he_\phi^{n+1}-\nabla_he_\phi^n\|_{TM}^2}{2}\\
&+K_1\frac{|e_R^{n+1}|^2-|e_R^n|^2}{2\Delta t}+K_1\frac{|e_R^{n+1}-e_R^{n}|^2}{2\Delta t}+\left(s-\frac{R^{n+1}}{2(E_h(\phi^n)+C)}\right)\epsilon^2MK_1\frac{\|\nabla_he_\phi^{n+1}-\nabla_he_\phi^n\|_{TM}^2}{\Delta t}\\
&\leq C|e_R^{n+1}|^2+C\|\nabla_he_\phi^n\|_{TM}^2+C\|e_\phi^n\|_m^2+C\|d_t e_\phi^{n+1}\|^2_m |e_R^{n+1}|^2++R_h^{n}\epsilon^2M\frac{\|\nabla_he_\phi^n\|_{TM}^2}{2\Delta t}+C(h^4+\Delta t^2).
\endaligned
\end{equation} 

Multiplying both sides of above inequality \eqref{err-e24} by $\Delta t$ and making summation on $n$ from $0$ to $k$ yields
\begin{equation}\label{err-e25}
\aligned
&\displaystyle\frac{K_1}{4}\sum\limits_{n=0}^{k}\Delta t\|d_t e_\phi^{n+1}\|^2_m +\frac{\epsilon^2M}{2}\sum\limits_{n=0}^{k}R_h^{n+1}\|\nabla_he_\phi^{n+1}\|_{TM}^2+\frac{K_1}{2}\|e_\phi^{k+1}\|_m^2+\frac{K_1}{2}\sum\limits_{n=0}^{k}\|e_\phi^{n+1}-e_\phi^n\|_m^2\\
&+\frac{Ms\epsilon^2K_1}{2}\Delta t\|\nabla_he_\phi^{k+1}\|_{TM}^2+\frac{Ms\epsilon^2K_1}{2}\sum\limits_{n=0}^{k}\Delta t\|\nabla_he_\phi^{n+1}-\nabla_he_\phi^n\|_{TM}^2+\frac{K_1}{2}|e_R^{k+1}|^2\\
&+\frac{K_1}{2}\sum\limits_{n=0}^{k}|e_R^{n+1}-e_R^{n}|^2+\left(s-\frac{R^{n+1}}{2(E_h(\phi^n)+C)}\right)\epsilon^2MK_1\sum\limits_{n=0}^{k}\|\nabla_he_\phi^{n+1}-\nabla_he_\phi^n\|_{TM}^2\\
&\leq C\sum\limits_{n=0}^{k}\Delta t|e_R^{n+1}|^2+C\sum\limits_{n=0}^{k}\Delta t\|\nabla_he_\phi^n\|_{TM}^2+C\sum\limits_{n=0}^{k}\Delta t\|e_\phi^n\|_m^2+C\sum\limits_{n=0}^{k}\Delta t\|d_t e_\phi^{n+1}\|^2_m |e_R^{n+1}|^2\\
&\quad+\frac{\epsilon^2M}{2}\sum\limits_{n=0}^{k}R_h^{n}\|\nabla_he_\phi^n\|_{TM}^2+C(h^4+\Delta t^2).
\endaligned
\end{equation} 

Noting that $0<R_h^{n+1}<R^0$, we have $|e_R^{n+1}|^2\leq C$ for a constant $C$. Then for the fourth term in the right-hand side of above inequality \eqref{err-e25}, using Lemma \ref{err-lemma1}, we have
\begin{equation}\label{err-e26}
\aligned
C\sum\limits_{n=0}^{k}\Delta t & \|d_t e_\phi^{n+1}\|^2_m  |e_R^{n+1}|^2
\leq C\sum\limits_{n=0}^{k}\Delta t\|d_t e_\phi^{n+1}\|^2_m \\
&\leq C\sum\limits_{n=0}^{k}\Delta t|e_R^{n+1}|^2+C\sum\limits_{n=0}^{k}\Delta t\|\nabla_he_\phi^n\|_{TM}^2+C\sum\limits_{n=0}^{k}\Delta t\|e_\phi^n\|_m^2+C(h^4+\Delta t^2).
\endaligned
\end{equation} 

Subtracting above inequality \eqref{err-e26} into \eqref{err-e25}, we have
\begin{equation}\label{err-e27}
\aligned
&\displaystyle\frac{K_1}{4}\sum\limits_{n=0}^{k}\Delta t\|d_t e_\phi^{n+1}\|^2_{m} +\frac{\epsilon^2M}{2}R_h^{k+1}\|\nabla_he_\phi^{k+1}\|_{TM}^2+\frac{K_1}{2}\|e_\phi^{k+1}\|_m^2+\frac{K_1}{2}|e_R^{k+1}|^2\\
&\leq C\sum\limits_{n=1}^{k}\Delta t|e_R^{n+1}|^2+C\sum\limits_{n=1}^{k}\Delta t\|\nabla_he_\phi^n\|_{TM}^2+C\sum\limits_{n=1}^{k}\Delta t\|e_\phi^n\|_m^2+C(h^4+\Delta t^2).
\endaligned
\end{equation} 

From energy dissipation law in Theorem \ref{PSSAV-theorem2}, we know that 
\begin{equation}\label{err-e28}
\aligned
R_h^{k+1}=R_h^k-\frac{\Delta t}{M}\|d_t\phi^{n+1}\|^2_m =R_h^k-\frac{\Delta t}{M}\|d_te_\phi^{n+1}+d_t\phi(t^{n+1})\|^2_m.
\endaligned
\end{equation}

Subtracting above equation \eqref{err-e28} into the second term of the left-hand side of \eqref{err-e27}, we have
\begin{equation}\label{err-e29}
\aligned
\frac{\epsilon^2M}{2}R_h^{k+1}\|\nabla_he_\phi^{k+1}\|_{TM}^2=\frac{\epsilon^2M}{2}R_h^k\|\nabla_he_\phi^{k+1}\|_{TM}^2-\frac{\epsilon^2\Delta t}{2}\|d_te_\phi^{k+1}+d_t\phi(t^{k+1})\|^2_m \|\nabla_he_\phi^{k+1}\|_{TM}^2
\endaligned
\end{equation} 
 
For the first term in the right-hand side of above equation \eqref{err-e29}, noting that $R(t^k)=E(\phi(t^k))+C>K_1$, we have   
\begin{equation}\label{err-e30}
\aligned
\frac{\epsilon^2M}{2}R_h^k\|\nabla_he_\phi^{k+1}\|_{TM}^2=\frac{\epsilon^2M}{2}(e_R^k+R(t^k))\|\nabla_he_\phi^{k+1}\|_{TM}^2\geq\frac{\epsilon^2MK_1}{4}\|\nabla_he_\phi^{k+1}\|_{TM}^2.
\endaligned
\end{equation}

For the second term in the right-hand side of above equation \eqref{err-e29}, using Cauchy-Schwartz inequality, we have   
\begin{equation}\label{err-e31}
\aligned
\frac{\epsilon^2\Delta t}{2}\|d_te_\phi^{k+1}+d_t\phi(t^{n+1})\|^2_m \|\nabla_he_\phi^{k+1}\|_{TM}^2
&\leq\epsilon^2\Delta t\left(\|d_te_\phi^{k+1}\|^2_m +\|d_t\phi(t^{k+1})\|^2_m \right)\|\nabla_he_\phi^{k+1}\|_{TM}^2\\
&\leq\epsilon^2\Delta t\|d_te_\phi^{k+1}\|^2_m \|\nabla_he_\phi^{k+1}\|_{TM}^2+C\Delta t\|\nabla_he_\phi^{k+1}\|_{TM}^2.
\endaligned
\end{equation} 

Multiplying both sides of above inequality \eqref{err-lemma1-e1} by $2\epsilon^2\Delta t$ and using \eqref{err-e4b} yield
\begin{equation}\label{err-e32}
\aligned
\epsilon^2\Delta t\|d_te_\phi^{k+1}\|^2_m 
&\leq C\Delta t|e_R^{k+1}|^2+C\Delta t\|\nabla_he_\phi^k\|_{TM}^2+C\Delta t\|e_\phi^{k}\|_m^2+R_h^{k}\epsilon^4M\|\nabla_he_\phi^k\|_{TM}^2+C(h^4+\Delta t^2)\\
&\leq C\Delta t|e_R^{k+1}|^2+C(\Delta t+h^2).
\endaligned
\end{equation}

Noting that $0<R_h^{k+1}<R^0$ and $K_1<R(t^{k+1})<K_2$, we have $|e_R^{k+1}|<C$. Combining inequality \eqref{err-e32} with \eqref{err-e31}, we obtain 
\begin{equation}\label{err-e33}
\aligned
\frac{\epsilon^2\Delta t}{2}\|d_te_\phi^{k+1}+d_t\phi(t^{n+1})\|^2_m \|\nabla_he_\phi^{k+1}\|_{TM}^2
\leq C\Delta t\|\nabla_he_\phi^{k+1}\|_{TM}^2.
\endaligned
\end{equation}

Subtracting \eqref{err-e28}$\sim$\eqref{err-e33} into \eqref{err-e27}, we have 
\begin{equation}\label{err-e34}
\aligned
&\displaystyle\frac{K_1}{4}\sum\limits_{n=0}^{k}\Delta t\|d_t e_\phi^{n+1}\|^2_m +\frac{\epsilon^2MK_1}{4}\|\nabla_he_\phi^{k+1}\|_{TM}^2+\frac{K_1}{2}\|e_\phi^{k+1}\|_m^2+\frac{K_1}{2}|e_R^{k+1}|^2\\
&\leq C\sum\limits_{n=1}^{k}\Delta t|e_R^{n+1}|^2+C\sum\limits_{n=1}^{k}\Delta t\|\nabla_he_\phi^n\|_{TM}^2+C\sum\limits_{n=1}^{k}\Delta t\|e_\phi^n\|_m^2+C(h^4+\Delta t^2).
\endaligned
\end{equation} 
 
Using Gronwall inequality for above inequality, we obtain  
\begin{equation}\label{err-e35}
\aligned
&\displaystyle\frac{K_1}{4}\sum\limits_{n=1}^{k}\Delta t\|d_t e_\phi^{n+1}\|^2_m +\frac{\epsilon^2MK_1}{4}\|\nabla_he_\phi^{k+1}\|_{TM}^2+\frac{K_1}{2}\|e_\phi^{k+1}\|_m^2+\frac{K_1}{2}|e_R^{k+1}|^2\leq C(h^4+\Delta t^2).
\endaligned
\end{equation} 
\section{Second-order PS-SAV scheme}
A similar PS-SAV approach can also be extended to construct a second-order Crank-Nicloson formulation for the $L^2$ gradient flow. We find that a straightforward extension of the first-order PS-SAV scheme to the second-order scheme can not preserve the positive property of $R^{n+1}$. we add a stabilization term $s^{n+1}\Delta t(R^{n+1}-R^n)$ to overcome this problem. 

The second-order PS-SAV scheme based on the Crank-Nicolson formulation is given by:
\begin{equation}\label{PSSAV-e10}
   \begin{array}{l}
\displaystyle\frac{\phi^{n+1}-\phi^n}{\Delta t}=-M\mu^{n+\frac12},\\
\displaystyle\mu^{n+\frac12}=\frac12\epsilon^2A(\phi^{n+1}+\phi^{n})+(\frac{R^{n+1}+R^n}{2(E(\widehat{\phi}^{n+\frac{1}{2}})+C)}-1)\epsilon^2A\phi^{n}+\frac{R^{n+1}+R^n}{2(E(\widehat{\phi}^{n+\frac{1}{2}})+C)}F'(\widehat{\phi}^{n+\frac{1}{2}}),\\
\displaystyle \frac{R^{n+1}-R^n}{\Delta t}+s^{n+1}\Delta t(R^{n+1}-R^n)=\displaystyle-\frac1M(\frac{\phi^{n+1}-\phi^n}{\Delta t},\frac{\phi^{n+1}-\phi^n}{\Delta t}).
   \end{array}
\end{equation}

The second equation in \eqref{PSSAV-e10} can be rewritten as the following equivalent system:
\begin{equation}\label{PSSAV-e10-2}
\displaystyle\mu^{n+\frac12}=\frac12\epsilon^2A(\phi^{n+1}-\phi^{n})+\frac{R^{n+1}+R^n}{2(E(\widehat{\phi}^{n+\frac{1}{2}})+C)}\left[\epsilon^2A\phi^{n}+F'(\widehat{\phi}^{n+\frac{1}{2}})\right].
\end{equation} 
Combining the first equation in \eqref{PSSAV-e10} with the equivalent equation \eqref{PSSAV-e10-2} of the second one, we can rewrite \eqref{PSSAV-e10} equivalently as the following formulation:
\begin{equation}\label{PSSAV-e11}
   \begin{array}{l}
\displaystyle2(E(\widehat{\phi}^{n+\frac{1}{2}})+C)(I+\frac{1}{2}M\epsilon^2\Delta tA)\frac{\phi^{n+1}-\phi^n}{\Delta t}=-M(R^{n+1}+R^n)\left[\epsilon^2A \phi^{n}+F'(\widehat{\phi}^{n+\frac{1}{2}})\right],\\
\displaystyle \frac{1}{\Delta t}(R^{n+1}-R^n)+s^{n+1}\Delta t(R^{n+1}-R^n)=\displaystyle-\frac1M(\frac{\phi^{n+1}-\phi^n}{\Delta t},\frac{\phi^{n+1}-\phi^n}{\Delta t}).
   \end{array}
\end{equation}
Setting $\phi^{n+1}=\phi^{n}+\Delta t(R^{n+1}+R^n)\phi_1^{n+1}$, we also find that $\phi_1^{n+1}$ is the solution of the following linear equation with constant coefficients:
\begin{equation}\label{PSSAV-e12}
\displaystyle2(E(\widehat{\phi}^{n+\frac{1}{2}})+C)(I+M\epsilon^2\Delta tA)\phi_1^{n+1}=-M\left[\epsilon^2A \phi^{n}+F'(\widehat{\phi}^{n+\frac{1}{2}})\right].
\end{equation}
Once $\phi_1^{n+1}$ is known, noting that  
\begin{equation}\label{PSSAV-e13}
\displaystyle\phi^{n+1}-\phi^n=\Delta t(R^{n+1}+R^n)\phi_1^{n+1},
\end{equation}
and combining it with the second equation in \eqref{PSSAV-e11}, we get
\begin{equation}\label{PSSAV-e14}
\displaystyle a(R^{n+1})^2+bR^{n+1}+c=0,
\end{equation}
where the coefficients $a$, $b$ and $c$ of the above quadratic equation satisfy
\begin{equation*}
\aligned
&a=(\phi_1^{n+1},\phi_1^{n+1}),\quad b=\frac{M}{\Delta t}+Ms^{n+1}\Delta t+2R^n(\phi_1^{n+1},\phi_1^{n+1}),\\ &c=(R^n)^2(\phi_1^{n+1},\phi_1^{n+1})-\frac{M}{\Delta t}R^n-Ms^{n+1}\Delta tR^n.
\endaligned
\end{equation*}   
If $\phi_1^{n+1}=0$, we set $s^{n+1}=0$. Then we have $a=0$, $b=\frac{M}{\Delta t}$ and $c=-\frac{M}{\Delta t}R^n$, then we immediately obtain $\phi^{n+1}=\phi^n$ and $R^{n+1}=R^n$. If $\phi_1^{n+1}\neq0$, we obtain $(\phi_1^{n+1},\phi_1^{n+1})>0$. Then the above equation \eqref{PSSAV-e14} is a quadratic equation with one variable for $R^{n+1}$. 

\begin{theorem}\label{PSSAV-theorem3}
If we choose the stabilized variable $s^{n+1}$ to satisfy 
 \begin{equation}\label{PSSAV-e15}
s^{n+1}=\left\{
   \begin{array}{llr}
0,&&R^n(\phi_1^{n+1},\phi_1^{n+1})\leq\frac{M}{\Delta t},\\
\frac{1}{M\Delta t}R^n(\phi_1^{n+1},\phi_1^{n+1})-\frac{1}{\Delta t^2},&& R^n(\phi_1^{n+1},\phi_1^{n+1})>\frac{M}{\Delta t}.
   \end{array}
   \right.
\end{equation} 
then the quadratic equation with one variable for $R^{n+1}$ \eqref{PSSAV-e11} has and only one positive solution:
\begin{equation}\label{PSSAV-e16}
R^{n+1}=\frac{-b+\sqrt{b^2-4ac}}{2a}>0.
\end{equation}
\end{theorem}
\begin{proof}
Noting that $R^{0}=E(\phi_0)+C>0$, we then assume that $R^n>0$. Noting that $a>0$, if the stabilized variable $s^{n+1}$ is chosen as in \eqref{PSSAV-e15}, then we are easy to obtain $c<0$, then the quadratic equation \eqref{PSSAV-e11} is determined to have a solution because of 
\begin{equation*}
\aligned
\displaystyle \Delta=b^2-4ac>0.
\endaligned
\end{equation*}
Similarly, one can see that \eqref{PSSAV-e14} has the following two solutions:
\begin{equation*}
\begin{array}{l}
\displaystyle R_1^{n+1}=\frac{-b-\sqrt{b^2-4ac}}{2a}<0,\\
\displaystyle R_2^{n+1}=\frac{-b+\sqrt{b^2-4ac}}{2a}>0.
\end{array}
\end{equation*}
By the positive property of $R$, we have that $R^{n+1}=R_2^{n+1}$.
\end{proof}

Then we can compute $\phi^{n+1}$ by the following equation:
\begin{equation*}
\displaystyle\phi^{n+1}=\phi^n+\Delta t(R^{n+1}+R^n)\phi_1^{n+1}.
\end{equation*}

To summarize, the Second-order PS-SAV scheme \eqref{PSSAV-e11} can be implemented as follows:
\begin{itemize}
  \item solve $\phi_1^{n+1}$ from \eqref{PSSAV-e12};
  \item compute $R^{n+1}$ from \eqref{PSSAV-e16};
  \item update $\phi^{n+1}=\phi^n+\Delta t(R^{n+1}+R^n)\phi_1^{n+1}$ and goto next time step.
\end{itemize}

We observe that the above procedure only requires solving one linear equation with constant coefficients as in a semi-implicit scheme with stabilization. As for the energy stability, we have the following result easily.
\begin{theorem}\label{PSSAV-theorem4}
Given $R^0>0$, we have $R^n>0$ for all $n>0$, and the second-order PS-SAV scheme \eqref{PSSAV-e11} is unconditionally energy stable in the sense that
\begin{equation*}
R^{n+1}-R^{n}\leq\displaystyle-\frac{\Delta t}{M}(\frac{\phi^{n+1}-\phi^n}{\Delta t},\frac{\phi^{n+1}-\phi^n}{\Delta t})\leq0.
\end{equation*}
\end{theorem}
\section{The PS-SAV approach for $H^{-1}$ gradient flow}
The proposed positivity-preserving technique can also be used to solve $H^{-1}$ gradient flow. By setting $\mathcal{G}=-\Delta$ to transform the gradient flow \eqref{intro-e1} into the following $H^{-1}$ gradient flow:
\begin{equation}\label{H-1-gradient-flow}
   \begin{array}{l}
\displaystyle\frac{\partial \phi}{\partial t}=M\Delta\mu,\\
\displaystyle\mu=\epsilon^2A\phi+F'(\phi).
   \end{array}
\end{equation}

The $H^{-1}$ gradient flow model \eqref{H-1-gradient-flow} is mass preserving since
\begin{equation*}
\forall t\geq0,\quad \frac{d}{dt}\int_\Omega\phi d\mathbf{x}=\int_\Omega\frac{\partial \phi}{\partial t} d\mathbf{x}=0.
\end{equation*}

To construct PS-SAV scheme for the $H^{-1}$ gradient flow \eqref{H-1-gradient-flow}, we need to define the $H^{-1}_{per}$ inner product firstly. Suppose $f\in L^2_0(\Omega)=\{v\in L^2(\Omega)|(v,1)=0\}$, define $\mu_f\in H^2_{per}(\Omega)\cap L_0^2(\Omega)$ to be the unique solution to the following problem with periodic boundary condition:
\begin{equation}\label{PSSAV-H-1-e1}
-\Delta \mu_f=f\quad \text{in}\ \Omega.
\end{equation}
We then define $\mu_f:=(-\Delta)^{-1}f$, and for any $f,g\in L^2_0(\Omega)$, the $H^{-1}_{per}$ inner product and norm can be defined as follows:
\begin{equation}\label{PSSAV-H-1-e2}
(f,g)_{-1}=(\nabla \mu_f,\nabla \mu_g),\quad \|f\|_{-1}=\sqrt{(f,f)_{-1}}.
\end{equation}
It is easy to obtain the following identity:
\begin{equation}\label{PSSAV-H-1-e3}
(f,g)_{-1}=((-\Delta)^{-1}f,g)=(f,(-\Delta)^{-1}g)=(g,f)_{-1}.
\end{equation}

Given a same SAV $R(t)$ with \eqref{PSSAV-e2}, the corresponding derivative equation for $R$ will take the following formulation:
\begin{equation}\label{PSSAV-H-1-e4}
\displaystyle\frac{dR}{dt}=\frac{dE}{dt}=(\frac{\delta E}{\delta\phi},\frac{\partial \phi}{\partial t})=(\mu,\frac{\partial \phi}{\partial t})=\frac{1}{M}\left(-(-\Delta)^{-1}\frac{\partial \phi}{\partial t},\frac{\partial \phi}{\partial t}\right)=-\frac{1}{M}(\frac{\partial \phi}{\partial t},\frac{\partial \phi}{\partial t})_{-1}.
\end{equation}

Combining above equation \eqref{PSSAV-H-1-e4} with \eqref{H-1-gradient-flow}, we can reformulate the $H^{-1}$ gradient flow to the following equivalent system:
\begin{equation}\label{PSSAV-H-1-e5}
   \begin{array}{l}
\displaystyle\frac{\partial \phi}{\partial t}=M\Delta\mu,\\
\displaystyle\mu=\frac{R(t)}{E(\phi)+C}\left(\epsilon^2A\phi+F'(\phi)\right),\\
\displaystyle\frac{dR}{dt}=-\frac{1}{M}(\frac{\partial \phi}{\partial t},\frac{\partial \phi}{\partial t})_{-1}.
   \end{array}
\end{equation}
One can see the third equation in \eqref{PSSAV-H-1-e5} can keep the energy dissipation law.

Similar as the PS-SAV schemes for the $L^2$ gradient flow, we next consider the first-order PS-SAV scheme for the $H^{-1}$ gradient flow \eqref{H-1-gradient-flow}. The first-order PS-SAV scheme based on the backward Euler formulation for the $H^{-1}$ gradient flow \eqref{H-1-gradient-flow} can be given by:
\begin{equation}\label{PSSAV-H-1-e6}
   \begin{array}{l}
\displaystyle\frac{\phi^{n+1}-\phi^n}{\Delta t}=M\Delta\mu^{n+1},\\
\displaystyle\mu^{n+1}=s\epsilon^2(A\phi^{n+1}-A\phi^{n})+\frac{R^{n+1}}{E(\widehat{\phi}^n)+C}\left[\epsilon^2A \widehat{\phi}^{n}+F'(\widehat{\phi}^{n})\right],\\
\displaystyle \frac{R^{n+1}-R^n}{\Delta t}=\displaystyle-\frac{1}{M}(\frac{\phi^{n+1}-\phi^n}{\Delta t},\frac{\phi^{n+1}-\phi^n}{\Delta t})_{-1}.
   \end{array}
\end{equation}
From the first two equations in \eqref{PSSAV-H-1-e6}, we can rewrite \eqref{PSSAV-H-1-e6} equivalently as the following:
\begin{equation}\label{PSSAV-H-1-e7}
   \begin{array}{l}
\displaystyle2(E(\widehat{\phi}^{n+\frac{1}{2}})+C)(I-M\epsilon^2s\Delta t\Delta A)\frac{\phi^{n+1}-\phi^n}{\Delta t}=M(R^{n+1}+R^n)\left[\epsilon^2\Delta A \widehat{\phi}^{n+\frac{1}{2}}+\Delta F'(\widehat{\phi}^{n+\frac{1}{2}})\right],\\
\displaystyle \frac{M}{\Delta t}(R^{n+1}-R^n)=\displaystyle-(\frac{\phi^{n+1}-\phi^n}{\Delta t},\frac{\phi^{n+1}-\phi^n}{\Delta t})_{-1}.
   \end{array}
\end{equation}
Setting $\phi^{n+1}=\phi^{n}+\Delta t(R^{n+1}+R^n)\phi_1^{n+1}$, we also find that $\phi_1^{n+1}$ is the solution of the following linear equation with constant coefficients:
\begin{equation}\label{PSSAV-H-1-e8}
\displaystyle2(E(\widehat{\phi}^{n+\frac{1}{2}})+C)(I-Ms\epsilon^2\Delta t\Delta A)\phi_1^{n+1}=M\left[\epsilon^2\Delta A \widehat{\phi}^{n+\frac{1}{2}}+\Delta F'(\widehat{\phi}^{n+\frac{1}{2}})\right].
\end{equation}
Once $\phi_1^{n+1}$ is known, to compute $R^{n+1}$, we need to solve the following quadratic equation
\begin{equation}\label{PSSAV-H-1-e9}
\displaystyle a(R^{n+1})^2+bR^{n+1}+C=0,
\end{equation}
where the coefficients $a$, $b$ and $c$ satisfy
\begin{equation}\label{PSSAV-H-1-e10}
a=(\phi_1^{n+1},\phi_1^{n+1})_{-1},\quad b=\frac{M}{\Delta t},\quad c=-\frac{M}{\Delta t} R^n.
\end{equation}   
If $\phi_1^{n+1}=0$, we can immediately obtain $\phi^{n+1}=\phi^n$ and $R^{n+1}=R^n$. If $\phi_1^{n+1}\neq0$, we obtain $(\phi_1^{n+1},\phi_1^{n+1})_{-1}>0$. Then the above equation \eqref{PSSAV-H-1-e9} is a quadratic equation with one variable for $R^{n+1}$. 

\begin{theorem}\label{PSSAV-theorem5}
The quadratic equation with one variable for $R^{n+1}$ \eqref{PSSAV-H-1-e9} has and only one positive solution:
\begin{equation}\label{PSSAV-H-1-e11}
R^{n+1}=\frac{-\frac{M}{\Delta t}+\sqrt{\frac{M^2}{\Delta t^2}+4\frac{M}{\Delta t}R^n(\phi_1^{n+1},\phi_1^{n+1})_{-1}}}{2(\phi_1^{n+1},\phi_1^{n+1})_{-1}}>0.
\end{equation}
\end{theorem}
As for the energy stability, we have the following result easily.
\begin{theorem}\label{PSSAV-theorem6}
Given $R^0>0$, we have $R^n>0$ for all $n>0$, and the first-order PS-AV scheme \eqref{PSSAV-H-1-e6} is unconditionally energy stable in the sense that
\begin{equation*}
R^{n+1}-R^{n}=\displaystyle-\frac{\Delta t}{M}(\frac{\phi^{n+1}-\phi^n}{\Delta t},\frac{\phi^{n+1}-\phi^n}{\Delta t})_{-1}\leq0.
\end{equation*}
\end{theorem}
\begin{remark}
The first-order PS-SAV scheme \eqref{PSSAV-H-1-e6} also only requires solving one linear equation with constant coefficients as in a semi-implicit scheme with stabilization. In addition, it may add some additional small computation cost to obtain $(\phi_1^{n+1},\phi_1^{n+1})_{-1}$.
\end{remark}
\section{An energy optimization technique}
Noting that the proposed PS-SAV schemes are unconditionally energy stable with a modified energy, we give an energy optimization technique to make the modified energy to be close to the original energy. 

At each time step, after obtaining $R^{n+1}$, we calibrate it by using the following equation:
\begin{equation}\label{energy-opt-e1}
R^{n+1}=\min\{R^n,E(\phi^{n+1})+C\}.
\end{equation}
The above correction technique will not affect the energy dissipation law and the convergence rates.

We take the first-order PS-SAV scheme \eqref{PSSAV-e3} for the $L^2$ gradient flow as an example. 
\begin{theorem}\label{energy-opt-theorem1}
The first-order PS-SAV scheme \eqref{PSSAV-e3} with correction technique \eqref{energy-opt-e1} is unconditionally energy stable in the sense that
\begin{equation}\label{energy-opt-e2}
\mathcal{E}^{n+1}-\mathcal{E}^n\leq0,
\end{equation}
where $\mathcal{E}^n=R^n-C$ is the modified energy.

We further have the following original energy dissipation law:
\begin{equation*}
\aligned
E(\phi^{n+1})\leq E(\phi^{n}),
\endaligned
\end{equation*}
under the condition of $E(\phi^{n+1})+C\leq R^n$. Here $E(\phi^{n})=\frac{\epsilon^2}{2}(A\phi^n,\phi^n)+(F(\phi^n),1)$ is the original energy.
\end{theorem}
\begin{proof}
From the correction equation \eqref{energy-opt-e1}, we get $R^{n+1}\leq R^n$, then one can immediately obtain
\begin{equation*}
\mathcal{E}^{n+1}-\mathcal{E}^n\leq0,
\end{equation*}

We can also obtain $R^{n+1}\leq E(\phi^{n+1})+C$ from \eqref{energy-opt-e1}, then the following inequality is satisfied:
\begin{equation*}
\mathcal{E}^{n+1}=R^{n+1}-C\leq E(\phi^{n+1}),
\end{equation*}
which means 
\begin{equation}\label{energy-opt-e3}
\mathcal{E}(\phi^{n})\leq E(\phi^{n}), \quad \forall n\geq0.
\end{equation}

If $E(\phi^{n+1})+C\leq R^n$, we get $R^{n+1}=\min\{R^n,E(\phi^{n+1})+C\}=E(\phi^{n+1})+C$, then the following equation will hold:
\begin{equation}\label{energy-opt-e4}
\aligned
\mathcal{E}(\phi^{n+1})=R^{n+1}-C=E(\phi^{n+1}).
\endaligned
\end{equation}
Combining the inequality \eqref{energy-opt-e2} with \eqref{energy-opt-e2} and \eqref{energy-opt-e4}, we obtain
\begin{equation}\label{energy-opt-e5}
\aligned
E(\phi^{n+1})=\mathcal{E}(\phi^{n+1})\leq\mathcal{E}(\phi^{n})\leq E(\phi^{n}),
\endaligned
\end{equation}
which means the first-order PS-SAV scheme \eqref{PSSAV-e3} with correction technique \eqref{energy-opt-e1} is unconditionally energy stable with original energy under the condition of $E(\phi^{n+1})+C\leq R^n$.
\end{proof}

\section{Examples and discussion}
In this section, we consider some numerical examples to illustrate the simplicity and efficiency of our proposed method. In all considered examples, we consider the periodic boundary conditions and use a Fourier spectral method in space.

\begin{example}\label{ex:AC}
	\rm
	The following Allen-Cahn equation is under our consideration,
	\begin{equation}\label{eq:Allen-Cahn}
		\frac{\partial \phi}{\partial t}=M\left(\alpha_{0} \Delta \phi+\left(1-\phi^{2}\right) \phi\right),
	\end{equation}
	subject to periodic boundary conditions.
	
	{\em Case A.} We give the exact solution
	\begin{equation} \label{eq:AC-CH-exact-solution-example}
		\phi(x, y, t)=\exp (\sin (\pi x) \sin (\pi y)) \sin (t),
	\end{equation}
	by introducing an external force $f$ into \eqref{eq:Allen-Cahn} in the domain $\Omega=(0, 2)^{2}$. 
	We set the values of the parameters $M$ and $\alpha_{0}$ to $1$ and $0.01^2$, respectively. 
	To ensure that the spatial discretization error is much smaller than the time discretization error, we adopt $128^2$ Fourier modes for space discretization. 
	
	In Table \ref{table:AC-1st} and Table \ref{table:AC-CN}, we present the $L^2$-norm error convergence rate for SAV, GSAV and PS-SAV approaches at $T = 1$ obtained using first-order and Crank-Nicolson scheme, respectively. 
	We have observed that the expected convergence rates are achieved for all cases. Furthermore, in this example, the errors for the first-order scheme satisfy the following order: SAV $\approx$ PS-SAV $<$ GSAV. On the other hand, for the Crank-Nicolson scheme, the errors follow the order: PS-SAV $\leq$ SAV $<$ GSAV. These findings indicate that the proposed PS-SAV method performs equally or slightly better than the SAV method, with these two methods slightly outperforming the GSAV method in terms of error reduction. 
	
	\begin{table}[!h] 
		\centering
		\caption{Example \ref{ex:AC} Convergence test for Allen-Cahn equation using the first-order scheme by different approaches.}
		\label{table:AC-1st}
			\begin{tabular}{||c||cc||cc||cc||}
				\hline
				& \multicolumn{2}{c||}{SAV} & \multicolumn{2}{c||}{GSAV} & \multicolumn{2}{c||}{PS-SAV} \\
				\hline
				$\Delta t$ & $\|e_{\phi}\|_{L^{2}}$ & Rate & $\|e_{\phi}\|_{L^{2}}$ & Rate & $\|e_{\phi}\|_{L^{2}}$ & Rate \\
				\hline
				1.00E-2  & 1.19E-02 &  --    & 2.53E-02   & --    & 1.35E-02 &  --   \\
				5.00E-3  & 5.93E-03 &  1.01  & 1.21E-02   & 1.06  & 6.74E-03 &  1.00 \\
				2.50E-3  & 2.96E-03 &  1.00  & 5.94E-03   & 1.03  & 3.37E-03 &  1.00 \\
				1.25E-3  & 1.48E-03 &  1.00  & 2.94E-03   & 1.01  & 1.68E-03 &  1.00 \\
				6.25E-4  & 7.38E-04 &  1.00  & 1.46E-03   & 1.01  & 8.41E-04 &  1.00 \\
				\hline
			\end{tabular}
		\end{table}
	
		\begin{table}[!h] 
		\centering
		\caption{Example \ref{ex:AC} Convergence test for Allen-Cahn equation using the Crank-Nicolson scheme by different approaches.}
		\label{table:AC-CN}
			\begin{tabular}{||c||cc||cc||cc||}
				\hline
				& \multicolumn{2}{c||}{SAV} & \multicolumn{2}{c||}{GSAV} & \multicolumn{2}{c||}{PS-SAV} \\
				\hline
				$\Delta t$ & $\|e_{\phi}\|_{L^{2}}$ & Rate & $\|e_{\phi}\|_{L^{2}}$ & Rate & $\|e_{\phi}\|_{L^{2}}$ & Rate \\
				\hline
				1.00E-2  & 5.48E-05 &  --    & 8.23E-04   & --    & 5.47E-05 &  --   \\
				5.00E-3  & 1.37E-05 &  2.00  & 1.98E-04   & 2.05  & 1.37E-05 &  2.00 \\
				2.50E-3  & 3.44E-06 &  2.00  & 4.88E-05   & 2.03  & 3.43E-06 &  2.00 \\
				1.25E-3  & 8.61E-07 &  2.00  & 1.21E-05   & 2.01  & 8.59E-07 &  2.00 \\
				6.25E-4  & 2.15E-07 &  2.00  & 3.01E-06   & 2.01  & 2.15E-07 &  2.00 \\
				\hline
			\end{tabular}
		\end{table}
	
	{\em Case B.}  We choose the initial condition as            
	\begin{equation}
		\begin{aligned}
			& \phi(x, y)=\tanh \frac{1.5+1.2 \cos (6 \theta)-2 \pi r}{\sqrt{2\alpha}}, \\ 
			& \theta=\arctan \frac{y-0.5 L_{y}}{x-0.5 L_{x}}, \quad r=\sqrt{\left(x-\frac{L_{x}}{2}\right)^{2}+\left(y-\frac{L_{y}}{2}\right)^{2}},
		\end{aligned}
	\end{equation}
		where $(\theta,r)$ are the polar coordinates of $(x, y)$. 
	We set $\Omega=[0, L_x]\times[0, L_{y}]$ with $L_x=L_y=1$ and the other parameters are $\alpha_{0}=0.01^2, M=1$ and $128^2$ Fourier modes. 
	We use the results of the semi-implicit/first-order scheme with $\Delta t = 1E-5$ as the reference solution. 
	The $L^{2}$-norm error of four schemes at $T=200$ with different time steps are shown in Table \ref{table:comparison-dt-schemes}. 
	In this particular case, we observed that the errors of the SAV, GSAV, and PS-SAV approaches follow the order: PS-SAV $<$ SAV $<$ GSAV. However, upon applying the energy optimization technique, the error of R-PS-SAV approach is slightly larger than that of PS-SAV approach, but still smaller than the errors of SAV and GSAV approaches. 
	In Fig.\,\ref{Fig:AC-star-shape-energy}, we provide a comparison of the energy (first), energy error (second), and error of $\xi^{n+1}$ (third) for the SAV, GSAV, PS-SAV, and R-PS-SAV approaches. These results are obtained using the first-order scheme with a time step size of $\Delta t = 1E-3$. 
	We can observe that for the majority of the time, the error in modified energy and the error in $\xi^{n+1}$ follow the following order: R-PS-SAV $<$ SAV $<$ PS-SAV $<$ GSAV.
	
	\begin{table}[htbp] 
		\centering
		\caption{Example \ref{ex:AC} (Case B). A comparison of  $L^2$-error obtained by four approaches based on first-order scheme for Allen-Cahn equation at $T=200$ with various time steps.}
		\label{table:comparison-dt-schemes}
		\begin{tabular}{||c||c||c||c||c||}
			\hline
			$\Delta t$ & SAV & GSAV & PS-SAV & R-PS-SAV \\
			\hline
			1.00E-1  & 1.75E-03   & 3.31E-03   & 4.88E-04  & 9.31E-04 \\
			5.00E-2  & 8.77E-04   & 1.83E-03   & 2.48E-04  & 4.66E-04 \\
			1.00E-2  & 1.76E-04   & 4.07E-04   & 5.03E-05  & 9.32E-05 \\
			5.00E-3  & 8.77E-05   & 2.07E-04   & 2.51E-05  & 4.65E-05 \\
			1.00E-3  & 1.74E-05   & 4.18E-05   & 5.04E-06  & 9.16E-06 \\
			\hline
		\end{tabular}
	\end{table}	
	
	\begin{figure}[htbp]
		\centering
		\begin{minipage}{0.32\textwidth}
			\centering
			\includegraphics[width=5.3cm]{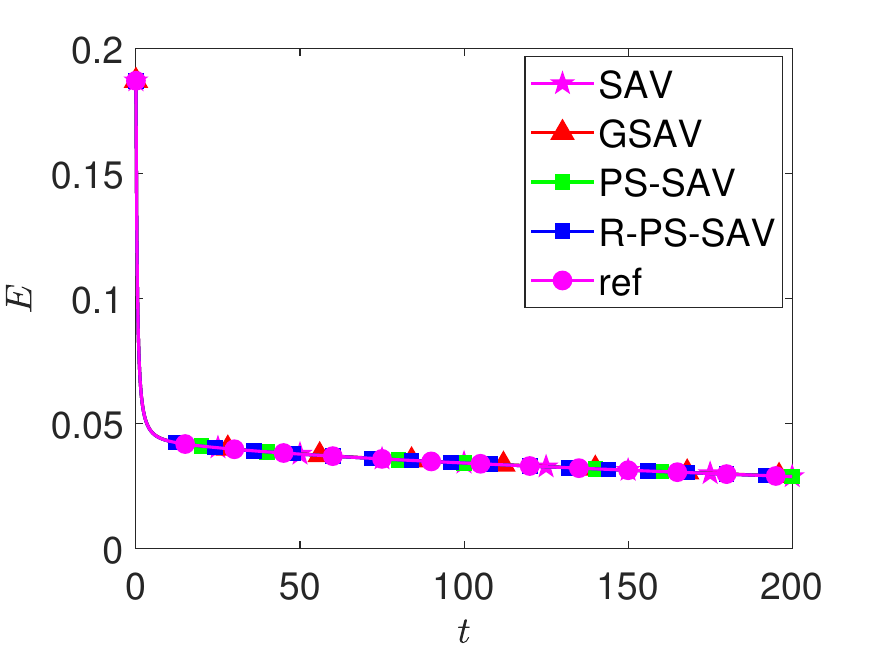}
		\end{minipage}
		\begin{minipage}{0.32\textwidth}
			\centering
			\includegraphics[width=5.3cm]{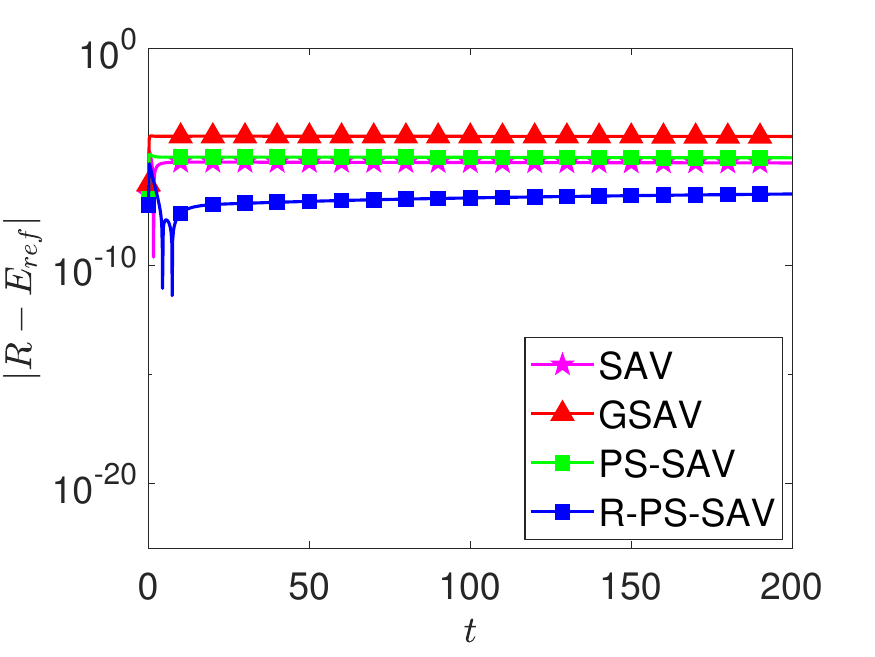} 
		\end{minipage}
		\begin{minipage}{0.32\textwidth}
			\centering
			\includegraphics[width=5.3cm]{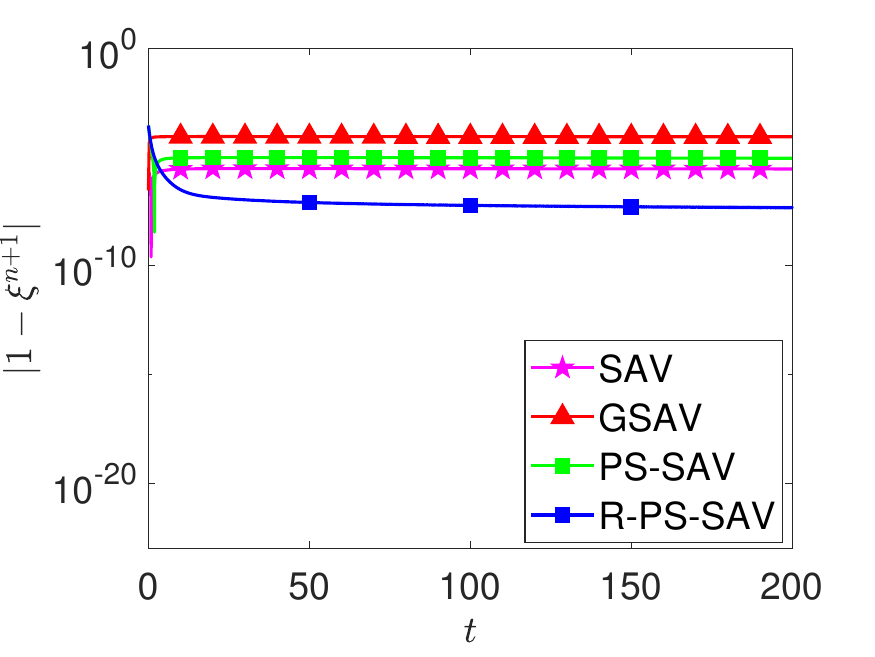}
		\end{minipage}
		\caption{Example \ref{ex:AC} (Case B). Allen-Cahn equation: a comparison of energy (first), errors of energy (second) and errors of $\xi^{n+1}$ (third) obtained by four approaches with $\Delta t = 1E-3$ based on first-order scheme.}
		\label{Fig:AC-star-shape-energy}
	\end{figure}  
\end{example}

\begin{example}\label{ex:CH}
	\rm
	We consider Cahn-Hilliard equation
	\begin{equation}\label{eq:Cahn-Hilliard}
		\frac{\partial \phi}{\partial t}=-M \Delta\left(\alpha_0 \Delta \phi+\frac{1}{\epsilon^2}\left(1-\phi^{2}\right) \phi\right).
	\end{equation}
	
	{\em Case A.} We give the exact solution
	\begin{equation} 
		\phi(x, y, t)=\cos (\pi x) \cos (\pi y) \sin (t),
	\end{equation}
	by introducing an external force $f$ into \eqref{eq:Cahn-Hilliard} in the domain $\Omega=(0, 2)^{2}$. 
	We set the values of the parameters $\alpha_0=0.04$, $M=0.005$, and $\epsilon=1$. 
	To ensure that the spatial discretization error is much smaller than the time discretization error, we adopt $128^2$ Fourier modes for space discretization. 
	
	In Table \ref{table:CH-1st} and Table \ref{table:CH-CN}, we present the $L^2$-norm error convergence rate for SAV, GSAV and PS-SAV approaches at $T = 1$ obtained using first-order and Crank-Nicolson scheme, respectively. We can observed that the expected convergence rates are obtained for all cases. 
	
		\begin{table}[!h] 
		\centering
		\caption{Example \ref{ex:CH} (Case A). Convergence test for Cahn-Hilliard equation using the first-order scheme by different approaches.}
		\label{table:CH-1st}
			\begin{tabular}{||c||cc||cc||cc||}
				\hline
				& \multicolumn{2}{c||}{SAV} & \multicolumn{2}{c||}{GSAV} & \multicolumn{2}{c||}{PS-SAV} \\
				\hline
				$\Delta t$ & $\|e_{\phi}\|_{L^{2}}$ & Rate & $\|e_{\phi}\|_{L^{2}}$ & Rate & $\|e_{\phi}\|_{L^{2}}$ & Rate \\
				\hline
				1.00E-2  & 2.85E-03 &  --    & 2.83E-03   & --    & 2.24E-03 &  --   \\
				5.00E-3  & 1.42E-03 &  1.00  & 1.41E-03   & 1.00  & 1.12E-03 &  1.00 \\
				2.50E-3  & 7.12E-04 &  1.00  & 7.06E-04   & 1.00  & 5.58E-04 &  1.00 \\
				1.25E-3  & 3.56E-04 &  1.00  & 3.53E-04   & 1.00  & 2.79E-04 &  1.00 \\
				6.25E-4  & 1.78E-04 &  1.00  & 1.76E-04   & 1.00  & 1.39E-04 &  1.00 \\
				\hline
			\end{tabular}
		\end{table}
		
		\begin{table}[!h] 
			\centering
			\caption{Example \ref{ex:CH} (Case A). Convergence test for Chan-Hilliard equation using the Crank-Nicolson scheme by different approaches.}
			\label{table:CH-CN}
				\begin{tabular}{||c||cc||cc||cc||}
					\hline
					& \multicolumn{2}{c||}{SAV} & \multicolumn{2}{c||}{GSAV} & \multicolumn{2}{c||}{PS-SAV} \\
					\hline
					$\Delta t$ & $\|e_{\phi}\|_{L^{2}}$ & Rate & $\|e_{\phi}\|_{L^{2}}$ & Rate & $\|e_{\phi}\|_{L^{2}}$ & Rate \\
					\hline
					1.00E-2  & 4.96E-06 &  --    & 4.89E-06   & --    & 3.94E-06 &  --   \\
					5.00E-3  & 1.25E-06 &  1.99  & 1.23E-06   & 1.99  & 9.91E-07 &  1.99 \\
					2.50E-3  & 3.12E-07 &  2.00  & 3.08E-07   & 2.00  & 2.48E-07 &  2.00 \\
					1.25E-3  & 7.82E-08 &  2.00  & 7.71E-08   & 2.00  & 6.22E-08 &  2.00 \\
					6.25E-4  & 1.96E-08 &  2.00  & 1.93E-08   & 2.00  & 1.56E-08 &  2.00 \\
					\hline
				\end{tabular}
			\end{table}

	{\em Case B.}  As the initial condition, we consider a rectangular arrangement of $19 \times 19$ circles 
	\begin{equation}
		\phi_{0}(\boldsymbol{x}, t)=360-\sum_{m=1}^{19} \sum_{n=1}^{19} \tanh\left(\frac{\sqrt{\left(x-x_{m}\right)^{2}+\left(y-y_{n}\right)^{2}}-r_{0}}{\sqrt{2} \epsilon}\right),
	\end{equation}
	where $r_0 = 0.085, x_{m} = 0.2\times m, y_{n} = 0.2\times n$ for $m, n = 1, 2, \cdots, 19$. 
	For our simulations, we use a computational domain of $[0,4]^2$. The parameters $M$, $\alpha_0$, and $\epsilon$ are set to $1E-6$, $1.6032$, and $0.0079$, respectively. We adopt a spatial discretization scheme using $512^2$ Fourier modes. 
	The PS-SAV approach proposed in this study guarantee the unconditional positivity of the computed $R(t)$ values, regardless of the time step size. Fig.\,\ref{Fig:CH-r-R-xi} first and second subfigures illustrate the time history of the auxiliary variable $r(t)$ computed using the SAV and the auxiliary variable $R(t)$ obtained by the PS-SAV approach, both with a time step size of $\Delta t = 0.5$.
	In the PS-SAV approach, $R(t)$ is computed using a dynamic equation derived from the relation $R(t) = E(\phi)+C>0$, ensuring the positivity of $R(t)$. On the other hand, in the SAV method, the auxiliary variable $r(t)$ is computed using a dynamic equation based on the relation $r(t) = \sqrt{E_1(\phi)+C}$. However, SAV lacks the property of guaranteeing the positivity of the auxiliary variable, and as shown in first subfigure of Fig.\,\ref{Fig:CH-r-R-xi}, the computed $r(t)$ values can take negative values.  
	The first two subfigures of Fig.\,\ref{Fig:CH-circles-comparison} show the snapshots of field function at $T=100$ using SAV and PS-SAV approaches with Euler scheme and a time step size $\Delta t =0.1$. The discrepancy between the two results suggests that the PS-SAV approach yields more accurate results compared to the SAV approach. The last two subfigures of Fig.\,\ref{Fig:CH-circles-comparison} show the snapshots of field function at $T=100$ using SAV and PS-SAV approaches with Euler scheme and a time step size $\Delta t =1E-3$. The results obtained from both figures are consistent with each other.
	
	\begin{figure}[htbp]
		\centering
		\begin{minipage}{0.32\textwidth}
			\centering
			\includegraphics[width=5.3cm]{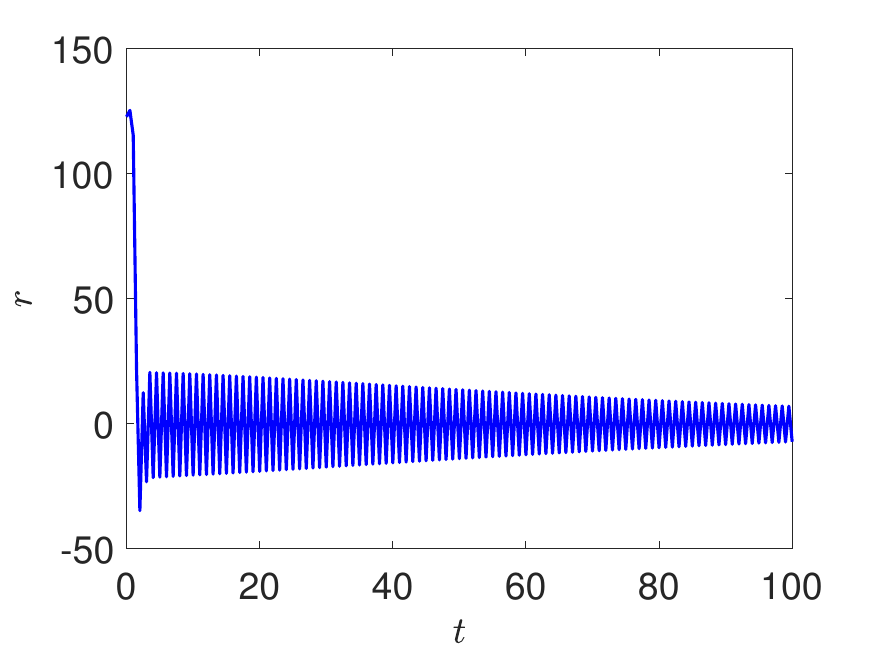}
		\end{minipage}
		\begin{minipage}{0.32\textwidth}
			\centering
			\includegraphics[width=5.3cm]{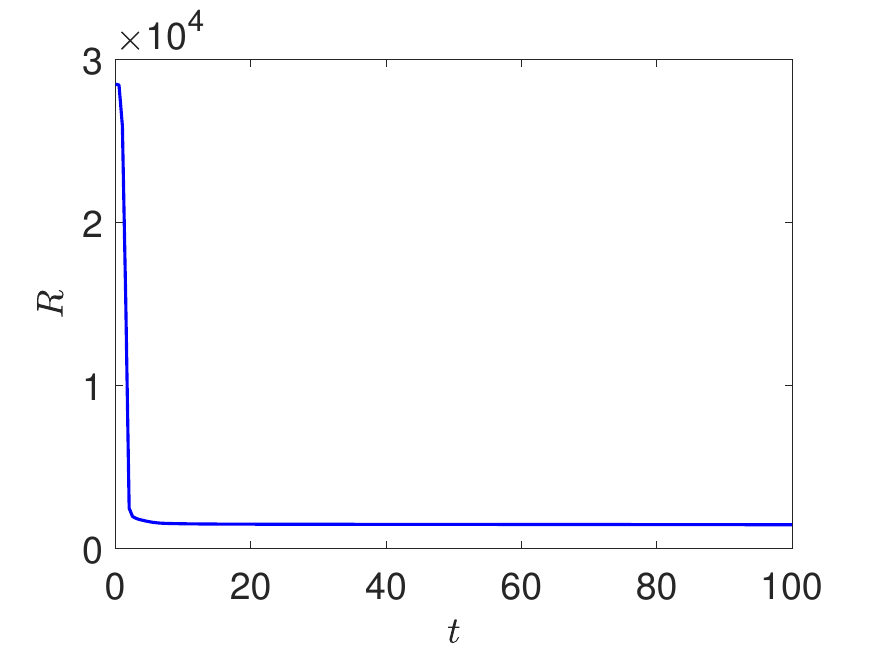}
		\end{minipage}
		\begin{minipage}{0.32\textwidth}
			\centering
			\includegraphics[width=5.3cm]{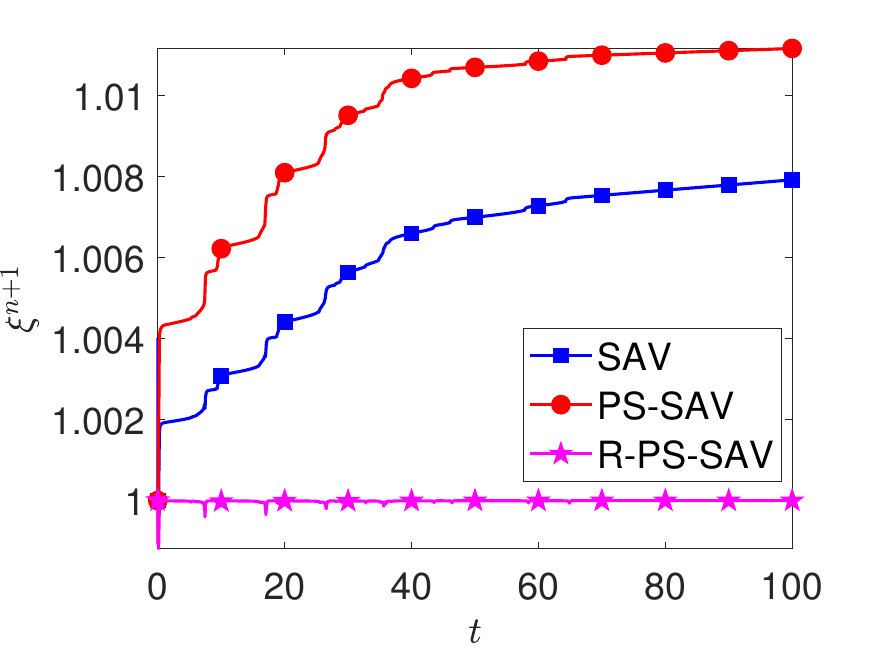}
		\end{minipage}
			\caption{Example \ref{ex:CH} (Case B). The history of $r$ obtained by SAV/Crank-Nicolson scheme (first) and the history of $R$ obtained by PS-SAV/Crank-Nicolson scheme (second) with $\Delta t = 0.5$. Third subfigure is the history of $\xi$ for three schemes with $\Delta t = 1E-3$.}
		\label{Fig:CH-r-R-xi}
	\end{figure} 

	\begin{figure}[htbp]
		\centering
		\begin{minipage}{0.24\textwidth}
			\centering
			\includegraphics[width=5.0cm]{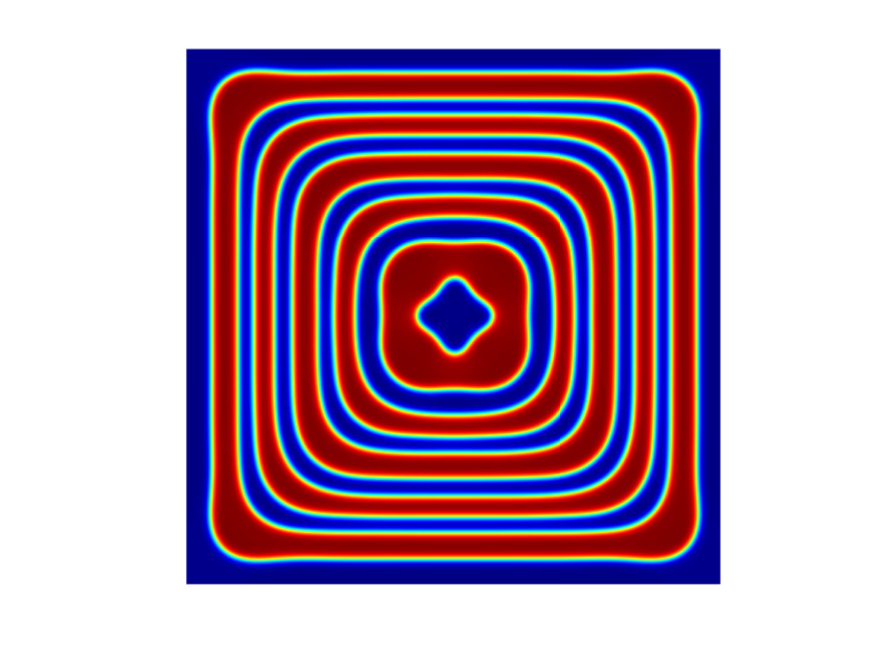}
		\end{minipage}
		\begin{minipage}{0.24\textwidth}
			\centering
			\includegraphics[width=5.0cm]{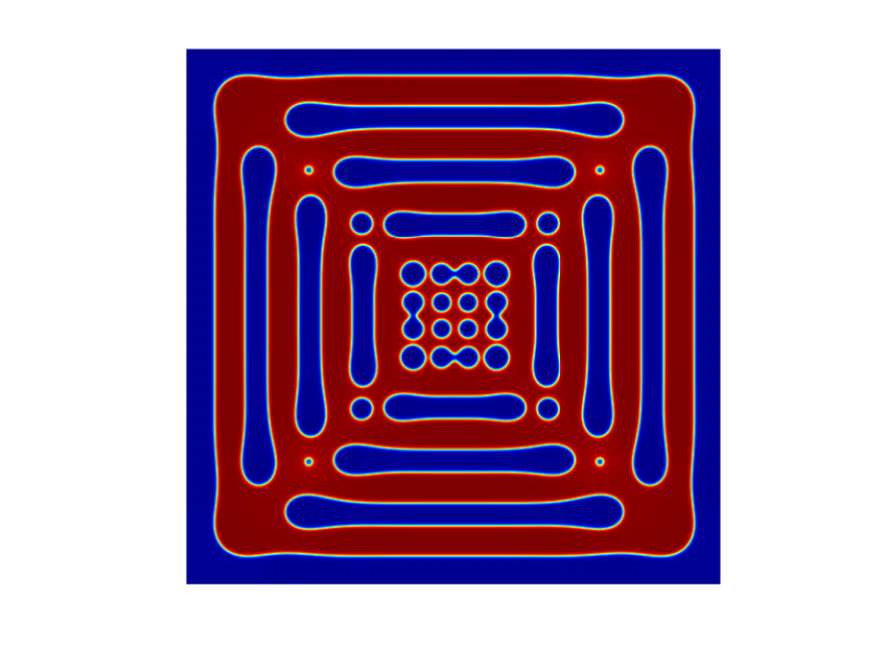}
		\end{minipage}
		\begin{minipage}{0.24\textwidth}
			\centering
			\includegraphics[width=5.0cm]{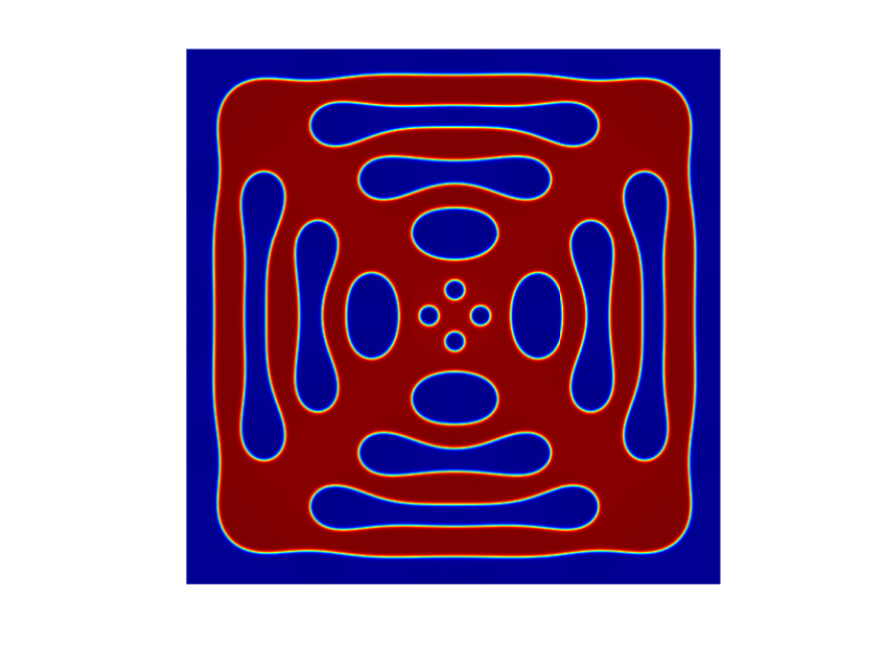}
		\end{minipage}
		\begin{minipage}{0.24\textwidth}
			\centering
			\includegraphics[width=5.0cm]{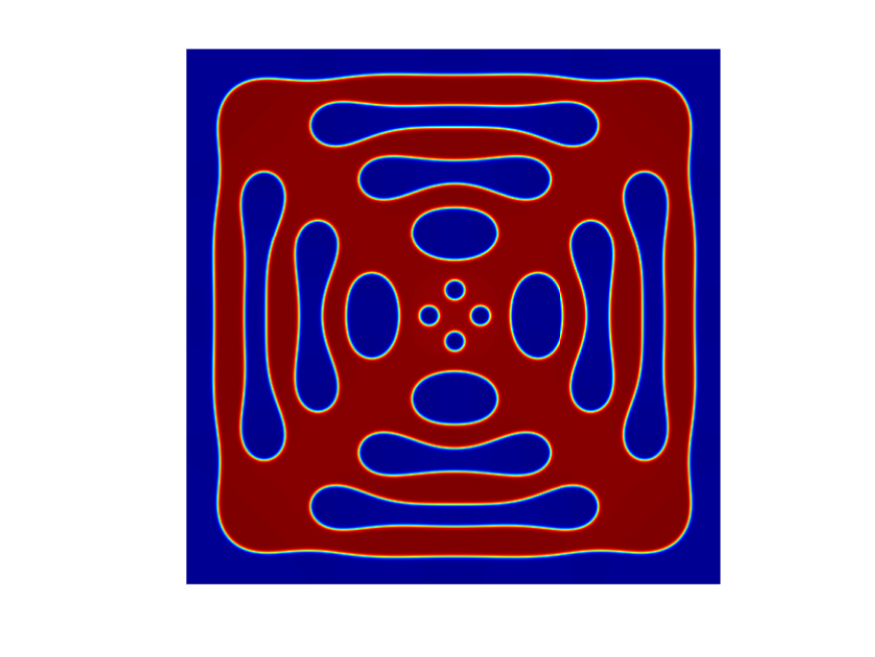}
		\end{minipage}
		\caption{Example \ref{ex:CH} (Case B). The snapshots of the field function at $T=100$ computed using different schemes and time step sizes: SAV/first-order scheme with $\Delta t =0.1$ (first); PS-SAV/first-order scheme with $\Delta t =0.1$ (second); SAV/first-order scheme with $\Delta t =0.001$ (third); PS-SAV/first-order scheme with $\Delta t =0.001$ (fourth).}
		\label{Fig:CH-circles-comparison}
	\end{figure} 
	
\end{example}

\begin{example}\label{ex:MBE}
	\rm 
	We consider the thin film epitaxy growth model. Let $\phi(\mathbf{x)}: \Omega \rightarrow \mathbf{R}$ represents the height of the thin film. The total free energy can be expressed as:
	
	\begin{equation}
		E(\phi)=\int_{\Omega}\left(F(\nabla \phi)+\frac{\epsilon^{2}}{2}(\Delta \phi)^{2}\right) d \mathbf{x}.
	\end{equation}
	
	Here, $F(\mathbf{y})$ is a smooth function, and $\epsilon$ is the gradient energy coefficient. The first term $\int_{\Omega} F(\nabla \phi) d \mathbf{x}$ represents a continuum description of the Ehrlich-Schwoedel effect, while the second term $\int_{\Omega} \frac{\epsilon^{2}}{2}(\Delta \phi)^{2} d \mathbf{x}$ represents the surface diffusion effect.

Two common choices for the nonlinear potential $F(\nabla \phi)$ are frequently employed.

(i) Double well potential for the model with slope selection:
\begin{equation*}
	F(\nabla \phi)=\frac{1}{4}\left(|\nabla \phi|^{2}-1\right)^{2}.
\end{equation*}

(ii) Logarithmic potential for the model without slope selection:
\begin{equation*}
	F(\nabla \phi)=-\frac{1}{2} \ln \left(1+|\nabla \phi|^2\right).
\end{equation*}

The evolution equation governing the height function $\phi$ is governed by the gradient flow, given by:
\begin{equation}\label{eq:MBE}
	\phi_t=-M\left(\epsilon^2 \Delta^2 \phi+f(\nabla \phi)\right),
\end{equation}
where $M$ is the mobility constant, and 
\begin{equation*}
	f(\nabla \phi)=-\nabla \cdot F^{\prime}(\nabla \phi)= \begin{cases}\nabla \cdot\left(\left(1-|\nabla \phi|^2\right) \nabla \phi\right), & \text { Model with slope selection, } \\ \nabla \cdot\left(\frac{\nabla \phi}{1+|\nabla \phi|^2}\right), & \text { Model without slope selection. }\end{cases}
\end{equation*}
The energy dissipation property for the aforementioned two models can be obtained by taking the $L^2$ inner product of \eqref{eq:MBE} with $\phi_{t}$ and applying integration by parts
\begin{equation*}
	\frac{d}{d t} E(\phi)=-\frac{1}{M}\left\|\phi_t\right\|^2 \leq 0.
\end{equation*}

To simulate the coarsening dynamics, we select a random initial condition ranging from $-0.001$ to $0.001$. The parameters are as follows:
\begin{equation*}
	\epsilon=0.03, M=1.
\end{equation*}
The computational domain is $\Omega=[0, 12.8)^2$, and we utilize $512^2$ Fourier modes for spatial discretization. In Fig.\,\ref{Fig:MBE-with-slope-2D-PS-SAV} and Fig.\,\ref{Fig:MBE-without-slope-2D-PS-SAV}, snapshots of the numerical solutions for the height function $\phi$ and its Laplacian $\Delta \phi$ at different times are presented for both models, respectively.

In the left subplot of Fig.\,\ref{Fig:MBE-energy}, the evolution of energy for the model with slope selection is plotted. It can be observed that the energy decays following a $t^{-\frac{1}{3}}$ trend. In the right subplot of Fig.\,\ref{Fig:MBE-energy}, the evolution of energy for the model without slope selection is depicted. It is notable that the energy decays logarithmically with respect to $-\log_{10}(t)$. These results are consistent with the findings reported in \cite{cheng2019highly}.

		\begin{figure}[htbp]
			\centering
			\subfigure[$t = 0$]{
			\begin{minipage}{0.23\textwidth}
				\centering
				\includegraphics[width=4.5cm]{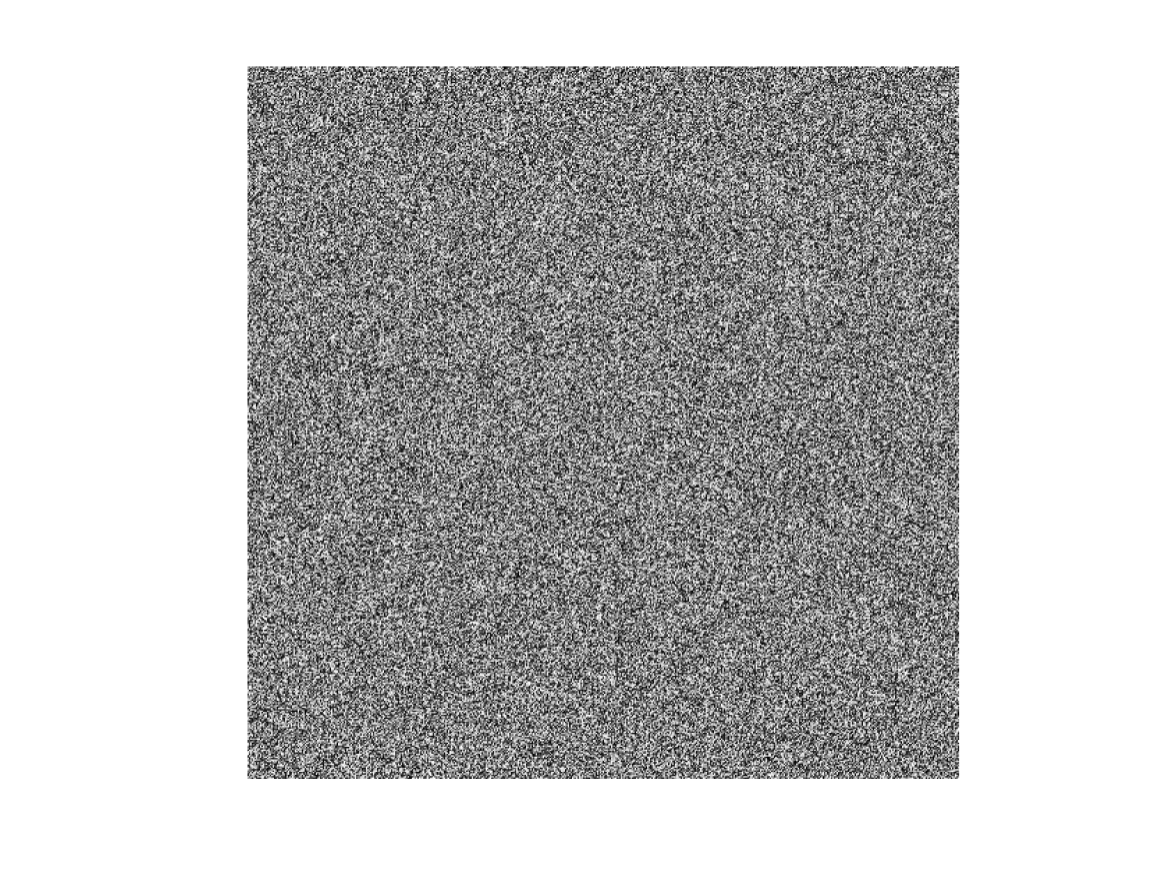}
			\end{minipage}
			\begin{minipage}{0.23\textwidth}
				\centering
				\includegraphics[width=4.5cm]{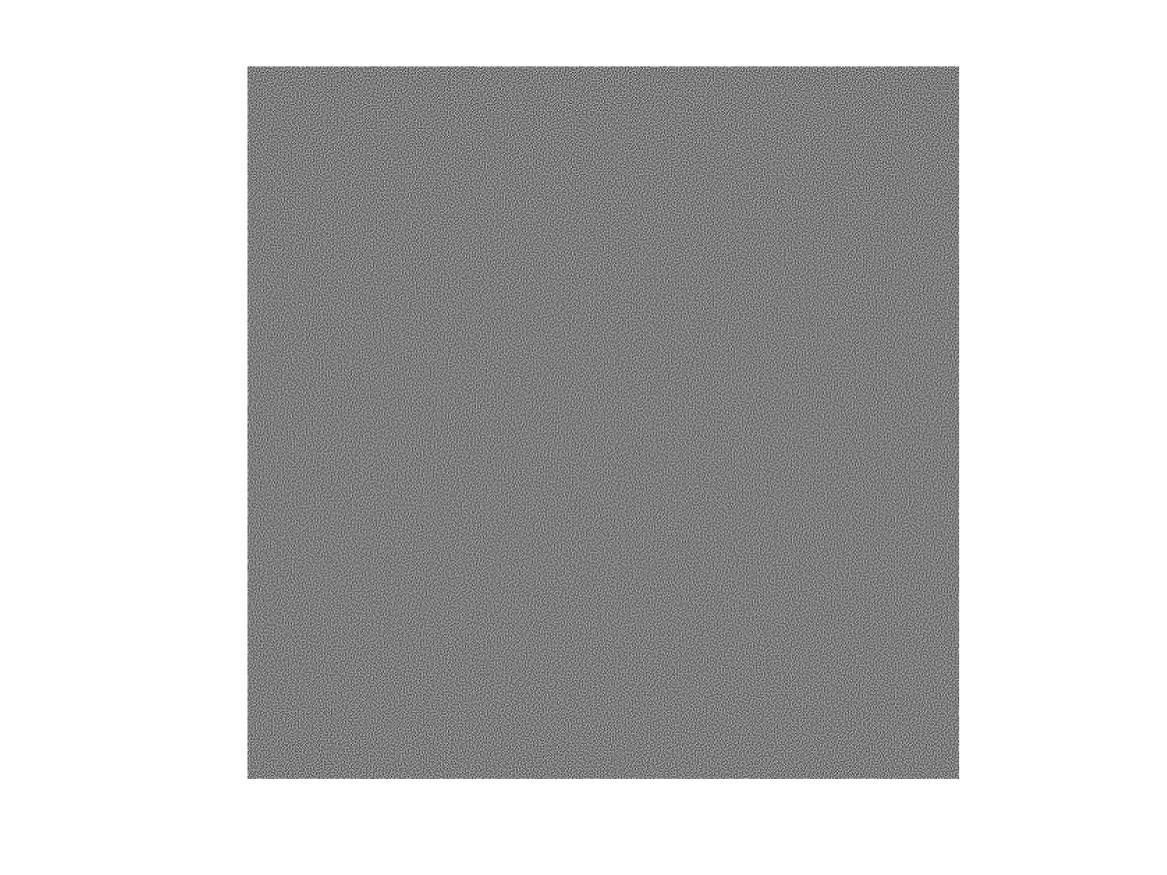}
			\end{minipage}
		}
	\subfigure[$t = 1$]{
			\begin{minipage}{0.23\textwidth}
				\centering
				\includegraphics[width=4.5cm]{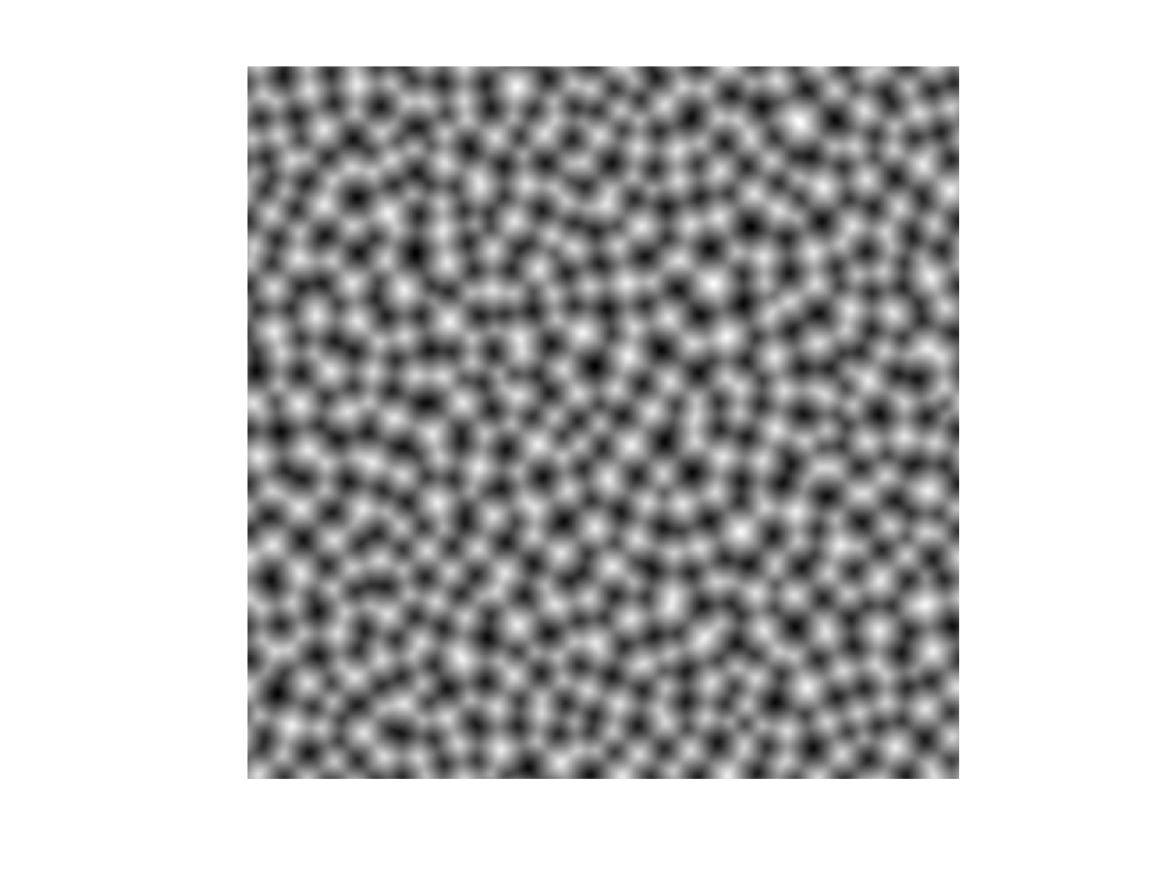}
			\end{minipage}
			\begin{minipage}{0.23\textwidth}
				\centering
				\includegraphics[width=4.5cm]{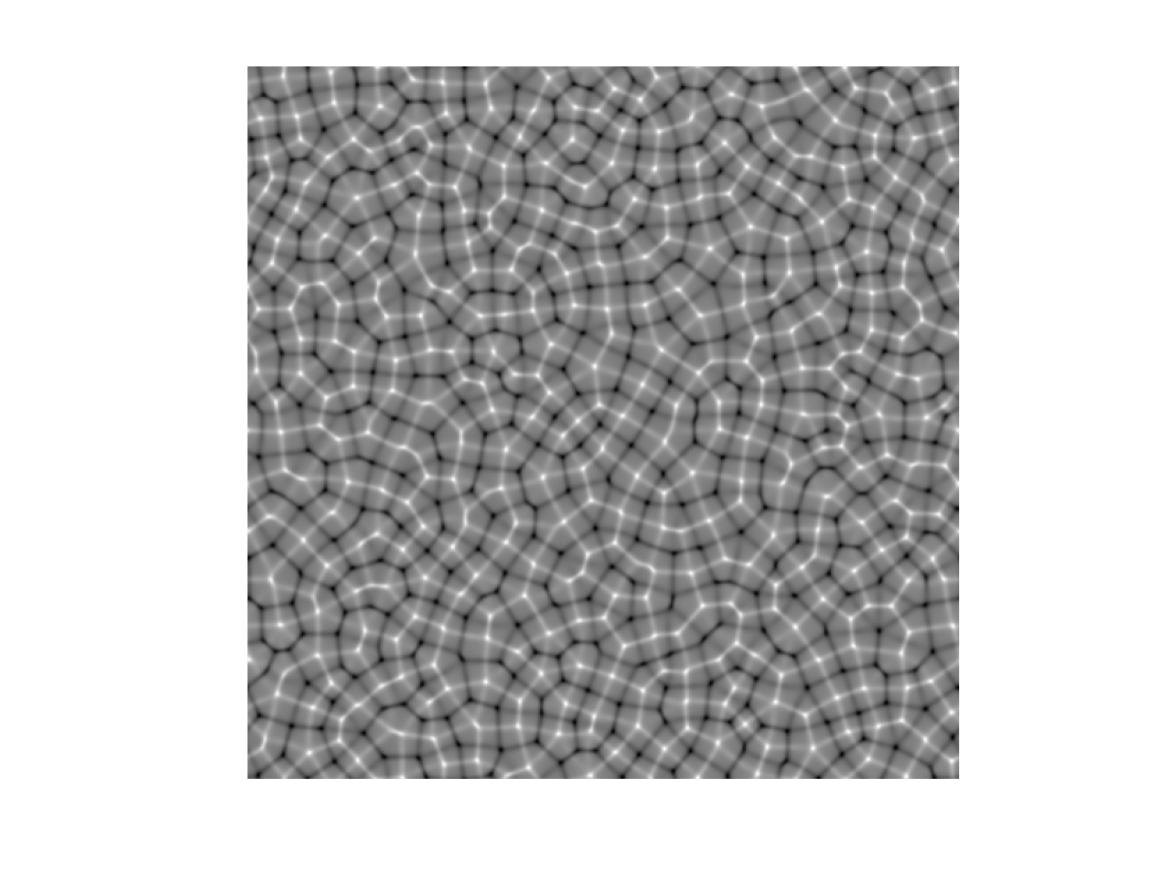}
			\end{minipage}
		}
		\subfigure[$t =10$]{
			\begin{minipage}{0.23\textwidth}
				\centering
				\includegraphics[width=4.5cm]{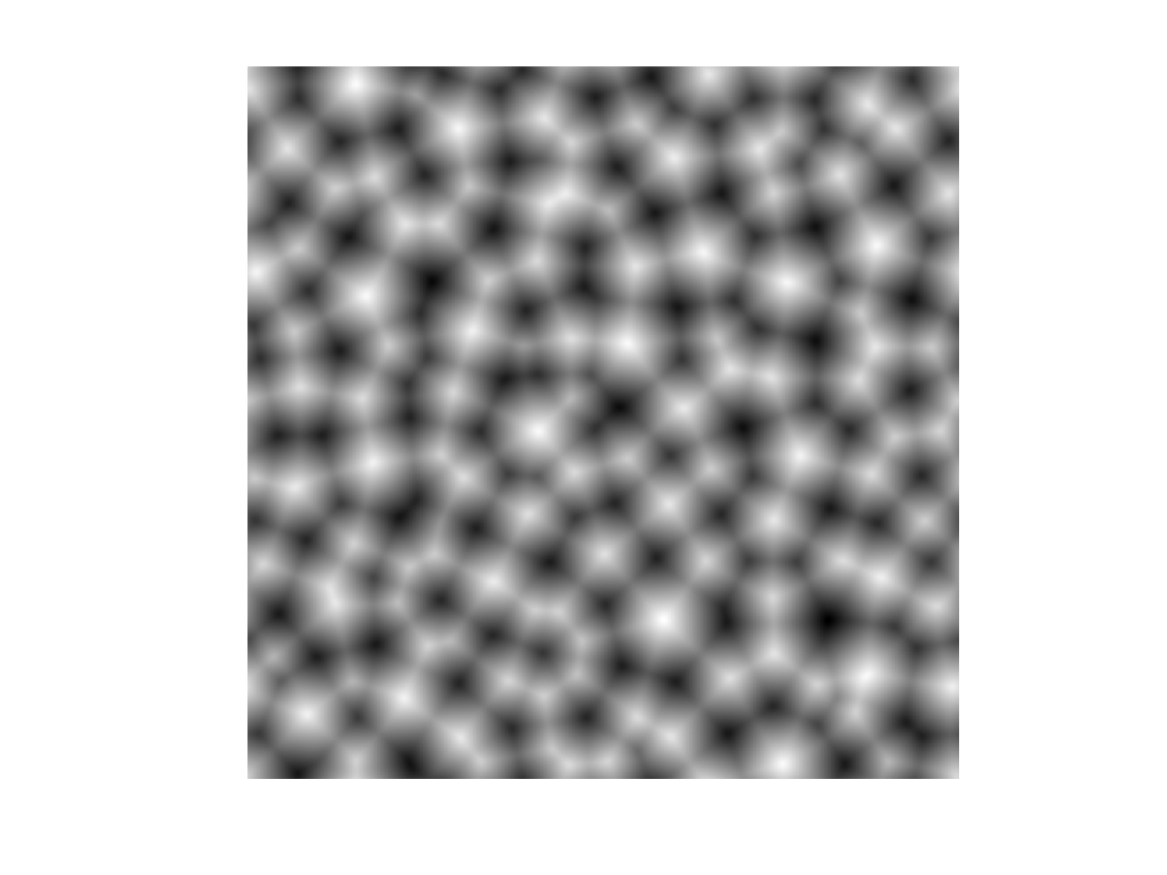}
			\end{minipage}
			\begin{minipage}{0.23\textwidth}
				\centering
				\includegraphics[width=4.5cm]{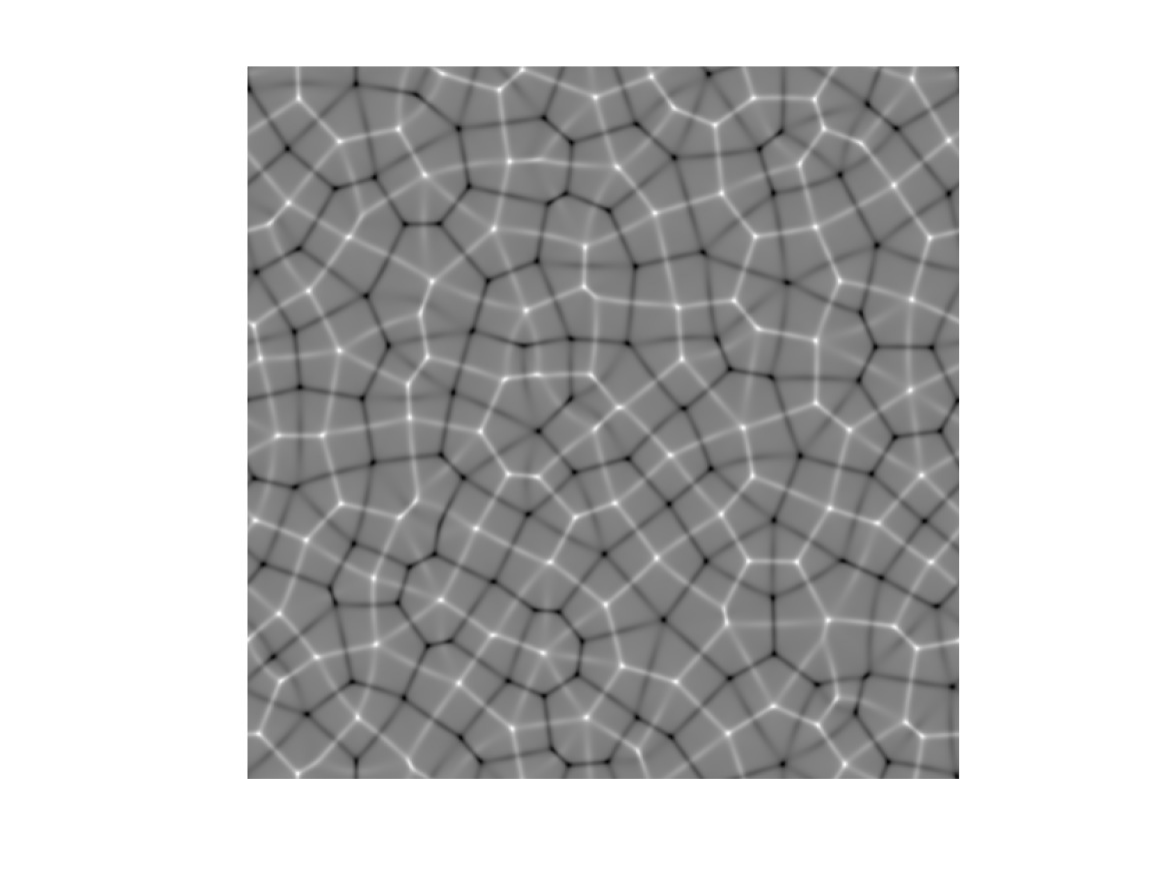}	
			\end{minipage}
		}
	\subfigure[$t = 50$]{
			\begin{minipage}{0.23\textwidth}
				\centering
				\includegraphics[width=4.5cm]{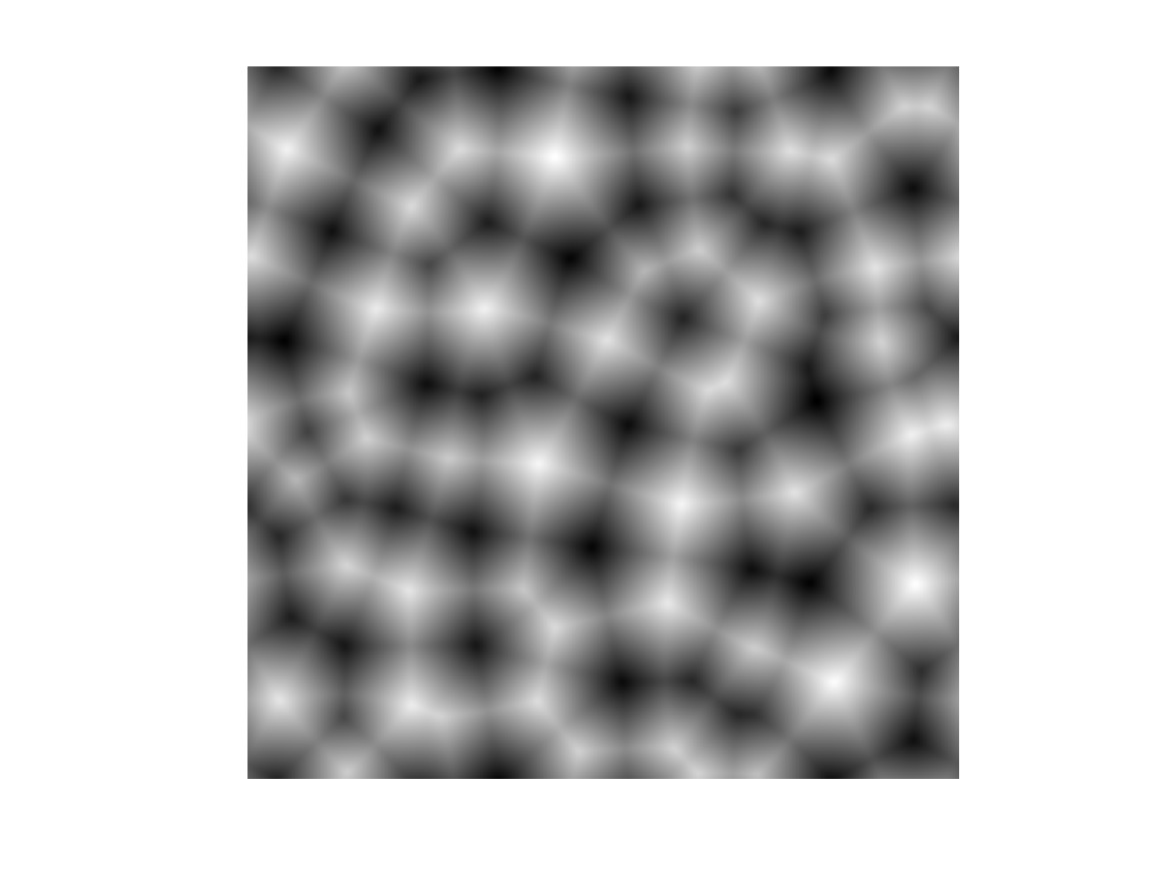}
			\end{minipage}
			\begin{minipage}{0.23\textwidth}
				\centering
				\includegraphics[width=4.5cm]{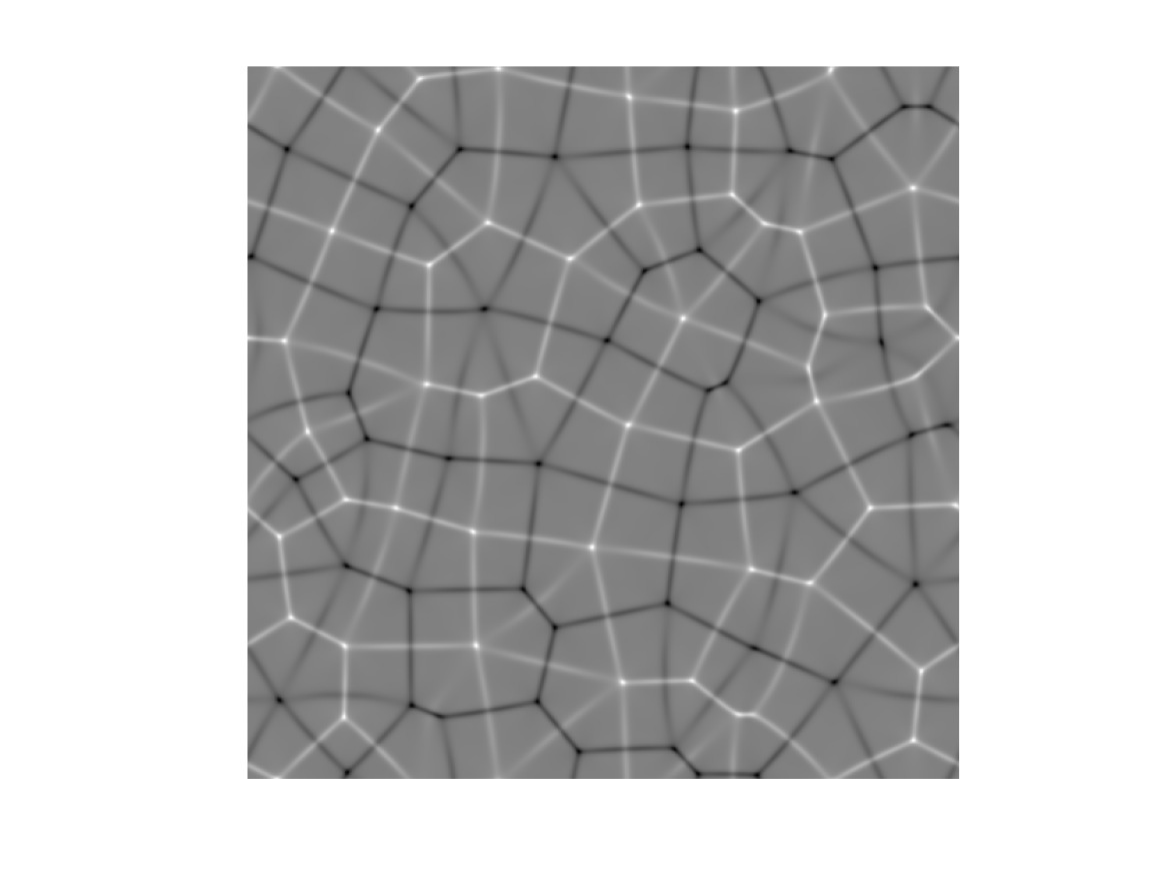}
			\end{minipage}
		}
		\subfigure[$t = 100$]{
			\begin{minipage}{0.23\textwidth}
				\centering
				\includegraphics[width=4.5cm]{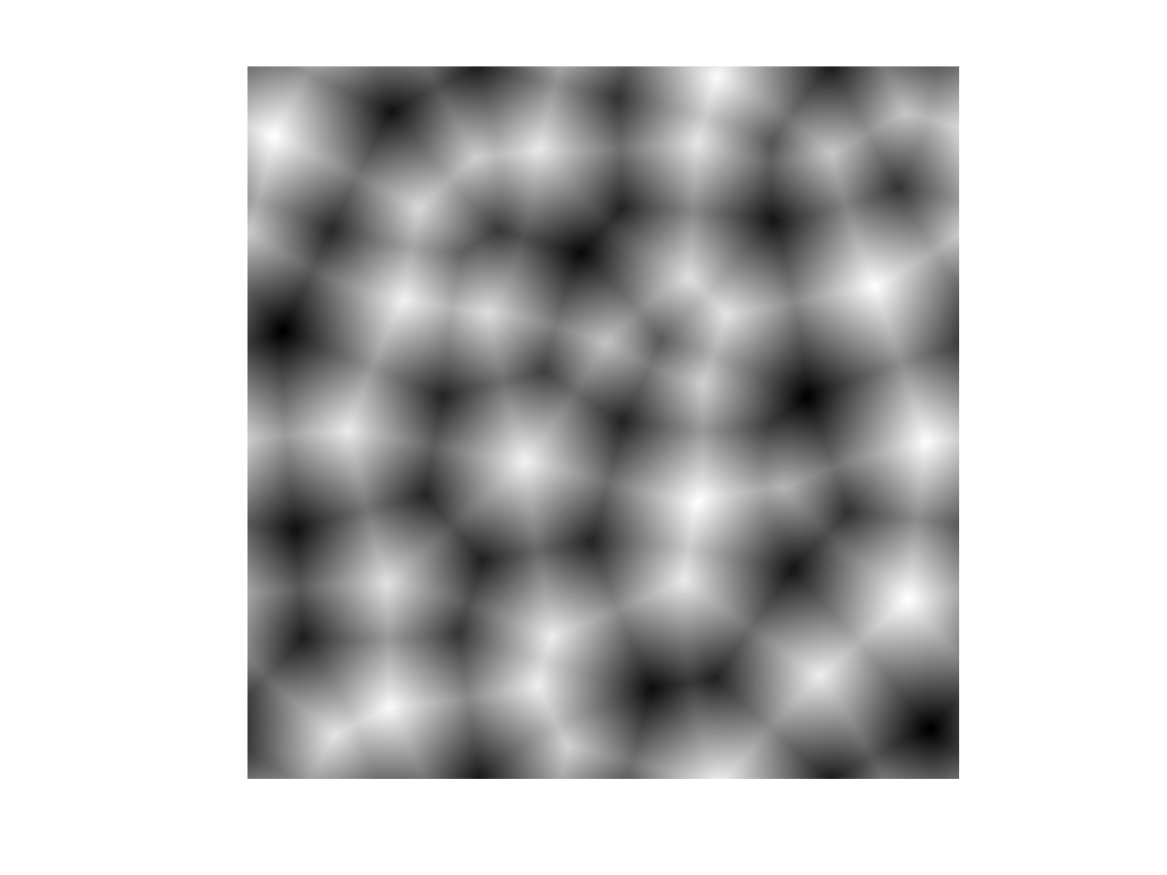}
			\end{minipage}
			\begin{minipage}{0.23\textwidth}
				\centering
				\includegraphics[width=4.5cm]{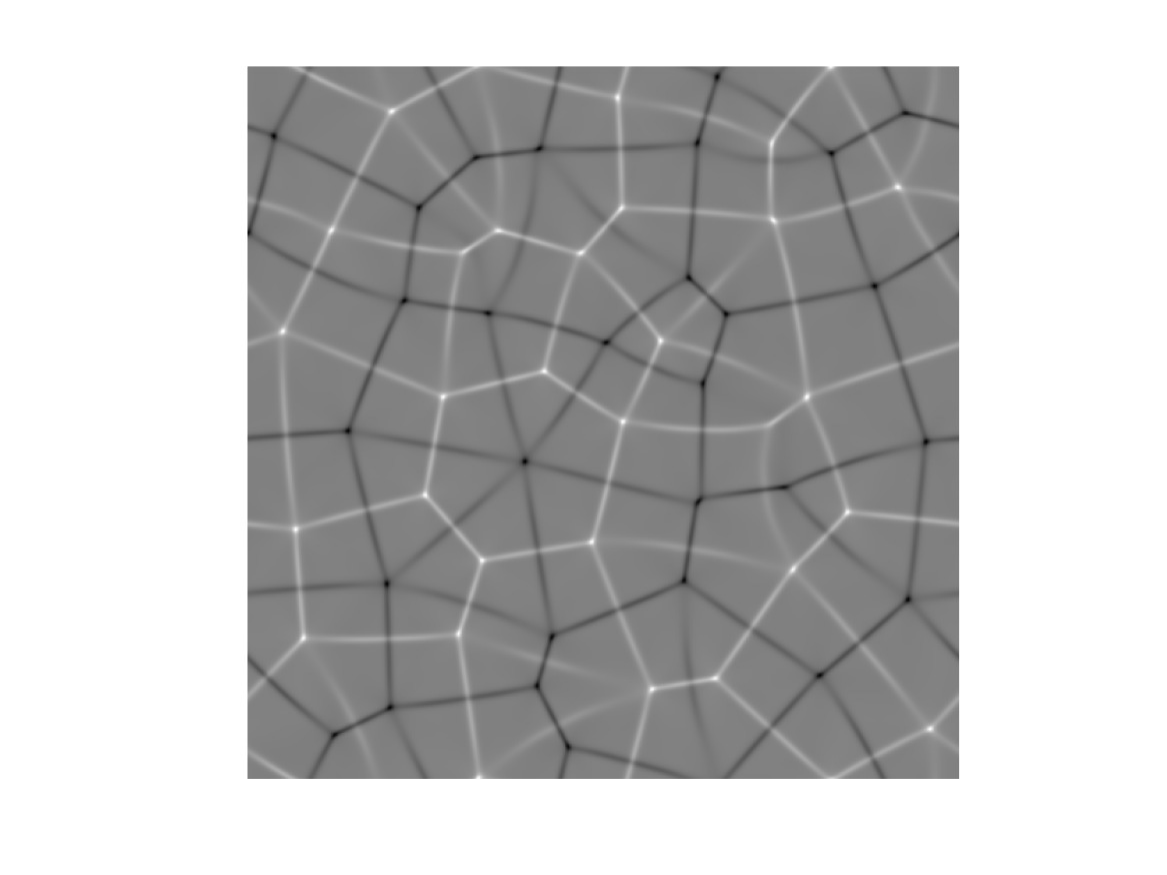}
			\end{minipage}
		}
	\subfigure[$t = 500$]{
			\begin{minipage}{0.23\textwidth}
				\centering
				\includegraphics[width=4.5cm]{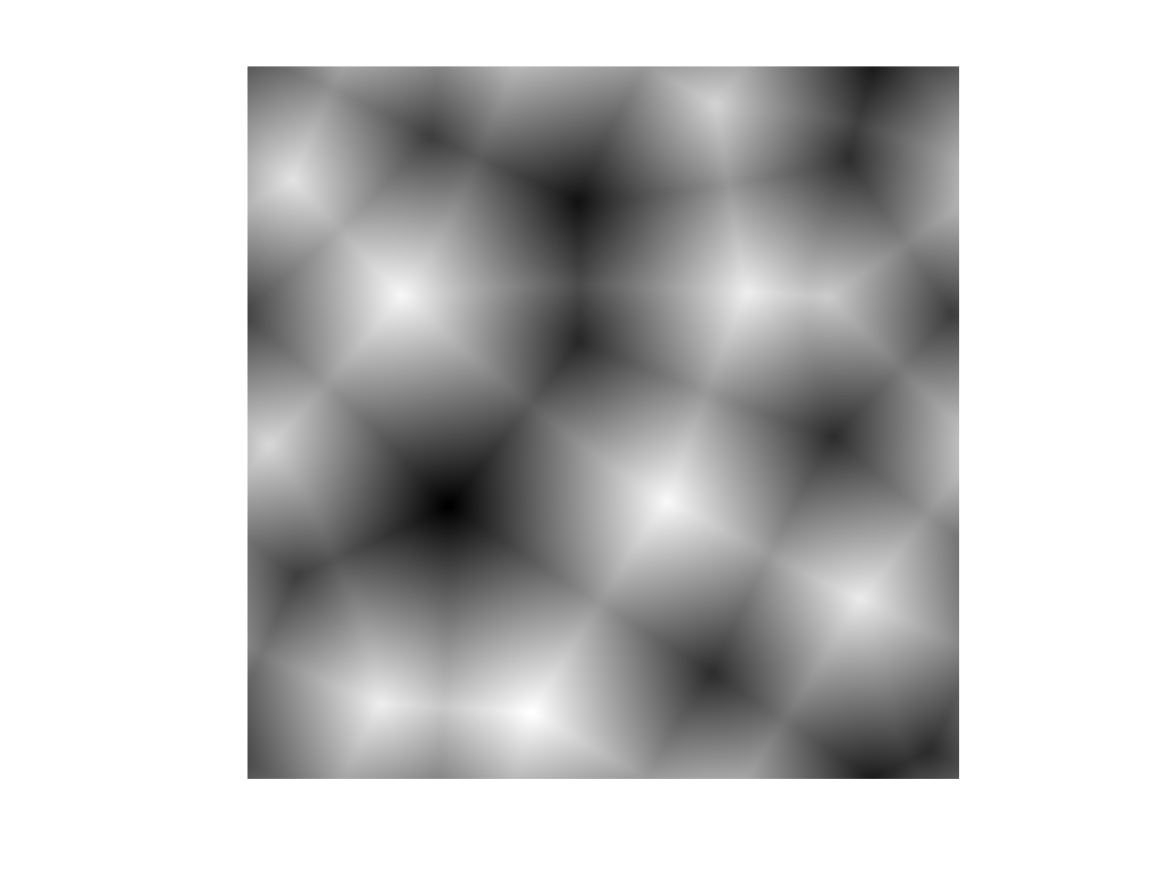}
			\end{minipage}
			\begin{minipage}{0.23\textwidth}
				\centering
				\includegraphics[width=4.5cm]{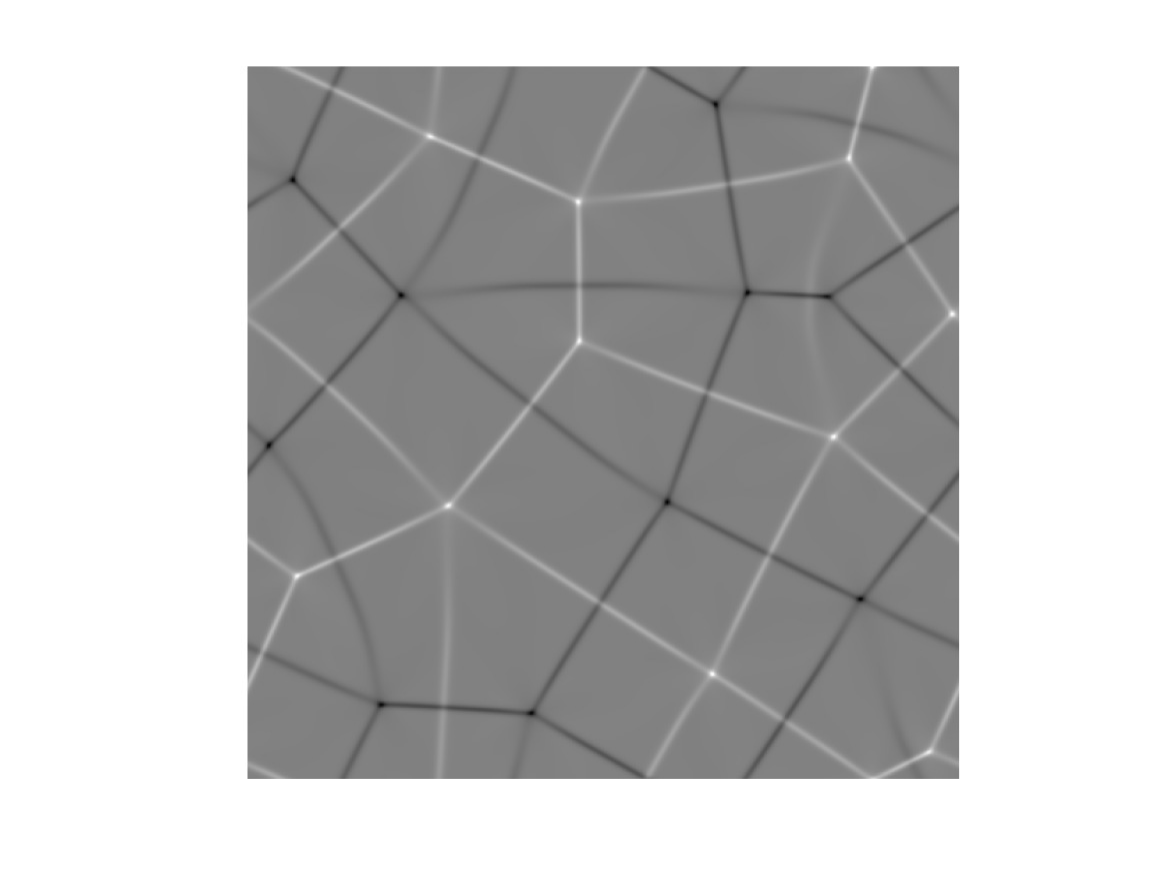}
			\end{minipage}
		}
			\caption{(Example \ref{ex:MBE}.) The isolines of the numerical solutions for the height function $\phi$ and its $\Delta \phi$ for the thin film epitaxy growth model with slope selection, using a random initial condition. In each subfigure, the left side represents $\phi$, while the right side represents $\Delta \phi$.}
			\label{Fig:MBE-with-slope-2D-PS-SAV}
		\end{figure}
	
		\begin{figure}[htbp]
			\centering
		\subfigure[$t = 0$]{
			\begin{minipage}{0.23\textwidth}
				\centering
				\includegraphics[width=4.5cm]{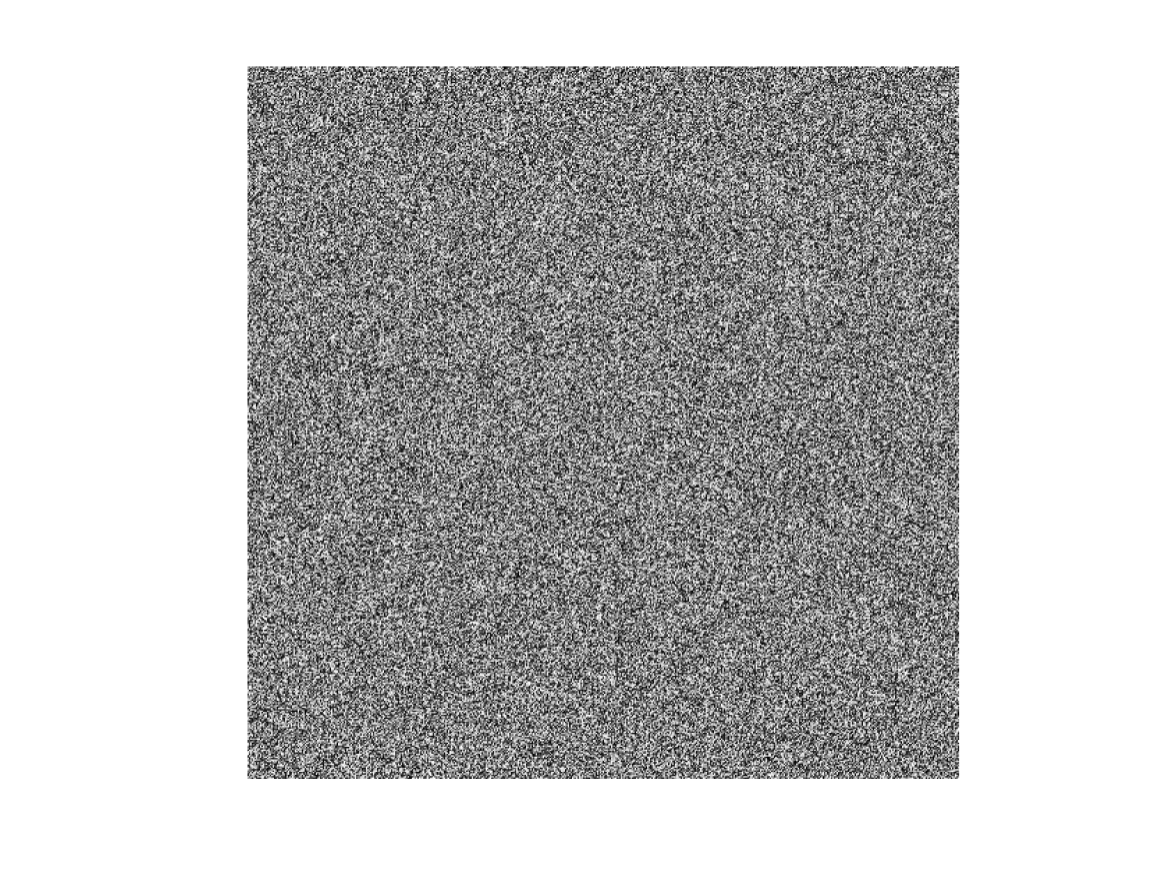}
			\end{minipage}
			\begin{minipage}{0.23\textwidth}
				\centering
				\includegraphics[width=4.5cm]{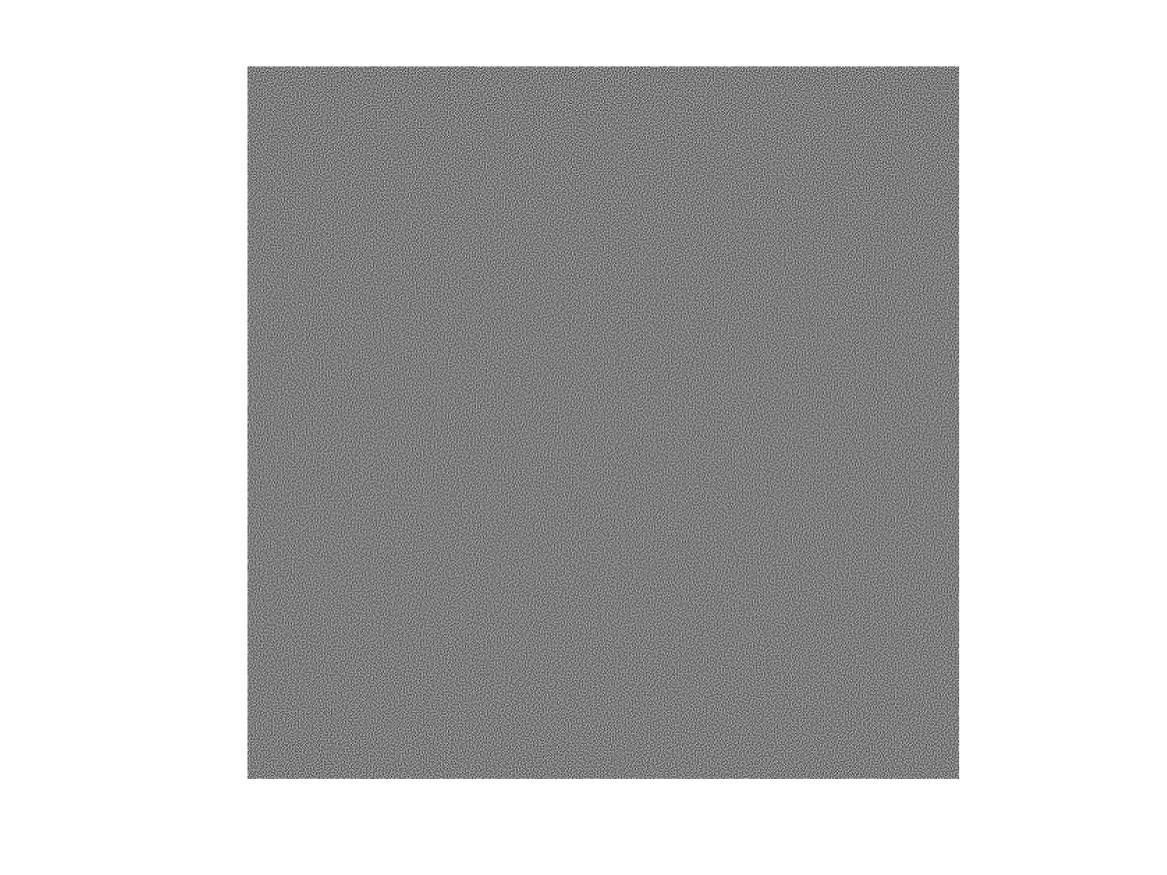}
			\end{minipage}
		}
	\subfigure[$t = 1$]{
			\begin{minipage}{0.23\textwidth}
				\centering
				\includegraphics[width=4.5cm]{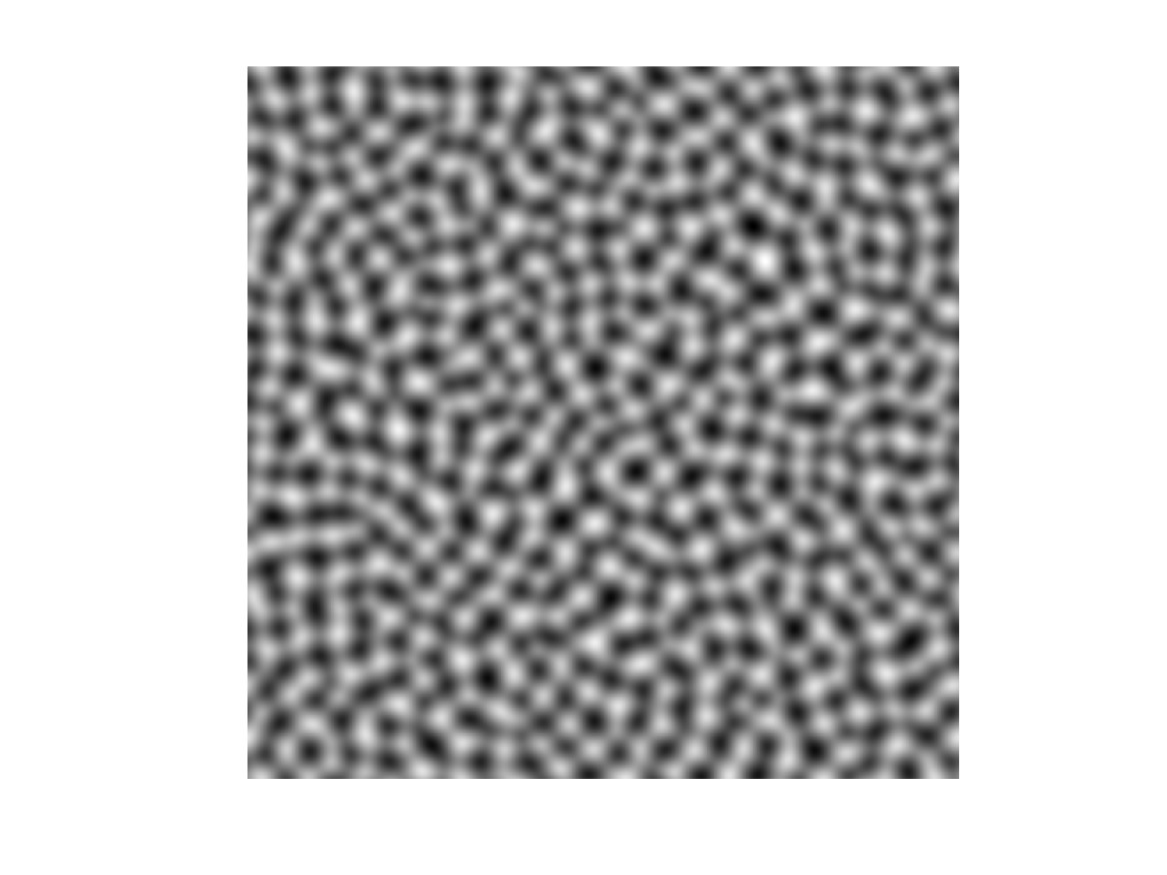}
			\end{minipage}
			\begin{minipage}{0.23\textwidth}
				\centering
				\includegraphics[width=4.5cm]{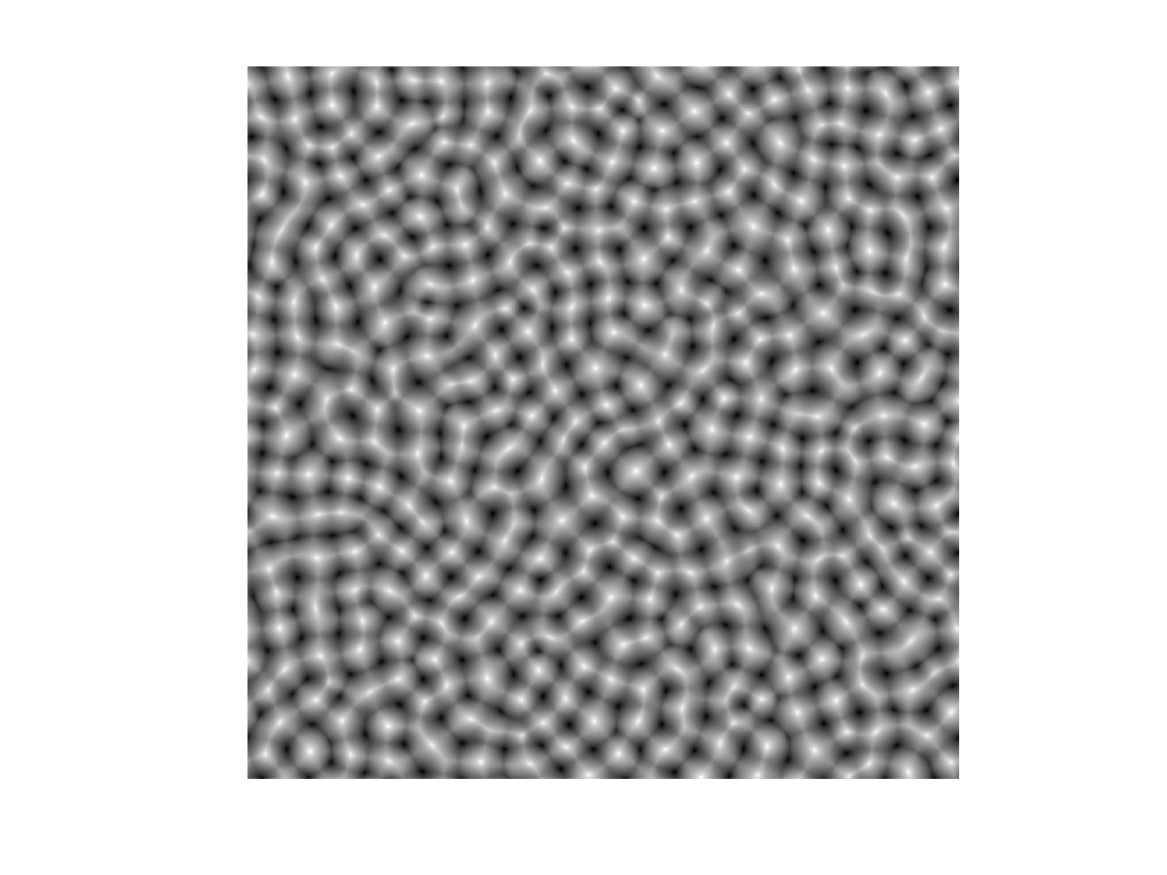}
			\end{minipage}
		}
	\subfigure[$t = 10$]{
			\begin{minipage}{0.23\textwidth}
				\centering
				\includegraphics[width=4.5cm]{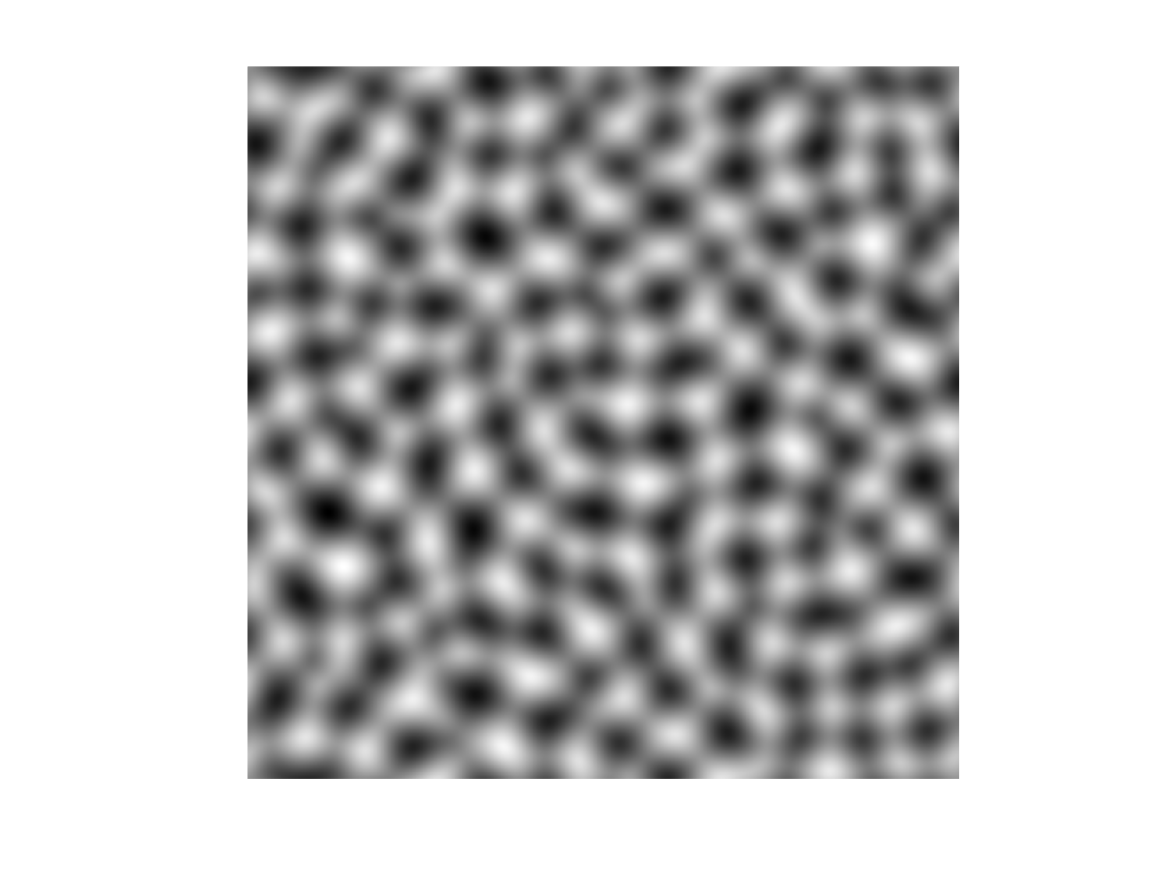}
			\end{minipage}
			\begin{minipage}{0.23\textwidth}
				\centering
				\includegraphics[width=4.5cm]{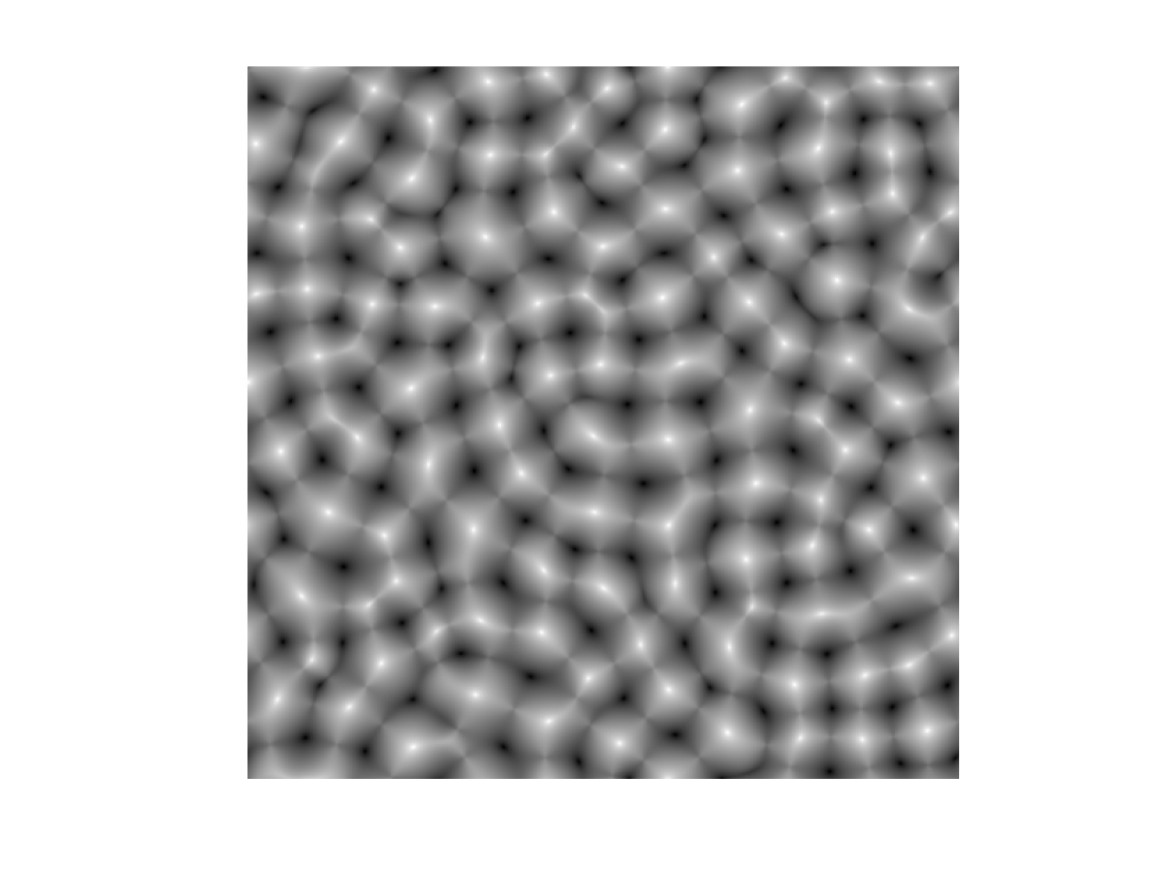}	
			\end{minipage}
		}
	\subfigure[$t = 50$]{
			\begin{minipage}{0.23\textwidth}
				\centering
				\includegraphics[width=4.5cm]{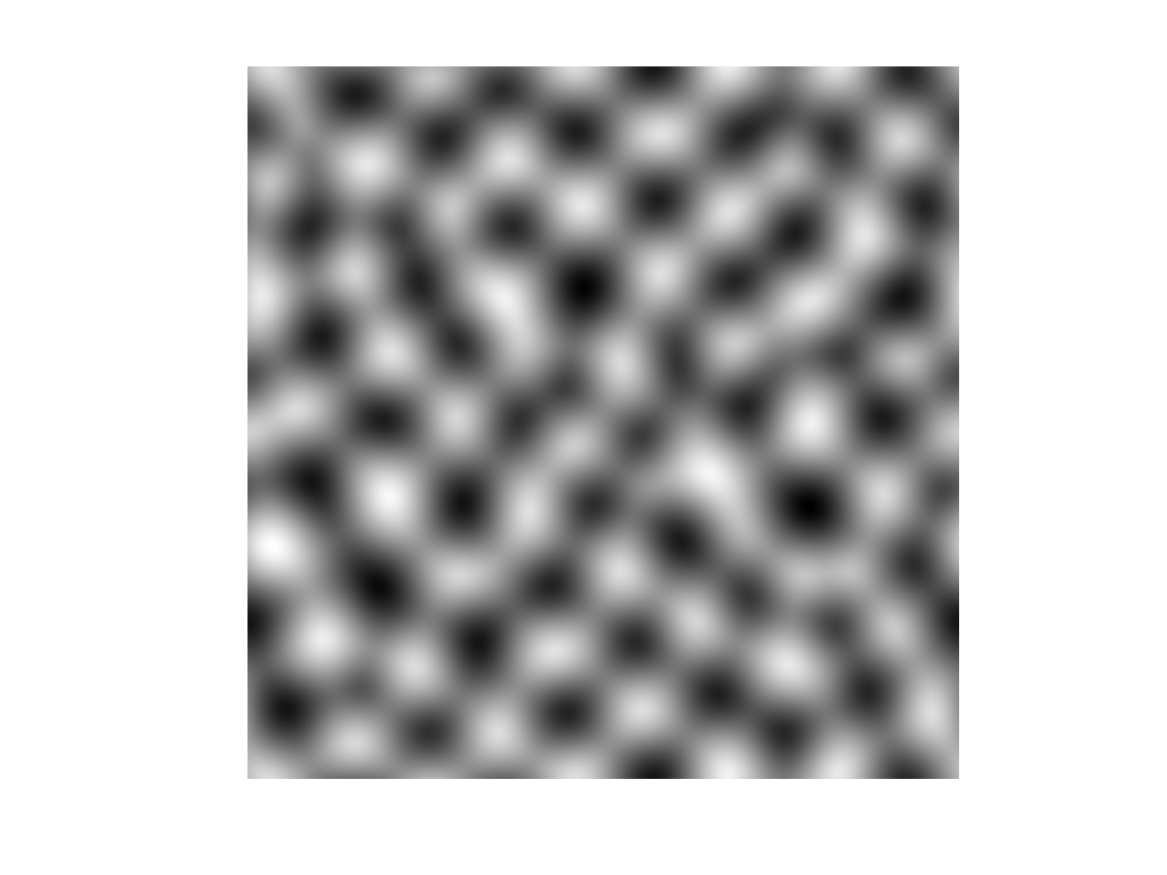}
			\end{minipage}
			\begin{minipage}{0.23\textwidth}
				\centering
				\includegraphics[width=4.5cm]{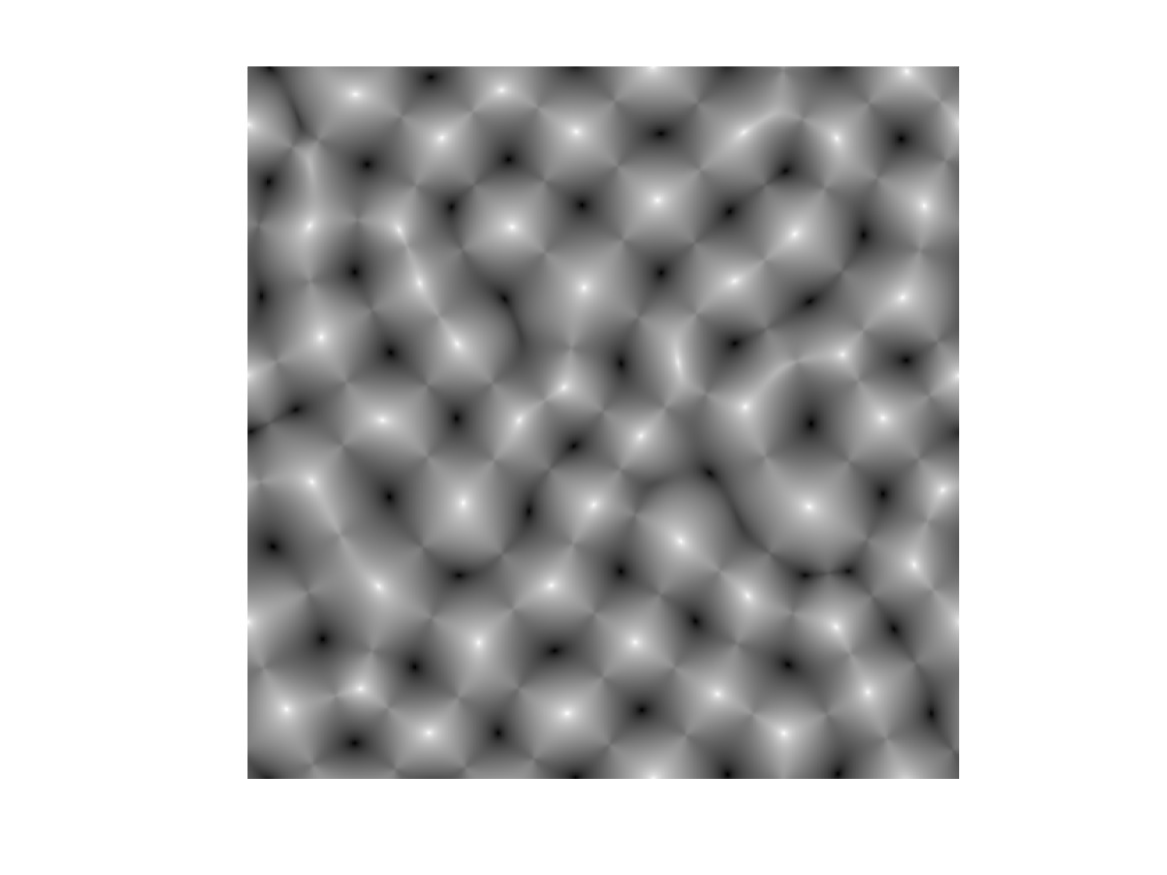}
			\end{minipage}
		}
	\subfigure[$t = 100$]{
			\begin{minipage}{0.23\textwidth}
				\centering
				\includegraphics[width=4.5cm]{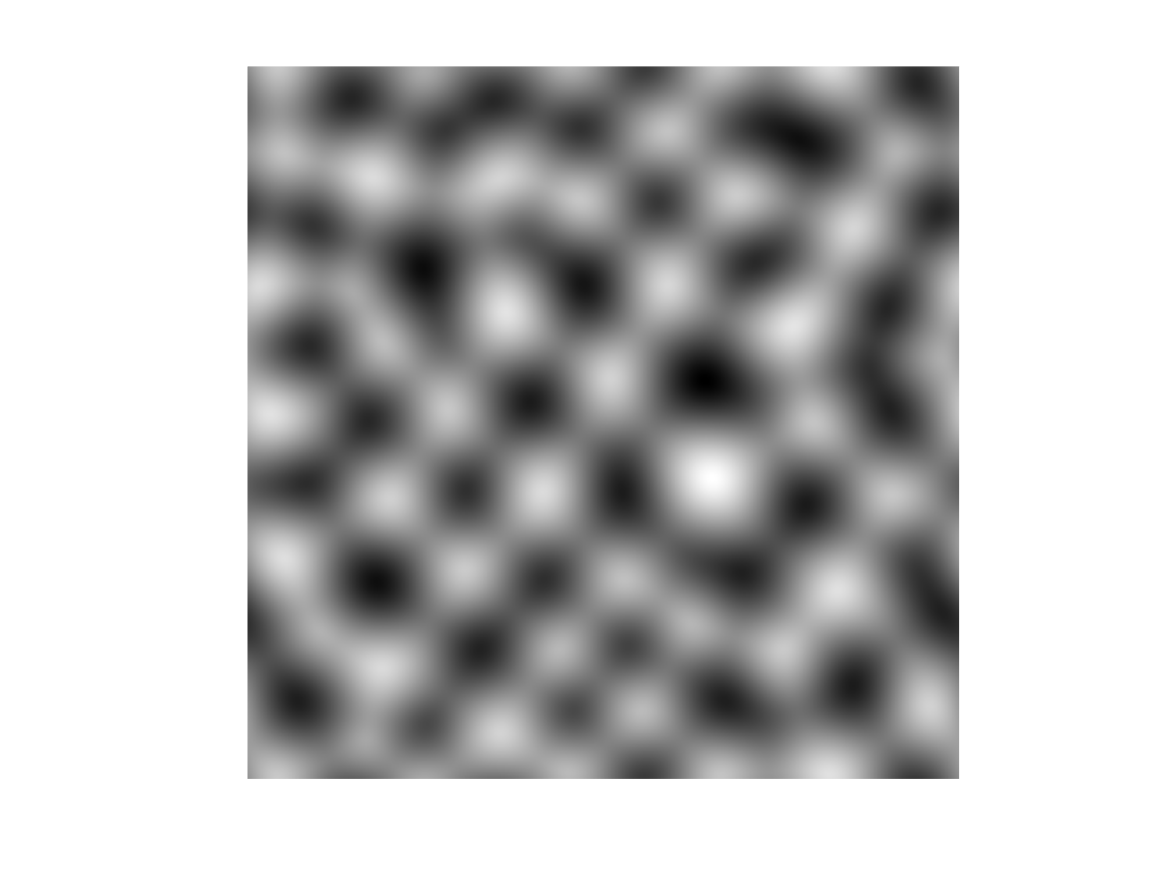}
			\end{minipage}
			\begin{minipage}{0.23\textwidth}
				\centering
				\includegraphics[width=4.5cm]{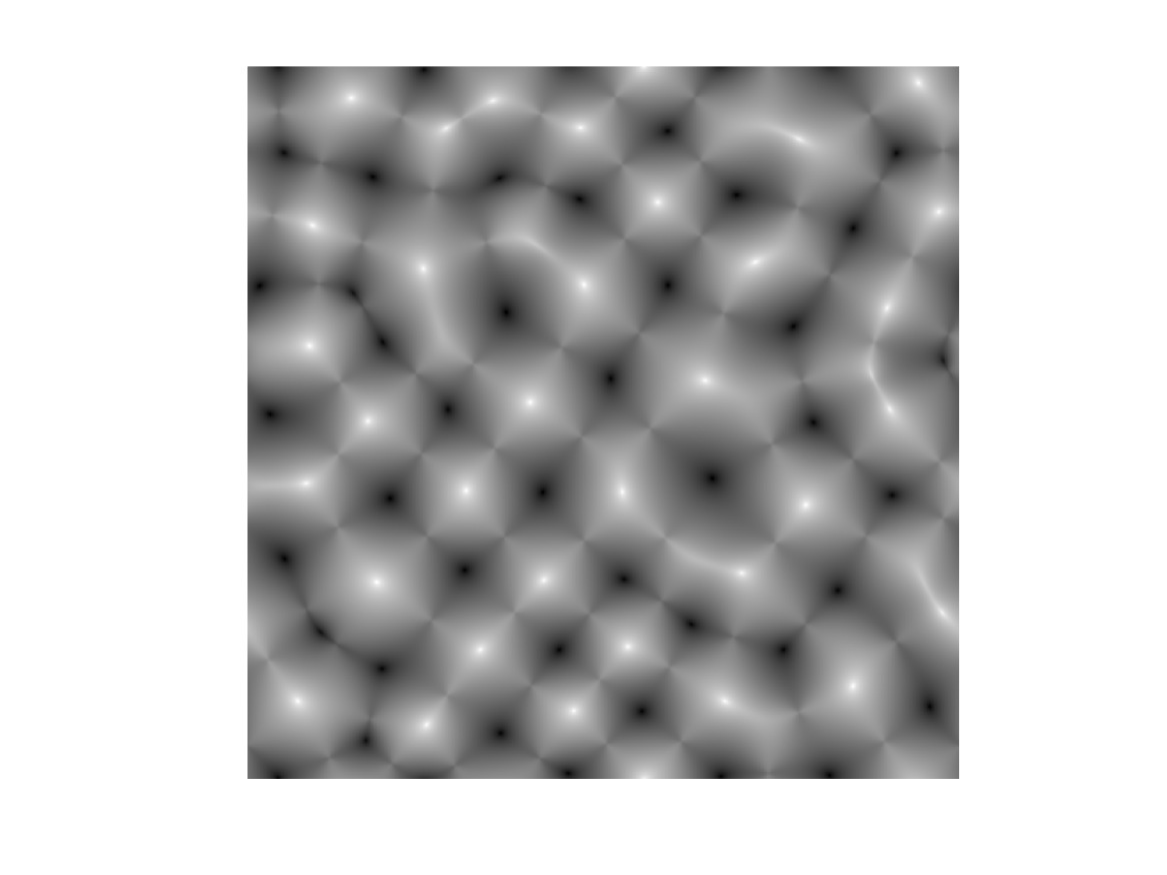}
			\end{minipage}
		}
	\subfigure[$t = 500$]{
			\begin{minipage}{0.23\textwidth}
				\centering
				\includegraphics[width=4.5cm]{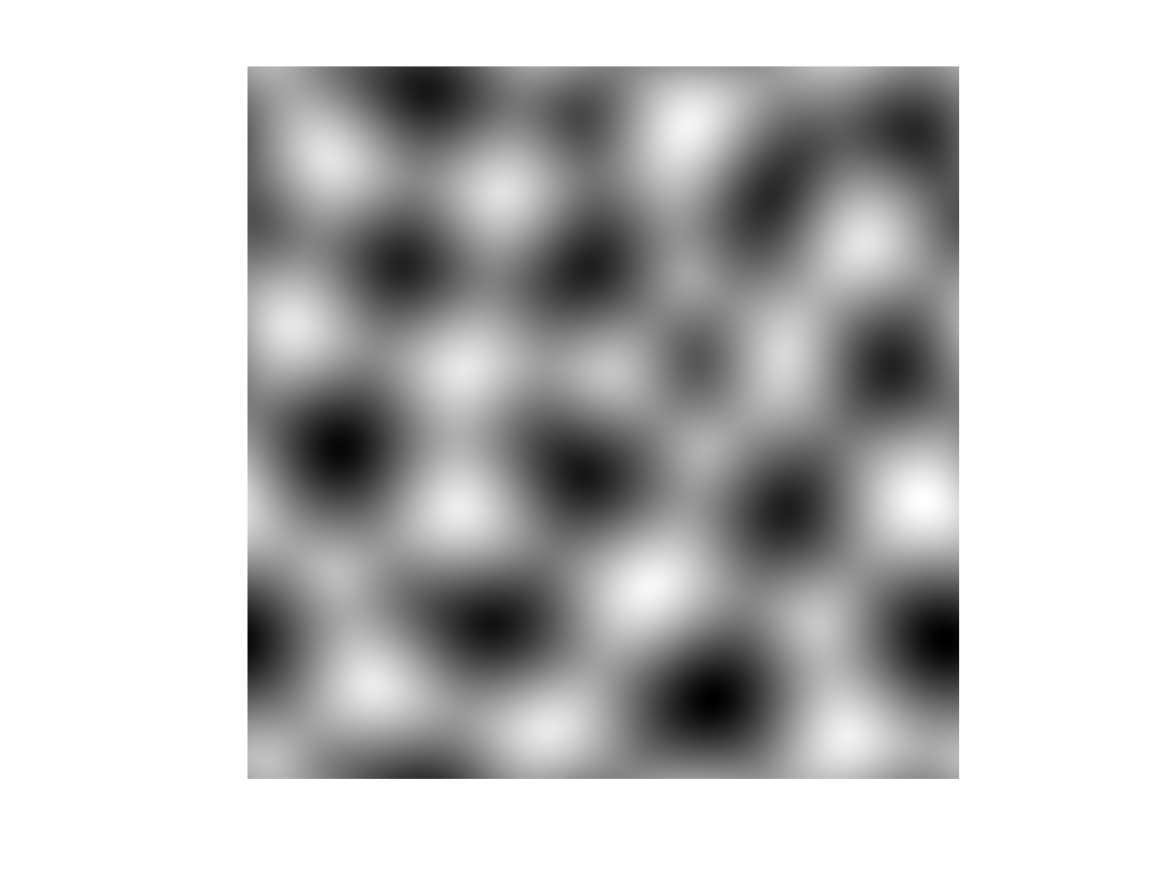}
			\end{minipage}
			\begin{minipage}{0.23\textwidth}
				\centering
				\includegraphics[width=4.5cm]{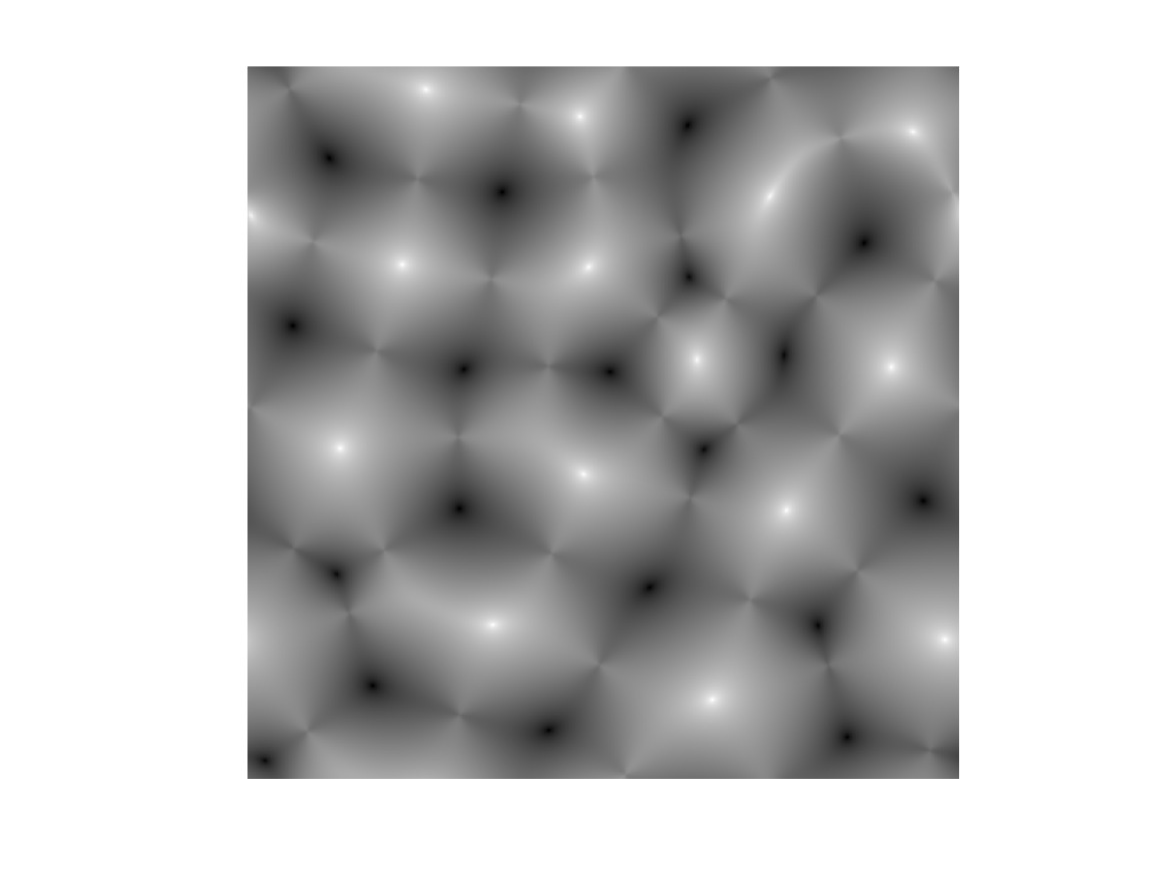}
			\end{minipage}
		}
			\caption{(Example \ref{ex:MBE}.) The isolines of the numerical solutions for the height function $\phi$ and its $\Delta \phi$ for the thin film epitaxy growth model without slope selection, using a random initial condition. In each subfigure, the left side represents $\phi$, while the right side represents $\Delta \phi$.}
			\label{Fig:MBE-without-slope-2D-PS-SAV}
		\end{figure}
	
		\begin{figure}[htbp]
		\centering
		\begin{minipage}{0.4\textwidth}
			\centering
			\includegraphics[width=5.3cm]{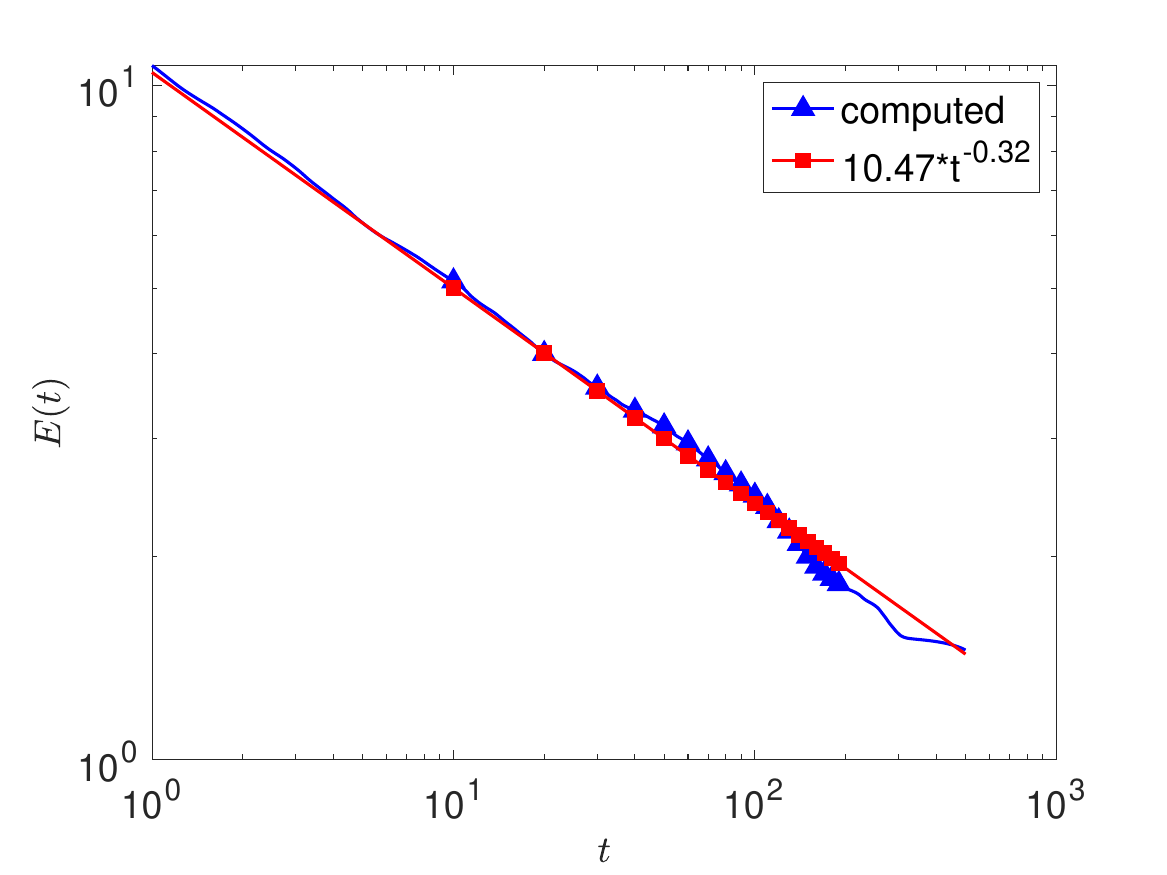}
		\end{minipage}
		\begin{minipage}{0.4\textwidth}
			\centering
			\includegraphics[width=5.3cm]{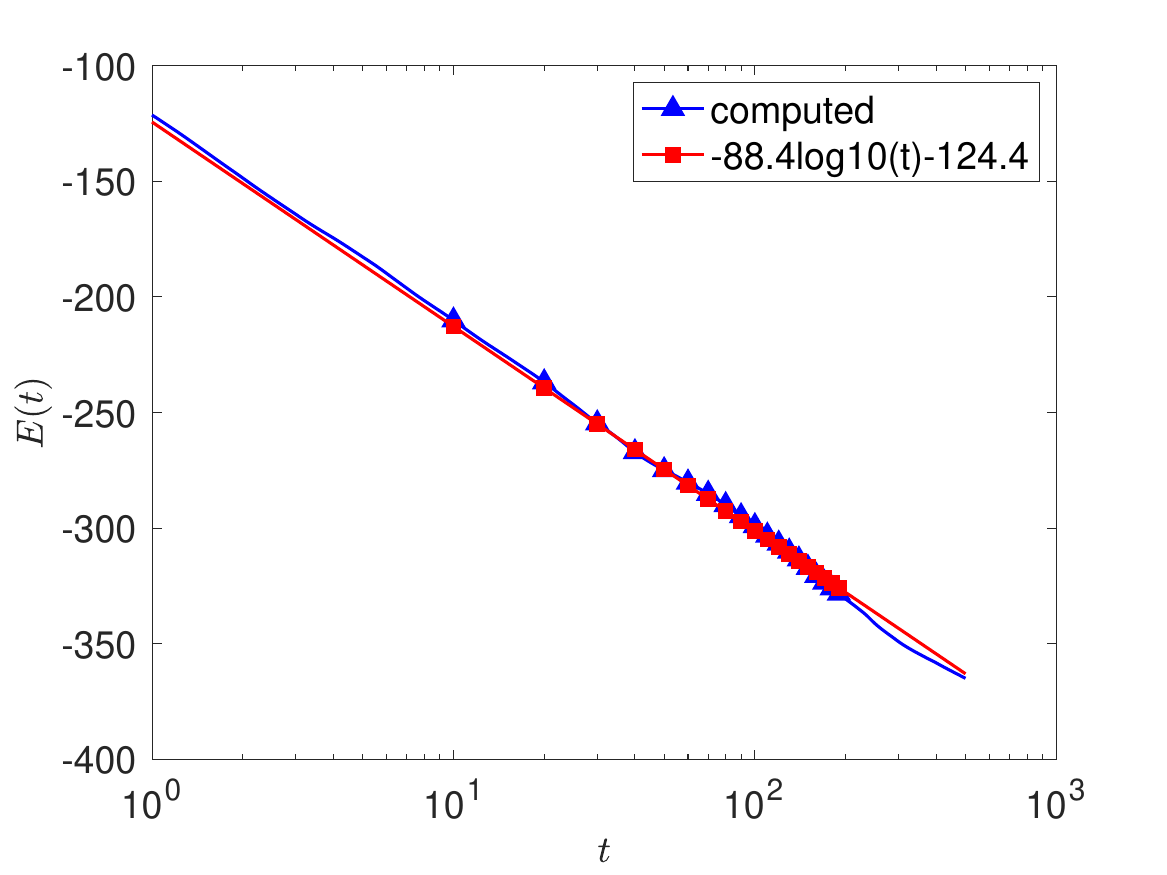}
		\end{minipage}
		\caption{(Example \ref{ex:MBE}.) First: the log-log plots of the free energy for the thin film epitaxy growth model with slope selection. Second: the semi-log plots of the free energy for the thin film epitaxy growth model without slope selection.}
		\label{Fig:MBE-energy}
	\end{figure}

\end{example}

\section*{Acknowledgement}
No potential conflict of interest was reported by the author. We would like to acknowledge the assistance of volunteers in putting together this example manuscript and supplement.
\bibliographystyle{siamplain}
\bibliography{Reference}

\end{document}